\documentclass[12pt,a4paper]{amsart}
\usepackage{amsmath, amsthm, amscd, amsfonts,graphicx}
\usepackage[matrix,arrow]{xy}
 \newtheorem{proposition}{Proposition}[section]
\newtheorem{theorem}{Theorem}[section]
\newtheorem{lemma}{Lemma}[section]
\newtheorem{corollary}{Corollary}[section]

\newtheorem{example}{Example}[section]

\theoremstyle{definition}
\newtheorem{definition}{Definition}[section]

\newtheorem{remark}{Remark}[section]

\begin{document}

\title[Monotone-light factorization for $n$-categories]{The monotone-light factorization for $n$-categories via $n$-preorders}

 \author[Jo\~{a}o J. Xarez]{Jo\~{a}o J. Xarez}



 \address{CIDMA - Center for Research and Development in Mathematics and Applications,
Department of Mathematics, University of Aveiro, 3810-193 Aveiro, Portugal}

 \email{xarez@ua.pt}




\begin{abstract}
Starting with a symmetric monoidal adjunction with certain properties, we derive another symmetric monoidal adjunction with the same properties between the respective categories of all $\mathcal{V}$-categories. If we begin with a reflection of a full replete subcategory, the derived adjunction is also a reflection of the same kind. Semi-left-exactness (also called admissibility in categorical Galois theory) or the stronger stable units property is inherited by the derived reflection. Applying these results, we conclude that the reflection of the category of all $n$-categories into the category of $n$-preorders has stable units. Then, it is also shown that this reflection determines a monotone-light factorization system on $n$-categories, $n\geq 1$, and that the light morphisms are precisely the $n$-functors faithful with respect to $n$-cells. In order to achieve such results, it was also shown that $n$-functors surjective both on vertically composable triples of horizontally composable pairs of $n$-cells, and on horizontally composable triples of vertically composable pairs of $n$-cells, are effective descent morphisms in the category of all $n$-categories $nCat$, $n\geq 1$.
\end{abstract}

\maketitle

\section{Introduction}\label{sec-introduction}

It will be shown that the reflection $nCat\rightarrow nPreord$ of the category of all $n$-categories into the category of $n$-preorders determines a monotone-light factorization system on $nCat$, for every positive integer $n$. In order to achieve such result it was also proved that the reflection $nCat\rightarrow nPreord$ has stable units, a stronger condition than admissibility in categorical Galois theory, and that the $n$-functors surjective both on
vertically composable triples of horizontally composable pairs of $n$-cells, and on horizontally composable triples of vertically composable pairs of $n$-cells (cf.\ Proposition \ref{proposition:EDM(nCat)}), are effective descent morphisms in $nCat$. This generalizes the results in our previous papers \cite{X:ml} and \cite{X:2ml}.
\\

The characterization of effective descent morphisms is a generalization of what was done in \cite{X:2ml}, considering  a convenient description of the category of all $n$-categories as a full subcategory of a presheaf category, and using some results in \cite{JST:edm}.
\\

In order to prove that the reflection of $n$-categories into $n$-preorders has stable units in the sense of \cite{CHK:fact} (see \cite{CJKP:stab}), we took a very different path from the one trailed in \cite{X:2ml}. In fact, we considered the category of all $n$-categories in the context of $\mathcal{V}$-categories (see \cite{Kelly:enriched_cat}). Remark that, considering $\mathcal{V} = Set$ the category of sets and then iterating we obtain $n$-categories.

Beginning with a symmetric monoidal adjunction $\mathbb{C}\rightarrow\mathbb{X}$ with certain properties, we derive another symmetric monoidal adjunction $\mathbb{C}$-$Cat\rightarrow\mathbb{X}$-$Cat$ with the same properties between categories of all $\mathcal{V}$-categories ($\mathcal{V}\in\{\mathbb{C},\ \mathbb{X}\}$). In this context, if we begin with a reflection such that $\mathbb{X}$ is a full replete subcategory of $\mathbb{C}$, the derived adjunction is also a reflection in which $\mathbb{X}$-$Cat$ is a full replete subcategory of $\mathbb{C}$-$Cat$. It follows trivially, according to the way limits are calculated in $\mathbb{C}$-$Cat$, provided they exist in $\mathbb{C}$, that semi-left-exactness (also called admissibility in categorical Galois theory) or the stable units property is inherited by the derived reflection. For instance, if we begin with the stable units reflection from the category of all categories into the category of all preorders, then iterating we end with a stable units reflection from $nCat$ into $nPreord$.\\

The paper is organized as follows.


Consider \emph{a base monoidal adjunction}, that is, an adjunction $(A)$ $$(F,G,\eta,\varepsilon):(\mathbb{C},\otimes,E,\alpha,\gamma,\rho)\rightarrow (\mathbb{X},\lozenge,I,\mathfrak{a},t,r)$$ such that:
\begin{itemize}
\item[$(B)$] $F$ and $G$ are (strict) morphisms of (symmetric) monoidal categories,

\item[$(C)$] for every pair of objects $X,Y\in\mathbb{X}$, the counit of the tensor product is equal to the tensor product of the counits, $\varepsilon_X\lozenge\varepsilon_Y=\varepsilon_{X\lozenge Y}:FG(X\lozenge Y)\rightarrow X\lozenge Y$, and

\item[$(D)$] the counit morphism of $I$ is equal to the identity morphism of $I$, $\varepsilon_I=1_I:FG(I)=I\rightarrow I$.
\end{itemize}

Starting from any base monoidal adjunction as defined (cf.\ section \ref{sec-base monoidal adjunction}), we will show how to obtain another derived base monoidal adjunction (cf.\ section \ref{section:derivedMonoidal}) between the $\mathcal{V}$-categories $\mathbb{C}$-$Cat$ and $\mathbb{X}$-$Cat$, the category of all $\mathbb{C}$-categories and the category of all $\mathbb{X}$-categories, respectively,

\begin{center}
$(\mathbb{F},\mathbb{G},\Theta,\Upsilon):
(\mathbb{C}$-$Cat,\bigcirc,\mathfrak{E},\wedge,\Gamma,\mho)\rightarrow(\mathbb{X}$-$Cat,\nabla,\mathcal{I},\vee,\top,\Re)$.
\end{center}

This is given and proved in detail from sections \ref{sec-base monoidal adjunction} to \ref{section:derived adjunction monoidal}.\\

In section \ref{section:limitsV-Cat}, it is shown that limits in $\mathbb{C}$-$Cat$ may be calculated ``hom-componentwise" in $\mathbb{C}$. It follows that the derived monoidal structure is cartesian if it is so for the starting monoidal structure (cf.\ Remark \ref{remark:cartesian monoidal}).\\

In section \ref{section:derivedReflection}, in order to apply the above results to categorical Galois theory (cf.\ \cite{G. Janelidze}), when $\mathbb{X}$ is a full replete subcategory of $\mathbb{C}$ and $G:\mathbb{X}\subseteq\mathbb{C}$ is the inclusion functor, we show that it is a special case of the base monoidal adjunction, which we call \emph{base monoidal reflection}, and that using the same process of deriving a new base monoidal adjunction we get a new base monoidal reflection.

Furthermore, if the initial base monoidal reflection is semi-left-exact, or has the stronger property of having stable units (notions introduced in \cite{CHK:fact}), the same is true for the derived monoidal reflection, as shown in section \ref{sec:Simp., semi-left-exact. and stable units}.

Hence, the process can be iterated any number of times.\\

The second part of the paper consists in applying the previous results to the base monoidal reflection $H\vdash I:Cat\rightarrow Preord$, which has stable units, to conclude that there is a reflection with stable units $nCat\rightarrow nPreord$, from the category of all $n$-categories (with cartesian monoidal structure), $n\in\mathbb{N}$. We also conclude that there is a non trivial monotone-light factorization associated to every reflection $nCat\rightarrow nPreord$, $n\in\mathbb{N}$, in a similar way to that used in \cite{X:2ml} to obtain the same results for the special case $n=2$ (cf.\ sections \ref{sec:nCat}, \ref{sec:edm in nCat} and \ref{sec:stableunits&ml nCat}).

It is also given a needed characterization of a class of effective descent morphisms in $nCat$ (cf.\ section \ref{sec:edm in nCat}); and characterizations of the classes of vertical and stably-vertical $n$-functors (which by being distinct classes imply the non-triviality of the monotone-light factorizations), trivial coverings and coverings (cf.\ section \ref{sec:Vertical and stably-vertical}, and sections \ref{sec:Trivial coverings} and \ref{sec:Coverings}, respectively), for every $n\in \mathbb{N}$.


\section{The base monoidal adjunction}\label{sec-base monoidal adjunction}

Consider an adjunction $(A)$
$$(F,G,\eta,\varepsilon):\mathbb{C}\rightarrow\mathbb{X},$$
such that

\noindent $(B)$ both the left-adjoint $F$ and the right-adjoint $G$ are (strict) \emph{morphisms of symmetric monoidal categories}, with respect to the \emph{monoidal categories}
$$(\mathbb{C},\otimes,E,\alpha,\gamma,\rho)\ and\ (\mathbb{X},\lozenge,I,\mathfrak{a},t,r),$$
meaning that:

$E\in\mathbb{C}$, $I\in\mathbb{X}$, $F(E)=I$ and $G(I)=E$;

$\otimes:\mathbb{C}\times\mathbb{C}\rightarrow\mathbb{C}$ and $\lozenge:\mathbb{X}\times\mathbb{X}\rightarrow\mathbb{X}$ are (bi)functors;

$F$ \emph{preserves} $\otimes$ and $G$ \emph{preserves} $\lozenge$, i.e.,
$$F(A\otimes B^{\underrightarrow{f\otimes g}}C\otimes D)=F(A)\lozenge F(B)^{\underrightarrow{Ff\lozenge Fg}}F(C)\lozenge F(D),$$
$$G(X\lozenge Y^{\underrightarrow{u\lozenge v}}Z\lozenge W)=G(X)\otimes G(Y)^{\underrightarrow{Gu\otimes Gv}}G(Z)\otimes G(W),$$
for all morphisms $f:A\rightarrow C$ and $g:B\rightarrow D$ in $\mathbb{C}$, $u:X\rightarrow Z$ and $v:Y\rightarrow W$ in $\mathbb{X}$;
$$\alpha_{A,B,C}:(A\otimes B)\otimes C\rightarrow A\otimes (B\otimes C),$$
$$\mathfrak{a}_{X,Y,Z}:(X\lozenge Y)\lozenge Z\rightarrow X\lozenge (Y\lozenge Z),$$
$$\gamma_{A,B}:A\otimes B\rightarrow B\otimes A,\ \rho_A :A\otimes E\rightarrow A,$$
$$t_{X,Y}:X\lozenge Y\rightarrow Y\lozenge X\ and\ r_X :X\lozenge I\rightarrow X$$
are natural isomorphisms subject to the \emph{coherence axioms}\footnote{Cf.\ \cite[\S VII.2]{SM:cat} for more information about the consequences of these axioms.} expressing the commutativity of the following diagrams, for all $A,B,C,D\in\mathbb{C}$ and all $X,Y,Z,W\in\mathbb{X}$),

\begin{picture}(370,55)\setlength{\unitlength}{0.285mm}

\put(-20,40){$((D\otimes A)\otimes B)\otimes C$}\put(105,43){\vector(1,0){37}}\put(95,55){$\alpha_{D\otimes A,B,C}$}
\put(145,40){$(D\otimes A)\otimes (B\otimes C)$}\put(270,43){\vector(1,0){37}}\put(260,55){$\alpha_{D,A,B\otimes C}$}
\put(310,40){$D\otimes (A\otimes (B\otimes C))$}

\put(20,33){\vector(0,-1){20}}\put(23,20){$\alpha_{D,A,B}\otimes 1_C$}
\put(290,20){$1_D\otimes\alpha_{A,B,C}$}\put(370,13){\vector(0,1){20}}

\put(-20,0){$(D\otimes (A\otimes B))\otimes C$}\put(105,03){\vector(1,0){203}}\put(150,10){$\alpha_{D, A\otimes B,C}$}\put(310,0){$D\otimes ((A\otimes B)\otimes C),$}

\end{picture}

\begin{picture}(370,65)

\put(-20,40){$(A\otimes E)\otimes B$}\put(52,43){\vector(1,0){60}}\put(67,47){$\alpha_{A,E,B}$}
\put(115,40){$A\otimes (E\otimes B)$}

\put(230,40){$A\otimes B$}\put(267,43){\vector(1,0){40}}\put(272,47){$\gamma_{A,B}$}\put(310,40){$B\otimes A$}

\put(40,33){\vector(1,-1){20}}\put(0,20){$\rho_A\otimes 1_B$}\put(65,0){$A\otimes B$\ \ ,}\put(130,33){\vector(-1,-1){20}}\put(130,20){$1_A\otimes(\rho_B\circ\gamma_{E,B})$}

\put(240,20){$1_{A\otimes B}$}\put(253,33){\vector(2,-1){50}}
\put(327,33){\vector(0,-1){20}}\put(332,20){$\gamma_{B,A}$}\put(310,0){$A\otimes B$}

\end{picture}

and

\begin{picture}(370,60)\setlength{\unitlength}{0.315mm}

\put(-20,40){$(A\otimes B)\otimes C$}\put(62,43){\vector(1,0){80}}\put(87,47){$\alpha_{A,B,C}$}
\put(145,40){$A\otimes (B\otimes C)$}\put(227,43){\vector(1,0){80}}\put(247,47){$\gamma_{A,B\otimes C}$}
\put(310,40){$(B\otimes C)\otimes A$}

\put(20,33){\vector(0,-1){20}}\put(23,20){$\gamma_{A,B}\otimes 1_C$}
\put(310,20){$\alpha_{B,C,A}$}\put(350,33){\vector(0,-1){20}}\put(365,20){in $\mathbb{C}$,}

\put(-20,0){$(B\otimes A)\otimes C$}\put(62,3){\vector(1,0){80}}\put(87,7){$\alpha_{B,A,C}$}
\put(145,0){$B\otimes (A\otimes C)$}\put(227,3){\vector(1,0){80}}\put(242,7){$1_B\otimes\gamma_{A,C}$}
\put(310,0){$B\otimes (C\otimes A)$}

\end{picture}

\begin{picture}(370,60)\setlength{\unitlength}{0.3mm}

\put(-20,40){$((W\lozenge X)\lozenge Y)\lozenge Z$}\put(82,43){\vector(1,0){60}}\put(87,47){$\mathfrak{a}_{W\lozenge X,Y,Z}$}
\put(145,40){$(W\lozenge X)\lozenge (Y\lozenge Z)$}\put(247,43){\vector(1,0){60}}\put(252,47){$\mathfrak{a}_{W,X,Y\lozenge Z}$}
\put(310,40){$W\lozenge (X\lozenge (Y\lozenge Z))$}

\put(20,33){\vector(0,-1){20}}\put(23,20){$\mathfrak{a}_{W,X,Y}\lozenge 1_Z$}
\put(305,20){$1_W\lozenge \mathfrak{a}_{X,Y,Z}$}\put(370,13){\vector(0,1){20}}

\put(-20,0){$(W\lozenge (X\lozenge Y))\lozenge Z$}\put(82,03){\vector(1,0){226}}\put(150,7){$\mathfrak{a}_{W, X\lozenge Y,Z}$}\put(310,0){$W\lozenge ((X\lozenge Y)\lozenge Z),$}

\end{picture}

\begin{picture}(370,60)

\put(-20,40){$(X\lozenge I)\lozenge Y$}\put(52,43){\vector(1,0){60}}\put(67,47){$\mathfrak{a}_{X,I,Y}$}
\put(115,40){$X\lozenge (I\lozenge Y)$}

\put(220,40){$X\lozenge Y$}\put(257,43){\vector(1,0){40}}\put(262,47){$t_{X,Y}$}\put(300,40){$Y\lozenge X$}

\put(40,33){\vector(1,-1){20}}\put(0,20){$r_X\lozenge 1_Y$}\put(65,0){$X\lozenge Y$\ \ ,}\put(130,33){\vector(-1,-1){20}}\put(130,20){$1_X\lozenge(r_Y\circ t_{I,Y})$}

\put(230,20){$1_{X\lozenge Y}$}\put(243,33){\vector(2,-1){50}}
\put(317,33){\vector(0,-1){20}}\put(322,20){$t_{Y,X}$}\put(300,0){$X\lozenge Y$}

\end{picture}

and

\begin{picture}(370,60)\setlength{\unitlength}{0.31mm}

\put(-20,40){$(X\lozenge Y)\lozenge Z$}\put(50,43){\vector(1,0){90}}\put(87,47){$\mathfrak{a}_{X,Y,Z}$}
\put(145,40){$X\lozenge (Y\lozenge Z)$}\put(215,43){\vector(1,0){90}}\put(247,47){$t_{X,Y\lozenge Z}$}
\put(310,40){$(Y\lozenge Z)\lozenge X$}

\put(20,33){\vector(0,-1){20}}\put(23,20){$t_{X,Y}\lozenge 1_Z$}
\put(315,20){$\mathfrak{a}_{Y,Z,X}$}\put(350,33){\vector(0,-1){20}}\put(365,20){in $\mathbb{X}$;}

\put(-20,0){$(Y\lozenge X)\lozenge Z$}\put(50,3){\vector(1,0){90}}\put(87,7){$\mathfrak{a}_{Y,X,Z}$}
\put(145,0){$Y\lozenge (X\lozenge Z)$}\put(215,3){\vector(1,0){90}}\put(242,7){$1_Y\lozenge t_{X,Z}$}
\put(310,0){$Y\lozenge (Z\lozenge X)$}

\end{picture}\\

$F\alpha_{A,B,C}=\mathfrak{a}_{F(A),F(B),F(C)}$, $G\mathfrak{a}_{X,Y,Z}=\alpha_{G(X),G(Y),G(Z)}$, $F\gamma_{A,B}=t_{F(A),F(B)}$, $Gt_{X,Y}=\gamma_{G(X),G(Y)}$, $F\rho_A=r_{F(A)}$ and $Gr_X=\rho_{G(X)}$, for all $A,B,C\in\mathbb{C}$ and all $X,Y,Z\in\mathbb{X}$.\\

It is also assumed that, for every pair of objects $A,B\in\mathbb{C}$:\\

$(C)$ $\eta_A\otimes\eta_B=\eta_{A\otimes B}: A\otimes B\rightarrow GF(A\otimes B)$, \emph{the tensor product of two unit morphisms is the unit morphism of the tensor product}; and\\

$(D)$ $\eta_E=1_E:E\rightarrow E=GF(E)$, \emph{the unit morphism of the identity object of the monoidal category $\mathbb{C}$ is the identity morphism}.\\

The following Proposition \ref{proposition:duality eta epsilon} will be useful in identifying the monoidal adjunctions satisfying the conditions $(C)$ and $(D)$ just above.

\begin{proposition}\label{proposition:duality eta epsilon}

Consider an adjunction $(F,G,\eta,\varepsilon):\mathbb{C}\rightarrow\mathbb{X}$.
\begin{enumerate}

\item[$(i)$] If $F$ and $G$ preserve respectively the bifunctors $\otimes:\mathbb{C}\times\mathbb{C}\rightarrow\mathbb{C}$ and $\lozenge:\mathbb{X}\times\mathbb{X}\rightarrow\mathbb{X}$, then $$(\forall_{A,B\in\mathbb{C}}\ \eta_A\otimes\eta_B=\eta_{A\otimes B})\Leftrightarrow (\forall_{X,Y\in\mathbb{X}}\ \varepsilon_X\lozenge\varepsilon_Y=\varepsilon_{X\lozenge Y}).$$

\item[$(ii)$] If $F(E)=I$ and $G(I)=E$, then $$\eta_E=1_E\Leftrightarrow \varepsilon_I=1_I;$$ where $\eta_E$, $1_E$ are the unit and identity morphisms of $E\in\mathbb{C}$, and $\varepsilon_I$, $1_I$ are the counit and identity morphisms of $I\in\mathbb{X}$.

\end{enumerate}
\end{proposition}

\begin{proof}\

$(i)$

($\Rightarrow$) $1_{G(X\lozenge Y)}=G\varepsilon_{X\lozenge Y}\circ\eta_{G(X\lozenge Y)}$ (by $(A)$) and

$1_{G(X\lozenge Y)}=1_{G(X)\otimes G(Y)}$ (by $(B)$)

$=1_{G(X)}\otimes 1_{G(Y)}$ ($\otimes$ is a bifunctor)

$=(G\varepsilon_X\circ \eta_{G(X)})\otimes (G\varepsilon_Y\circ\eta_{G(Y)})$ (by $(A)$)

$=(G\varepsilon_X\otimes G\varepsilon_Y)\circ (\eta_{G(X)}\otimes\eta_{G(Y)})$ ($\otimes$ is a bifunctor)

$=G(\varepsilon_X\lozenge\varepsilon_Y)\circ (\eta_{G(X)}\otimes\eta_{G(Y)})$ (by $(B)$)

$=G(\varepsilon_X\lozenge\varepsilon_Y)\circ \eta_{G(X)\otimes G(Y)}$ (hypothesis)

$=G(\varepsilon_X\lozenge\varepsilon_Y)\circ \eta_{G(X\lozenge Y)}$ (by $(B)$)

imply that $\varepsilon_{X\lozenge Y}=\varepsilon_X\lozenge\varepsilon_Y$, because

$G\varepsilon_{X\lozenge Y}\circ\eta_{G(X\lozenge Y)}=G(\varepsilon_X\lozenge\varepsilon_Y)\circ\eta_{G(X\lozenge Y)}$ and $\eta_{G(X\lozenge Y)}$ is a unit morphism in the adjunction $(A)$;

($\Leftarrow$) $1_{F(A\otimes B)}=\varepsilon_{F(A\otimes B)}\circ F\eta_{A\otimes B}$ (by $(A)$) and

$1_{F(A\otimes B)}=1_{F(A)\lozenge F(B)}$ (by $(B)$)

$=1_{F(A)}\lozenge 1_{F(B)}$ ($\lozenge$ is a bifunctor)

$=(\varepsilon_{F(A)}\circ F\eta_{A})\lozenge (\varepsilon_{F(B)}\circ F\eta_{B})$ (by $(A)$)

$=(\varepsilon_{F(A)}\lozenge \varepsilon_{F(B)})\circ (F\eta_{A}\lozenge F\eta_{B})$ ($\lozenge$ is a bifunctor)

$=(\varepsilon_{F(A)}\lozenge \varepsilon_{F(B)})\circ F(\eta_{A}\otimes\eta_{B})$ (by $(B)$)

$=\varepsilon_{F(A)\lozenge F(B)}\circ F(\eta_{A}\otimes\eta_{B})$ (hypothesis)

$=\varepsilon_{F(A\otimes B)}\circ F(\eta_{A}\otimes\eta_{B})$ (by $(B)$)

imply that $\eta_{A\otimes B}=\eta_A\otimes\eta_B$, because

$\varepsilon_{F(A\otimes B)}\circ F\eta_{A\otimes B}=\varepsilon_{F(A\otimes B)}\circ F(\eta_A\otimes\eta_B)$ and $\varepsilon_{F(A\otimes B)}$ is a counit morphism in the adjunction $(A)$.

$(ii)$

($\Rightarrow$) $\varepsilon_{F(E)}\circ F\eta_E=1_{F(E)}$, because $(F,G,\eta,\varepsilon)$ is an adjunction (A),

$\Leftrightarrow\varepsilon_I\circ F\eta_E=1_{I}$, because $F(E)=I$ by (B),

$\Leftrightarrow\varepsilon_I=1_{I}$, because $\eta_E=1_E$ by hypothesis;

($\Leftarrow$) $G\varepsilon_I\circ \eta_{G(I)}=1_{G(I)}$, because $(F,G,\eta,\varepsilon)$ is an adjunction (A),

$\Leftrightarrow G\varepsilon_I\circ \eta_E=1_E$, because $G(I)=E$ by (B),

$\Leftrightarrow \eta_E=1_E$, because $\varepsilon_I=1_I$ by hypothesis.
\end{proof}

The data in this section \ref{sec-base monoidal adjunction} will be called \emph{the base monoidal adjunction}.

\section{Categories of all $\mathcal{V}$-categories ($\mathcal{V}=\mathbb{C},\mathbb{X}$)}\label{sec-Vcats}

A $\mathbb{C}$-category\footnote{Usually called $\mathcal{V}$-category, cf.\ \cite[\S 1.2]{Kelly:enriched_cat}. We will make $\mathcal{V}=\mathbb{C}$ or $\mathcal{V}=\mathbb{X}$ in the first part of the present paper.} $\mathcal{A}$ consists of a set of objects $ob(\mathcal{A})$, a \emph{hom-object} $\mathcal{A}(a,b)\in\mathbb{C}$ for each ordered pair $(a,b)$ of objects in $\mathcal{A}$, a \emph{composition law} $M^\mathcal{A}_{a,b,c}:\mathcal{A}(b,c)\otimes\mathcal{A}(a,b)\rightarrow\mathcal{A}(a,c)$ for each triple of objects, and an \emph{identity element} $j^\mathcal{A}_a:E\rightarrow\mathcal{A}(a,a)$ for each object, subject to the \emph{associativity} and \emph{unit axioms} expressed by the commutativity of the following two diagrams

\begin{picture}(370,60)\setlength{\unitlength}{0.315mm}

\put(-20,40){$(\mathcal{A}(c,d)\otimes \mathcal{A}(b,c))\otimes \mathcal{A}(a,b)$}\put(145,43){\vector(1,0){70}}
\put(135,55){$\alpha_{\mathcal{A}(c,d),\mathcal{A}(b,c),\mathcal{A}(a,b)}$}
\put(220,40){$\mathcal{A}(c,d)\otimes (\mathcal{A}(b,c)\otimes \mathcal{A}(a,b))$}

\put(20,33){\vector(0,-1){20}}\put(23,20){$M^\mathcal{A}_{b,c,d}\otimes 1_{\mathcal{A}(a,b)}$}
\put(300,33){\vector(0,-1){20}}\put(215,20){$1_{\mathcal{A}(c,d)}\otimes M^\mathcal{A}_{a,b,c}$}

\put(-20,0){$\mathcal{A}(b,d)\otimes\mathcal{A}(a,b)$}\put(80,3){\vector(1,0){60}}\put(90,-10){$M^\mathcal{A}_{a,b,d}$}
\put(150,0){$\mathcal{A}(a,d)$}\put(265,3){\vector(-1,0){70}}\put(220,-10){$M^\mathcal{A}_{a,c,d}$}
\put(270,0){$\mathcal{A}(c,d)\otimes \mathcal{A}(a,c)$,}

\end{picture}

\begin{picture}(370,80)\setlength{\unitlength}{0.33mm}

\put(-20,40){$\mathcal{A}(b,b)\otimes \mathcal{A}(a,b)$}\put(80,43){\vector(1,0){60}}\put(90,50){$M^\mathcal{A}_{a,b,b}$}
\put(150,40){$\mathcal{A}(a,b)$}\put(265,43){\vector(-1,0){70}}\put(220,50){$M^\mathcal{A}_{a,a,b}$}
\put(270,40){$\mathcal{A}(a,b)\otimes \mathcal{A}(a,a)$}

\put(20,13){\vector(0,1){20}}\put(23,20){$j^\mathcal{A}_b\otimes 1_{\mathcal{A}(a,b)}$}
\put(300,13){\vector(0,1){20}}\put(230,23){$1_{\mathcal{A}(a,b)}\otimes j^\mathcal{A}_a$}

\put(0,0){$E\otimes\mathcal{A}(a,b)$}\put(70,3){\vector(2,1){70}}
\put(95,10){$\rho_{\mathcal{A}(a,b)}\circ\gamma_{E,\mathcal{A}(a,b)}$}
\put(265,3){\vector(-2,1){70}}\put(200,10){$\rho_{\mathcal{A}(a,b)}$}
\put(270,0){$\mathcal{A}(a,b)\otimes E$.}

\end{picture}\\

For $\mathbb{C}$-categories $\mathcal{A}$ and $\mathcal{B}$, a $\mathbb{C}$-functor $T:\mathcal{A}\rightarrow \mathcal{B}$ consists of a function\footnote{In order to simplify the notation, we will use sometimes $T$ instead of $obT$, specially in diagrams.} $$obT:ob(\mathcal{A})\rightarrow ob(\mathcal{B})$$ together with, for each pair $a,b\in ob(\mathcal{A})$, a map $$T_{a,b}:\mathcal{A}(a,b)\rightarrow \mathcal{B}(T(a),T(b)),$$ subject to the \emph{compatibility with composition and with the identities} expressed by the commutativity of the following two diagrams,

\begin{picture}(360,60)\setlength{\unitlength}{0.3mm}

\put(45,40){$\mathcal{A}(b,c)\otimes \mathcal{A}(a,b)$}\put(145,43){\vector(1,0){60}}\put(155,50){$M^\mathcal{A}_{a,b,c}$}
\put(210,40){$\mathcal{A}(a,c)$}\put(335,40){$\mathcal{A}(a,a)$}

\put(85,33){\vector(0,-1){20}}\put(20,20){$T_{b,c}\otimes T_{a,b}$}
\put(225,33){\vector(0,-1){20}}\put(230,20){$T_{a,c}$}\put(260,20){$E$}
\put(275,25){\vector(3,1){55}}\put(295,40){$j^\mathcal{A}_a$}\put(275,20){\vector(3,-1){55}}
\put(295,-5){$j^\mathcal{B}_{T(a)}$}\put(345,33){\vector(0,-1){20}}\put(350,20){$T_{a,a}$}

\put(-10,0){$\mathcal{B}(T(b),T(c))\otimes \mathcal{B}(T(a),T(b))$}\put(170,3){\vector(1,0){25}}
\put(140,20){$M^\mathcal{B}_{T(a),T(b),T(c)}$}\put(200,0){$\mathcal{B}(T(a),T(c))$,}

\put(335,0){$\mathcal{B}(T(a),T(a))$.}

\end{picture}\vspace{10pt}

It is easy to verify that $\mathbb{C}$-categories and $\mathbb{C}$-functors constitute the category $\mathbb{C}$-$Cat$ of all $\mathbb{C}$-categories:

the composition $S\circ T$ of two $\mathbb{C}$-functors $S:\mathcal{B}\rightarrow \mathcal{C}$ and $T:\mathcal{A}\rightarrow \mathcal{B}$ is defined by $$ob(S\circ T):ob(\mathcal{A})\rightarrow ob(\mathcal{C}),\ ob(S\circ T)(a)=obS(obT(a)),$$ and $$(S\circ T)_{a,b}=S_{T(a),T(b)}\circ T_{a,b}$$ in $\mathbb{C}$, for every pair $a,b\in ob(\mathcal{A})$; the following two commutative diagrams show that $S\circ T$ as defined is a $\mathbb{C}$-functor provided $S$ and $T$ are also $\mathbb{C}$-functors,

\begin{picture}(370,80)\setlength{\unitlength}{0.26mm}

\put(40,80){$\mathcal{A}(b,c)\otimes\mathcal{A}(a,b)$}\put(160,83){\vector(1,0){160}}
\put(220,90){$M^\mathcal{A}_{a,b,c}$}\put(325,80){$\mathcal{A}(a,c)$}

\put(110,73){\vector(0,-1){20}}\put(40,60){$T_{b,c}\otimes T_{a,b}$}
\put(340,73){\vector(0,-1){20}}\put(345,60){$T_{a,c}$}

\put(-10,40){$\mathcal{B}(T(b),T(c))\otimes \mathcal{B}(T(a),T(b))$}\put(200,43){\vector(1,0){120}}
\put(220,50){$M^\mathcal{B}_{T(a),T(b),T(c)}$}\put(325,40){$\mathcal{B}(T(a),T(c))$}

\put(110,33){\vector(0,-1){20}}\put(-30,20){$S_{T(b),T(c)}\otimes S_{T(a),T(b)}$}
\put(340,33){\vector(0,-1){20}}\put(345,20){$S_{T(a),T(c)}$}

\put(-30,0){$\mathcal{C}(S\circ T(b),S\circ T(c))\otimes \mathcal{C}(S\circ T(a),S\circ T(b))$}\put(280,3){\vector(1,0){40}}
\put(200,20){$M^\mathcal{C}_{S\circ T(a),S\circ T(b),S\circ T(c)}$}\put(325,0){$\mathcal{C}(S\circ T(a),S\circ T(c))$,}

\end{picture}

\begin{picture}(240,75)

\put(3,40){$\mathcal{A}(a,a)$}\put(42,43){\vector(1,0){60}}\put(60,47){$T_{a,a}$}
\put(110,40){$\mathcal{B}(Ta,Ta)$}\put(167,43){\vector(1,0){60}}\put(172,47){$S_{T(a),T(a)}$}
\put(230,40){$C(S\circ T(a),S\circ T(a))$}

\put(127,13){\vector(-4,1){90}}\put(60,13){$j^\mathcal{A}_a$}
\put(130,13){\vector(0,1){20}}\put(135,25){$j^\mathcal{B}_{T(a)}$}\put(125,0){$E$\ ,}
\put(133,13){\vector(4,1){90}}\put(185,13){$j^\mathcal{C}_{S\circ T(a)}$}

\end{picture}\vspace{10pt}

\noindent notice that
$(S_{T(b),T(c)}\otimes S_{T(a),T(b)})\circ (T_{b,c}\otimes T_{a,b})=(S\circ T)_{b,c}\otimes (S\circ T)_{a,b}$, $S_{T(a),T(c)}\circ T_{a,c} =(S\circ T)_{a,c}$ and $S_{T(a),T(a)}\circ T_{a,a}=(S\circ T)_{a,a}$, because $\otimes$ is a bifunctor and by the definition of composition of $\mathbb{C}$-functors;

this composition of $\mathbb{C}$-functors is associative since

\noindent $(R\circ(S\circ T))_{a,b}=
R_{S\circ T(a),S\circ T(b)}\circ (S\circ T)_{a,b}=\\
R_{S\circ T(a),S\circ T(b)}\circ (S_{T(a),T(b)}\circ T_{a,b})=\\
(R_{S\circ T(a),S\circ T(b)}\circ S_{T(a),T(b)})\circ T_{a,b}=((R\circ S)\circ T)_{a,b}$;

for every $\mathcal{A}\in\mathbb{C}$-$Cat$ there is the unit $\mathbb{C}$-functor $1_\mathcal{A}:\mathcal{A}\rightarrow\mathcal{A}$, such that $ob1_\mathcal{A}:ob(\mathcal{A})\rightarrow ob(\mathcal{A})$ is the identity function $1_{ob(\mathcal{A})}$, and $$(1_\mathcal{A})_{a,b}=1_{\mathcal{A}(a,b)}$$
\noindent the unit morphism in $\mathcal{C}$, for every pair $a,b\in ob(\mathcal{A})$;

the two diagrams

\begin{picture}(350,60)\setlength{\unitlength}{0.33mm}

\put(40,40){$\mathcal{A}(b,c)\otimes\mathcal{A}(a,b)$}
\put(135,43){\vector(1,0){70}}\put(150,50){$M^\mathcal{A}_{a,b,c}$}
\put(210,40){$\mathcal{A}(a,c)$}\put(330,40){$\mathcal{A}(a,a)$}

\put(80,33){\vector(0,-1){20}}
\put(-10,20){$(1_\mathcal{A})_{b,c}\otimes (1_\mathcal{A})_{a,b}$}
\put(220,33){\vector(0,-1){20}}\put(180,20){$(1_\mathcal{A})_{a,c}$}\put(255,20){$E$}
\put(270,25){\vector(3,1){55}}\put(290,40){$j^\mathcal{A}_a$}\put(270,20){\vector(3,-1){55}}
\put(290,-5){$j^\mathcal{A}_a$}\put(350,33){\vector(0,-1){20}}\put(310,20){$(1_\mathcal{A})_{a,a}$}

\put(40,0){$\mathcal{A}(b,c)\otimes \mathcal{A}(a,b)$}
\put(135,3){\vector(1,0){70}}\put(150,10){$M^\mathcal{A}_{a,b,c}$}
\put(210,0){$\mathcal{A}(a,c)$,}

\put(330,0){$\mathcal{A}(a,a)$}

\end{picture}\vspace{10pt}

\noindent commute because $(1_\mathcal{A})_{b,c}\otimes (1_\mathcal{A})_{a,b}=1_{\mathcal{A}(b,c)}\otimes 1_{\mathcal{A}(a,b)}=1_{A(b,c)\otimes A(a,b)}$, $(1_\mathcal{A})_{a,c}=1_{\mathcal{A}(a,c)}$ and $(1_\mathcal{A})_{a,a}=1_{\mathcal{A}(a,a)}$, since $\otimes$ is a bifunctor and by definition of $1_\mathcal{A}$.\\

The following characterization of isomorphisms in $\mathbb{C}$-$Cat$ will be needed.

\begin{proposition}\label{proposition:IsosInC-Cat}
A $\mathbb{C}$-functor $T:\mathcal{A}\rightarrow\mathcal{B}$ is an isomorphism in $\mathbb{C}$-$Cat$ if and only if $obT:ob(\mathcal{A})\rightarrow ob(\mathcal{B})$ is a bijection (isomorphism in $Set$) and $T_{a,b}:\mathcal{A}(a,b)\rightarrow\mathcal{B}(T(a),T(b))$ is an isomorphism in $\mathbb{C}$, for each pair $a,b\in ob(\mathcal{A})$.
\end{proposition}

\begin{proof}
By definition of isomorphism in a category, $T$ is an isomorphism in $\mathbb{C}$-$Cat$ if there exists a $\mathbb{C}$-functor $S:\mathcal{B}\rightarrow\mathcal{A}$ such that $S\circ T=1_\mathcal{A}$ and $T\circ S=1_\mathcal{B}$. Hence, by definition of the unit $\mathbb{C}$-functors and of the composition of $\mathbb{C}$-functors, if $T$ is an isomorphism then $obT:ob(\mathcal{A})\rightarrow ob(\mathcal{B})$ must be a bijection and each $T_{a,b}$ must be an isomorphism in $\mathbb{C}$.

On the other hand, if $obT:ob(\mathcal{A})\rightarrow ob(\mathcal{B})$ is a bijection and each $T_{a,b}$ is an isomorphism, let's define

$obS=obT^{-1}:ob(\mathcal{B})\rightarrow ob(\mathcal{A})$, inverse of $obT$ in $Set$ and

 $S_{T(a),T(b)}=T_{a,b}^{-1}$, inverse of $T_{a,b}$ in $\mathbb{C}$.

Then, it is easy to show that $S$ is a $\mathbb{C}$-functor such that $S\circ T=1_\mathcal{A}$ and $T\circ S=1_\mathcal{B}$:

$obT^{-1}:ob(\mathcal{B})\rightarrow ob(\mathcal{A})$ is a bijection, therefore,

for all $c=T(a),\ d=T(b)\in ob(\mathcal{B})$,

$S_{c,d}=T_{a,b}^{-1}:\mathcal{B}(T(a),T(b))=\mathcal{B}(c,d)\rightarrow \mathcal{A}(a,b)=\mathcal{A}(T^{-1}(c),T^{-1}(d))$;

then the diagram

\begin{picture}(370,70)\setlength{\unitlength}{0.25mm}

\put(100,65){$\mathcal{B}(d,e)\otimes \mathcal{B}(c,d)$}\put(225,68){\vector(1,0){115}}
\put(245,75){$M^\mathcal{B}_{c,d,e}$}\put(345,65){$\mathcal{B}(c,e)$}

\put(130,58){\vector(0,-1){20}}
\put(21,45){$T^{-1}_{T^{-1}(d),T^{-1}(e)}\otimes T^{-1}_{T^{-1}(c),T^{-1}(d)}=S_{d,e}\otimes S_{c,d}$}
\put(360,58){\vector(0,-1){20}}\put(365,45){$S_{c,e}=T^{-1}_{T^{-1}(c),T^{-1}(e)}$}

\put(10,20){$\mathcal{A}(T^{-1}(d),T^{-1}(e))\otimes\mathcal{A}(T^{-1}(c),T^{-1}(d)))$}\put(305,23){\vector(1,0){35}}
\put(240,0){$M^\mathcal{A}_{T^{-1}(c),T^{-1}(d),T^{-1}(e)}$}
\put(345,20){$\mathcal{A}(T^{-1}(c),T^{-1}(e))$}

\end{picture}\vspace{10pt}

is commutative, for every $b,c,d,e\in ob(\mathcal{B})$, because

$T_{T^{-1}(c),T^{-1}(e)}\circ M^\mathcal{A}_{T^{-1}(c),T^{-1}(d),T^{-1}(e)}=M^\mathcal{B}_{c,d,e}\circ (T_{T^{-1}(d),T^{-1}(e)}\otimes T_{T^{-1}(c),T^{-1}(d)})$, since $T$ is a $C$-functor,

$\Leftrightarrow M^\mathcal{A}_{T^{-1}(c),T^{-1}(d),T^{-1}(e)}\circ (T^{-1}_{T^{-1}(d),T^{-1}(e)}\otimes T^{-1}_{T^{-1}(c),T^{-1}(d)})=\\
T^{-1}_{T^{-1}(c),T^{-1}(e)}\circ M^\mathcal{B}_{c,d,e}$,

\noindent notice that $(T_{T^{-1}(d),T^{-1}(e)}\otimes T_{T^{-1}(c),T^{-1}(d)})^{-1}=T^{-1}_{T^{-1}(d),T^{-1}(e)}\otimes T^{-1}_{T^{-1}(c),T^{-1}(d)}$ simply because $\otimes$ is a bifunctor,

$\Leftrightarrow M^\mathcal{A}_{T^{-1}(c),T^{-1}(d),T^{-1}(e)}\circ (S_{d,e}\otimes S_{c,d})=S_{c,e}\circ M^\mathcal{B}_{c,d,e}$;

$S_{T(a),T(a)}\circ j^\mathcal{B}_{T(a)}=j^\mathcal{A}_a\Leftrightarrow j^\mathcal{A}_a=T^{-1}_{a,a}\circ j^\mathcal{B}_{T(a)}\Leftrightarrow T_{a,a}\circ j^\mathcal{A}_a=j^\mathcal{B}_{T(a)}$, for every $a\in ob(\mathcal{A})$. \end{proof}

In exactly the same way $\mathbb{X}$-categories and $\mathbb{X}$-functors constitute the category $\mathbb{X}$-$Cat$ of all $\mathbb{X}$-categories.

\section{The derived monoidal adjunction}\label{sec:derived}

We will now derive a new monoidal adjunction, from the base monoidal adjunction presented in section \ref{sec-base monoidal adjunction}; its monoidal structure will be given in section \ref{section:derivedMonoidal}.\\

A new adjunction is going to be defined,

\begin{center}
$(\mathbb{F},\mathbb{G},\Theta,\Upsilon):\mathbb{C}$-$Cat\rightarrow\mathbb{X}$-$Cat$,
\end{center}

from the category of all $\mathbb{C}$-categories into the category of all $\mathbb{X}$-categories.\\

 $\mathbb{F}(\mathcal{A})$ is the object of $\mathbb{X}$-$Cat$ such that $ob\mathbb{F}(\mathcal{A})=ob(\mathcal{A})$, that is, the objects of $\mathbb{F}(\mathcal{A})$ are exactly the objects of $\mathcal{A}$, for every object $\mathcal{A}\in\mathbb{C}$-$Cat$;

the hom-object $\mathbb{F}(\mathcal{A})(a,b)=F(\mathcal{A}(a,b))\in \mathbb{X}$, for every pair $a,b\in ob(\mathcal{A})$;

the composition law $M^{\mathbb{F}(\mathcal{A})}_{a,b,c}=FM^\mathcal{A}_{a,b,c}$ in $\mathbb{X}$, for every triple $a,b,c\in ob(\mathcal{A})$;

the identity element $j^{\mathbb{F}(\mathcal{A})}_a=Fj^\mathcal{A}_a$, for every $a\in ob(\mathcal{A})$.\\

$\mathbb{F}T:\mathbb{F}(\mathcal{A})\rightarrow \mathbb{F}(\mathcal{B})$ is an $\mathbb{X}$-functor from the $\mathbb{X}$-category $\mathbb{F}(\mathcal{A})$ to the $\mathbb{X}$-category $\mathbb{F}(\mathcal{B})$, for every $\mathbb{C}$-functor $T:\mathcal{A}\rightarrow \mathcal{B}$:

the function $ob\mathbb{F}T:ob\mathbb{F}(\mathcal{A})\rightarrow ob\mathbb{F}(\mathcal{B})$ is $obT:ob(\mathcal{A})\rightarrow ob(\mathcal{B})$, that is, $ob\mathbb{F}T=obT$;

the map $(\mathbb{F} T)_{a,b}:\mathbb{F}(\mathcal{A})(a,b)\rightarrow\mathbb{F}(\mathcal{B})(T(a),T(b))$ is defined by $$(\mathbb{F}T)_{a,b}=FT_{a,b},$$ for each $a,b\in ob(\mathcal{A})$.\\

(1) First, let's check that $\mathbb{F}$ as defined is really a functor from $\mathbb{C}$-$Cat$ into $\mathbb{X}$-$Cat$.

(1.1) $\mathbb{F}(\mathcal{A})\in\mathbb{X}$-$Cat$, for any $\mathcal{A}\in\mathbb{C}$-$Cat$: consider the commutative diagrams which express associativity and unit axioms for $\mathcal{A}$ (cf.\ section \ref{sec-Vcats}); their $F$-image shows that $\mathbb{F}(\mathcal{A})$ as defined is an $\mathbb{X}$-category, since the functor $F$ is a monoidal morphism.

(1.2) For every $T:\mathcal{A}\rightarrow\mathcal{B}$ in $\mathbb{C}$-$Cat$, $\mathbb{F}T$ as defined above is an $\mathbb{X}$-functor: consider the commutative diagrams which express the compatibility with composition and with the identities of the $\mathbb{C}$-functor $T$ (cf.\ section \ref{sec-Vcats}); their $F$-image shows that $\mathbb{F}T$ as defined is an $\mathbb{X}$-functor, since the functor $F$ is a monoidal morphism.

(1.3) $\mathbb{F}$ is a functor if $\mathbb{F}1_\mathcal{A}=1_{\mathbb{F}(\mathcal{A})}$ and $\mathbb{F}(T'\circ T)=\mathbb{F}T'\circ\mathbb{F}T$, for every $\mathcal{A}\in\mathbb{C}$-$Cat$ and every pair $T,T'$ of composable $\mathbb{C}$-functors.

The unit $\mathcal{V}$-functor was characterized in general in section \ref{sec-Vcats} just before Proposition \ref{proposition:IsosInC-Cat}, so that $ob1_{\mathbb{F}(\mathcal{A})}=1_{ob\mathbb{F}(\mathcal{A})}$, the identity function, and $(1_{\mathbb{F}(\mathcal{A})})_{a,b}=1_{\mathbb{F}(\mathcal{A})(a,b)}$, the unit morphism in $\mathbb{X}$, for every pair $a,b\in ob\mathbb{F}(\mathcal{A})$. Hence, $\mathbb{F}1_\mathcal{A}=1_{\mathbb{F}(\mathcal{A})}$, the unit $\mathbb{X}$-functor, because $ob\mathbb{F}1_\mathcal{A}=ob1_\mathcal{A}=1_{ob(\mathcal{A})}=1_{ob\mathbb{F}(\mathcal{A})}$ and $(\mathbb{F}1_\mathcal{A})_{a,b}=F(1_\mathcal{A})_{a,b}=F1_{\mathcal{A}(a,b)}=1_{F(\mathcal{A}(a,b))}
=1_{\mathbb{F}(\mathcal{A})(a,b)}$.

$\mathbb{F}(T'\circ T)=\mathbb{F}T'\circ\mathbb{F}T$, for every pair $T,T'$ of composable $\mathbb{C}$-functors, because\footnote{Notice that ``$\circ$" in this paragraph means either the composition of morphisms in $\mathbb{C}$-$Cat$, $\mathbb{X}$-$Cat$, $\mathbb{C}$ or $\mathbb{X}$, either the composition of maps in $Set$. We leave to the reader the easy understanding of the meaning of each ``$\circ$" in what follows.} $ob\mathbb{F}(T'\circ T)=ob(T'\circ T)=obT'\circ obT=ob\mathbb{F}T'\circ ob\mathbb{F}T$ and $(\mathbb{F}(T'\circ T))_{a,b}=F(T'\circ T)_{a,b}=F(T'_{T(a),T(b)}\circ T_{a,b})=FT'_{T(a),T(b)}\circ FT_{a,b}=(\mathbb{F}T')_{T(a),T(b)}\circ (\mathbb{F}T)_{a,b}=(\mathbb{F}T'\circ \mathbb{F}T)_{a,b}$.\\

(2) The definition of $\mathbb{G}:\mathbb{X}$-$Cat\rightarrow \mathbb{C}$-$Cat$ is completely analogous to that of $\mathbb{F}$. Therefore, the proof that $\mathbb{G}$ is really a functor from $\mathbb{X}$-$Cat$ into $\mathbb{C}$-$Cat$ is also completely analogous to the one done for $\mathbb{F}$ just above in (1).\\

(3) Let $\mathcal{A}$ be a $\mathbb{C}$-category, then the unit morphism $\Theta_\mathcal{A}:\mathcal{A}\rightarrow \mathbb{GF}(\mathcal{A})$ is the one such that $ob\Theta_\mathcal{A}$ is the identity function $$1_{ob(\mathcal{A})}:ob(\mathcal{A})\rightarrow ob\mathbb{GF}(\mathcal{A})=ob\mathbb{F}(\mathcal{A})=ob(\mathcal{A}),$$ and, for every $a,b\in ob(\mathcal{A})$, $$(\Theta_\mathcal{A})_{a,b}=\eta_{\mathcal{A}(a,b)}:\mathcal{A}(a,b)\rightarrow GF(\mathcal{A}(a,b))$$ is the unit morphism for $\mathcal{A}(a,b)\in\mathbb{C}$ in the base adjunction $F\dashv G$.\\

\noindent $\Theta_\mathcal{A}$ is a $\mathbb{C}$-functor (morphism in $\mathbb{C}$-$Cat$):

\noindent (3.1) it is compatible with composition, since

\noindent $(\Theta_\mathcal{A})_{a,c}\circ M^\mathcal{A}_{a,b,c}=\eta_{\mathcal{A}(a,c)}\circ M^\mathcal{A}_{a,b,c}$, by definition of $\Theta$

\noindent $= GF(M^\mathcal{A}_{a,b,c})\circ\eta_{\mathcal{A}(b,c)\otimes \mathcal{A}(a,b)}$, because $\eta:1_\mathbb{C}\rightarrow GF$ is natural

\noindent $=M^{\mathbb{GF}(\mathcal{A})}_{a,b,c}\circ\eta_{\mathcal{A}(b,c)\otimes \mathcal{A}(a,b)}$, by definition of $\mathbb{F}$ and $\mathbb{G}$

\noindent $=M^{\mathbb{GF}(\mathcal{A})}_{a,b,c}\circ (\eta_{\mathcal{A}(b,c)}\otimes\eta_{\mathcal{A}(a,b)})$, by assumption $(C)$ (cf.\ section \ref{sec-base monoidal adjunction})

\noindent $=M^{\mathbb{GF}(\mathcal{A})}_{a,b,c}\circ ((\Theta_\mathcal{A})_{b,c}\otimes (\Theta_\mathcal{A})_{a,b})$, by definition of $\Theta$;

\noindent (3.2) it is compatible with the identities, since

$(\Theta_\mathcal{A})_{a,a}\circ j^\mathcal{A}_a=\eta_{\mathcal{A}(a,a)}\circ j^\mathcal{A}_a$, by definition of $\Theta$

$=GF(j^\mathcal{A}_a)\circ\eta_E$, because $\eta:1_\mathbb{C}\rightarrow GF$ is natural

$=j^{\mathbb{GF}(\mathcal{A})}_a\circ\eta_E$, by definition of $\mathbb{F}$ and $\mathbb{G}$

$=j^{\mathbb{GF}(\mathcal{A})}_a\circ 1_E$, by assumption $(D)$ (cf.\ section \ref{sec-base monoidal adjunction}).\\

(4) Let $\mathcal{X}\in\mathbb{X}$-category, then the counit morphism $\Upsilon_\mathcal{X}:\mathbb{FG}(\mathcal{X})\rightarrow\mathcal{X}$ is the one such that $ob\Upsilon_\mathcal{X}$ is the identity function $$1_{ob(\mathcal{X})}:ob\mathbb{FG}(\mathcal{X})=ob\mathbb{G}(\mathcal{X})= ob(\mathcal{X})\rightarrow ob(\mathcal{X}),$$ and, for every $x,y\in ob(\mathcal{X})$, $$(\Upsilon_\mathcal{X})_{x,y}=\varepsilon_{\mathcal{X}(x,y)}:FG(\mathcal{X}(x,y))\rightarrow\mathcal{X}(x,y)$$ is the counit morphism for $\mathcal{X}(x,y)\in\mathbb{X}$ in the base adjunction $F\dashv G$.\\

The proof that $\Upsilon_\mathcal{X}$ is an $\mathbb{X}$-functor (morphism in $\mathbb{X}$-$Cat$) is completely analogous to the one for $\Theta_\mathcal{A}$ given just above in (3).\\

(5) It will now be shown that $\Theta_\mathcal{A}:\mathcal{A}\rightarrow \mathbb{GF}(\mathcal{A})$ is a universal arrow from $\mathcal{A}$ to $\mathbb{G}$, for every $\mathcal{A}\in \mathbb{C}$-$Cat$:\footnote{Cf.\ last footnote.}

\noindent (5.1) consider any $\mathbb{C}$-functor $T:\mathcal{A}\rightarrow\mathbb{G}(\mathcal{X})$; if there is an $\mathbb{X}$-functor $U:\mathbb{F}(\mathcal{A})\rightarrow \mathcal{X}$ such that $\mathbb{G}U\circ\Theta_\mathcal{A}=T$, then it must be unique, since $obT=ob(\mathbb{G}U\circ\Theta_\mathcal{A})$

$=ob\mathbb{G}U\circ ob\Theta_\mathcal{A}$, by definition of composition in $\mathbb{C}$-$Cat$

$=obU\circ 1_{ob(\mathcal{A})}$, by the definitions of $\mathbb{G}$ and $\Theta$ (cf.\ (3))

$=obU$,

\noindent and $T_{a,b}=GU_{a,b}\circ\eta_{\mathcal{A}(a,b)}$ implies that $U_{a,b}$ is unique, for $\eta_{\mathcal{A}(a,b)}$ is a universal arrow from $\mathcal{A}(a,b)$ to $G$;

\noindent (5.2) let's check finally that $U$ given in (5.1) is an $\mathbb{X}$-functor (a morphism in $\mathbb{X}$-$Cat$);

\noindent (5.2.1) $U$ is compatible with composition:

$G(U_{a,c}\circ FM^\mathcal{A}_{a,b,c})\circ\eta_{\mathcal{A}(b,c)\otimes\mathcal{A}(a,b)}=GU_{a,c}\circ (\eta_{\mathcal{A}(a,c)}\circ M^\mathcal{A}_{a,b,c})$, because $\eta:1_\mathbb{C}\rightarrow GF$ is natural

$=T_{a,c}\circ M^\mathcal{A}_{a,b,c}$, because $T=\mathbb{G}U\circ\Theta_\mathcal{A}$

$=GM^\mathcal{X}_{T(a),T(b),T(c)}\circ (T_{b,c}\otimes T_{a,b})$, since $T$ is a $\mathbb{C}$-functor

$=GM^\mathcal{X}_{T(a),T(b),T(c)}\circ ((GU_{b,c}\circ\eta_{\mathcal{A}(b,c)})\otimes (GU_{a,b}\circ\eta_{\mathcal{A}(a,b)}))$, because $T=\mathbb{G}U\circ\Theta_\mathcal{A}$

$=G(M^\mathcal{X}_{T(a),T(b),T(c)}\circ (U_{b,c}\lozenge U_{a,b}))\circ (\eta_{\mathcal{A}(b,c)}\otimes\eta_{\mathcal{A}(a,b)})$, because $\otimes$ is a bifunctor

$=G(M^\mathcal{X}_{T(a),T(b),T(c)}\circ (U_{b,c}\lozenge U_{a,b}))\circ \eta_{\mathcal{A}(b,c)\otimes\mathcal{A}(a,b)}$, by assumption (C) (cf.\ section \ref{sec-base monoidal adjunction})

$\Rightarrow U_{a,c}\circ FM^\mathcal{A}_{a,b,c}=M^\mathcal{X}_{T(a),T(b),T(c)}\circ (U_{b,c}\lozenge U_{a,b})$, since $\eta_{\mathcal{A}(b,c)\otimes\mathcal{A}(a,b)}$ is universal from $\mathcal{A}(b,c)\otimes\mathcal{A}(a,b)$ to $G$.

\noindent (5.2.2) $U$ is compatible with the identities:

$Gj^\mathcal{X}_{T(a)}\circ\eta_E=T_{a,a}\circ j^\mathcal{A}_a\circ\eta_E$, because $T$ is a $\mathbb{C}$-functor and $GF(E)=G(I)=E$

$=GU_{a,a}\circ (\eta_{\mathcal{A}(a,a)}\circ j^\mathcal{A}_a)\circ\eta_E$

$=GU_{a,a}\circ (Gj^{\mathbb{F}(\mathcal{A})}_a\circ\eta_E)\circ\eta_E$, because $\eta:1_\mathbb{C}\rightarrow GF$ is natural

$=GU_{a,a}\circ Gj^{\mathbb{F}(\mathcal{A})}_a\circ\eta_E$, because $\eta_E=1_E$ by assumption (D) (cf.\ section \ref{sec-base monoidal adjunction})

$\Rightarrow j^\mathcal{X}_{T(a)}=U_{a,a}\circ j^{\mathbb{F}(\mathcal{A})}_a$, since $\eta_E$ is a universal arrow from $E$ to $G$.\\

(6) It will now be shown that $\Upsilon_\mathcal{X}:\mathbb{FG}(\mathcal{X})\rightarrow \mathcal{X} $ is a universal arrow from $\mathbb{F}$ to $\mathcal{X}$, for every $\mathcal{X}\in \mathbb{X}$-$Cat$:

\noindent (6.1) consider any $\mathcal{X}$-functor $T:\mathbb{F}(\mathcal{A})\rightarrow\mathcal{X}$; if there is a $\mathbb{C}$-functor $V:\mathcal{A}\rightarrow\mathbb{G}(\mathcal{X})$ such that $\Upsilon_\mathcal{X}\circ\mathbb{F}V=T$, it must be unique, since

$obT=ob(\Upsilon_\mathcal{X}\circ\mathbb{F}V)=ob\Upsilon_\mathcal{X}\circ ob\mathbb{F}V$, by definition of composition in $\mathbb{X}$-$Cat$

$=obV$, by the definitions of $\mathbb{F}$ and $\Upsilon$ (cf.\ (4)),

\noindent and $T_{a,b}=\varepsilon_{\mathcal{X}(T(a),T(b))}\circ FV_{a,b}$ implies that $V_{a,b}$ is unique, for $\varepsilon_{\mathcal{X}(T(a),T(b))}$ is a universal arrow from $F$ to $\mathcal{X}(T(a),T(b))$;

\noindent (6.2) let's check finally that $V$ given in (6.1) is an $\mathbb{X}$-functor (a morphism in $\mathbb{X}$-$Cat$);

\noindent (6.2.1) $V$ is compatible with composition:

$\varepsilon_{\mathcal{X}(T(a),T(c))}\circ F(V_{a,c}\circ M^\mathcal{A}_{a,b,c})=T_{a,c}\circ FM^\mathcal{A}_{a,b,c}$, because $T=\Upsilon_{\mathcal{X}}\circ \mathbb{F}V$

$=M^\mathcal{X}_{T(a),T(b),T(c)}\circ (T_{b,c}\lozenge T_{a,b})$, because $T$ is an $\mathbb{X}$-functor

$=M^\mathcal{X}_{T(a),T(b),T(c)}\circ ((\varepsilon_{\mathcal{X}(T(b),T(c))}\circ FV_{b,c})\lozenge (\varepsilon_{\mathcal{X}(T(a),T(b))}\circ FV_{a,b}))$, because $T=\Upsilon_\mathcal{X}\circ\mathbb{F}V$

$=M^\mathcal{X}_{T(a),T(b),T(c)}\circ ((\varepsilon_{\mathcal{X}(T(b),T(c))}\lozenge \varepsilon_{\mathcal{X}(T(a),T(b))})\circ (FV_{b,c}\lozenge FV_{a,b}))$, because $\lozenge$ is a bifunctor

$=M^\mathcal{X}_{T(a),T(b),T(c)}\circ (\varepsilon_{{\mathcal{X}(T(b),T(c))}\lozenge \mathcal{X}(T(a),T(b))}\circ F(V_{b,c}\otimes V_{a,b}))$, by assumption (C) and Proposition \ref{proposition:duality eta epsilon}

$=\varepsilon_{\mathcal{X}(T(a),T(c))}\circ FG(M^\mathcal{X}_{T(a),T(b),T(c)})\circ F(V_{b,c}\otimes V_{a,b})$, because $\varepsilon :FG\rightarrow 1_\mathbb{X}$ is natural

$=\varepsilon_{\mathcal{X}(T(a),T(c))}\circ F(GM^\mathcal{X}_{T(a),T(b),T(c)}\circ (V_{b,c}\otimes V_{a,b}))$

$\Rightarrow V_{a,c}\circ M^\mathcal{A}_{a,b,c}=GM^\mathcal{X}_{T(a),T(b),T(c)}\circ (V_{b,c}\otimes V_{a,b})$, since $\varepsilon_{\mathcal{X}(T(a),T(c))}$ is universal from $F$ to $\mathcal{X}(T(a),T(c))$.

\noindent (6.2.2) $V$ is compatible with the identities:

$\varepsilon_{\mathcal{X}(T(a),T(a))}\circ FGj^\mathcal{X}_{T(a)}=j^\mathcal{X}_{T(a)}\circ\varepsilon_{FG(I)}$, since $\varepsilon :FG\rightarrow 1_\mathbb{X}$ is natural

$=j^\mathcal{X}_{T(a)}\circ\varepsilon_I$, because $FG(I)=F(E)=I$

$=j^\mathcal{X}_{T(a)}$, by assumption (D) and Proposition \ref{proposition:duality eta epsilon}

$=T_{a,a}\circ j^{\mathbb{F}(\mathcal{A})}_a=T_{a,a}\circ Fj^\mathcal{A}_a$, because $T$ is a $\mathbb{C}$-functor

$=\varepsilon_{\mathcal{X}(T(a),T(a))}\circ (FV_{a,a}\circ Fj^\mathcal{A}_a)$

$\Rightarrow Gj^\mathcal{X}_{T(a)}(=j^{\mathbb{G}(\mathcal{X})}_{T(a)})=V_{a,a}\circ j^\mathcal{A}_a$, since $\varepsilon_{\mathcal{X}(T(a),T(a))}$ is a universal arrow from $F$ to $\mathcal{X}(T(a),T(a))$.\\

(7) It is easy to check that $\mathbb{G}\Upsilon\cdot\Theta\mathbb{G}=\mathbb{G}$ and $\Upsilon\mathbb{F}\cdot\mathbb{F}\Theta=\mathbb{F}$, using the assumptions $G\varepsilon\cdot\eta G=G$ and $\varepsilon F\cdot F\eta=F$.

\section{A monoidal structure for the derived adjunction}\label{section:derivedMonoidal}

In this section, the symmetric monoidal category

\begin{center}
$\mathbb{C}$-$Cat=(\mathbb{C}$-$Cat,\bigcirc,\mathfrak{E},\wedge,\Gamma,\mho)$,
\end{center}

will be presented, that is, a monoidal structure for the category of all $\mathbb{C}$-categories, giving the definition of all its items, and proving along the way that it obeys the axioms (cf.\ the first four sections of chapter 1 in \cite{Kelly:enriched_cat}).

\subsection{A bifunctor for $\mathbb{C}$-$Cat$}\label{subsection:A bifunctor}

Consider the (bi)functor

\begin{center}
$\bigcirc :\mathbb{C}$-$Cat\times\mathbb{C}$-$Cat\rightarrow\mathbb{C}$-$Cat$,
\end{center}

such that $T\bigcirc S:\mathcal{A}\bigcirc \mathcal{B}\rightarrow \mathcal{A}'\bigcirc \mathcal{B}'$ is the $\mathbb{C}$-functor image of $(T,S):(\mathcal{A},\mathcal{B})\rightarrow (\mathcal{A}',\mathcal{B}')$.\\

(I) Definition of the $\mathbb{C}$-category $\mathcal{A}\bigcirc \mathcal{B}$, for any pair $\mathcal{A},\mathcal{B}\in \mathbb{C}$-$Cat$:\\

$ob(\mathcal{A}\bigcirc\mathcal{B})=ob(\mathcal{A})\times ob(\mathcal{B})$, the set of objects of a tensor product $\mathcal{A}\bigcirc\mathcal{B}$ is the cartesian product of the sets of objects of the two $\mathbb{C}$-categories;\\

$\mathcal{A}\bigcirc\mathcal{B}((a,b),(\bar{a},\bar{b}))=\mathcal{A}(a,\bar{a})\otimes\mathcal{B}(b,\bar{b})$, a hom-object of the new tensor product is the old tensor product of the corresponding hom-objects for the two $\mathbb{C}$-categories;\\

the composition law $$M^{\mathcal{A}\bigcirc \mathcal{B}}_{(a,b),(\bar{a},\bar{b}),(\bar{\bar{a}},\bar{\bar{b}})}:
\mathcal{A}\bigcirc\mathcal{B}((\bar{a},\bar{b}),(\bar{\bar{a}},\bar{\bar{b}}))\otimes
\mathcal{A}\bigcirc\mathcal{B}((a,b),(\bar{a},\bar{b}))\rightarrow
\mathcal{A}\bigcirc\mathcal{B}((a,b),(\bar{\bar{a}},\bar{\bar{b}}))$$ is defined as\\

\noindent $M^{\mathcal{A}\bigcirc \mathcal{B}}_{(a,b),(\bar{a},\bar{b}),(\bar{\bar{a}},\bar{\bar{b}})}=
(M^\mathcal{A}_{a,\bar{a},\bar{\bar{a}}}\otimes M^\mathcal{B}_{b,\bar{b},\bar{\bar{b}}})\circ m_{\mathcal{A}(\bar{a},\bar{\bar{a}}),\mathcal{B}(\bar{b},\bar{\bar{b}}),\mathcal{A}(a,\bar{a}),\mathcal{B}(b,\bar{b})}:\\
(\mathcal{A}(\bar{a},\bar{\bar{a}})\otimes\mathcal{B}(\bar{b},\bar{\bar{b}}))\otimes
(\mathcal{A}(a,\bar{a})\otimes\mathcal{B}(b,\bar{b}))\rightarrow
(\mathcal{A}(\bar{a},\bar{\bar{a}})\otimes\mathcal{A}(a,\bar{a}))\otimes\\
\otimes (\mathcal{B}(\bar{b},\bar{\bar{b}})\otimes\mathcal{B}(b,\bar{b}))\rightarrow
\mathcal{A}(a,\bar{\bar{a}})\otimes\mathcal{B}(b,\bar{\bar{b}})$,\\

where, for instance, $m$ is the natural isomorphism whose components are\footnote{Remark that the results outside this subsection \ref{subsection:A bifunctor} do not need the symmetry in the monoidal structure; cf.\ \cite[I.1.4]{Kelly:enriched_cat}.}

$m_{A,B,C,D}=\alpha^{-1}_{A,C,B\otimes D}\circ (1_A\otimes\alpha_{C,B,D})\circ (1_A\otimes (\gamma_{B,C}\otimes 1_D))\circ \\
\circ (1_A\otimes\alpha^{-1}_{B,C,D})\circ\alpha_{A,B,C\otimes D}$,

for any quadruple $A,B,C,D\in\mathbb{C}$;\\

the unit law $$j^{\mathcal{A}\bigcirc \mathcal{B}}_{(a,b)}:E\rightarrow\mathcal{A}\bigcirc \mathcal{B}((a,b),(a,b))$$ is defined as $j^{\mathcal{A}\bigcirc \mathcal{B}}_{(a,b)}=(j^\mathcal{A}_a\otimes j^\mathcal{B}_b)\circ\rho_E:E\rightarrow E\otimes E\rightarrow\mathcal{A}(a,a)\otimes\mathcal{B}(b,b)$.\\

(II) Definition of the $\mathbb{C}$-functor $T\bigcirc S:\mathcal{A}\bigcirc \mathcal{B}\rightarrow \mathcal{A}'\bigcirc \mathcal{B}'$, for any pair $T:\mathcal{A}\rightarrow \mathcal{A}'$, $S:\mathcal{B}\rightarrow \mathcal{B}'$ in $\mathbb{C}$-$Cat$:\\

$T\bigcirc S$ consists of a function $obT\bigcirc S=obT\times obS$, which is the cartesian product of the two object functions for the two $\mathbb{C}$-functors, together with the map $$T\bigcirc S_{(a,b),(\bar{a},\bar{b})}=T_{a,\bar{a}}\otimes S_{b,\bar{b}}:\mathcal{A}(a,\bar{a})\otimes\mathcal{B}(b,\bar{b})\rightarrow\mathcal{A}'(T(a),T(\bar{a}))\otimes\mathcal{B}'(S(b),S(\bar{b})),$$ for each pair $(a,b),(\bar{a},\bar{b})\in ob(\mathcal{A}\bigcirc\mathcal{B})$.\\

We have to check that:

\noindent (1) $\mathcal{A}\bigcirc\mathcal{B}$ as defined is in fact a $\mathbb{C}$-category, that is,

(1.1) the associativity axiom, and

(1.2) the unit axioms hold; and then that,

\noindent (2) $T\bigcirc S:\mathcal{A}\bigcirc \mathcal{B}\rightarrow\mathcal{A}'\bigcirc \mathcal{B}'$ as defined is in fact a morphism in $\mathbb{C}$-$Cat$ (a $\mathbb{C}$-functor), that is,

(2.1) there is compatibility with composition, and

(2.2) with the identities; finally, that

\noindent (3) $\bigcirc$ is a (bi)functor, that is,

(3.1) $\bigcirc$ preserves the identities, and

(3.2) $\bigcirc$ preserves the composition.\\

\noindent (1) $\mathcal{A}\bigcirc \mathcal{B}\in\mathbb{C}$-$Cat$:\\

(1.1) Is the following equation true?\\

\noindent $
M^{\mathcal{A}\bigcirc \mathcal{B}}_{(a,b),(\bar{\bar{a}},\bar{\bar{b}}),(\bar{\bar{\bar{a}}},\bar{\bar{\bar{b}}})}\circ (1_{\mathcal{A}\bigcirc \mathcal{B}((\bar{\bar{a}},\bar{\bar{b}}),(\bar{\bar{\bar{a}}},\bar{\bar{\bar{b}}}))}\otimes M^{\mathcal{A}\bigcirc \mathcal{B}}_{(a,b),(\bar{a},\bar{b}),(\bar{\bar{a}},\bar{\bar{b}})})\circ\alpha_{\mathcal{A}\bigcirc \mathcal{B}((\bar{\bar{a}},\bar{\bar{b}}),(\bar{\bar{\bar{a}}},\bar{\bar{\bar{b}}})),\mathcal{A}\bigcirc \mathcal{B}((\bar{a},\bar{b}),(\bar{\bar{a}},\bar{\bar{b}})),\mathcal{A}\bigcirc \mathcal{B}((a,b),(\bar{a},\bar{b}))}=\\
=M^{\mathcal{A}\bigcirc \mathcal{B}}_{(a,b),(\bar{a},\bar{b}),(\bar{\bar{\bar{a}}},\bar{\bar{\bar{b}}})}\circ (M^{\mathcal{A}\bigcirc \mathcal{B}}_{(\bar{a},\bar{b}),(\bar{\bar{a}},\bar{\bar{b}}),(\bar{\bar{\bar{a}}},\bar{\bar{\bar{b}}})}\otimes 1_{\mathcal{A}\bigcirc \mathcal{B}((a,b),(\bar{a},\bar{b}))})
$\\

which is equivalent to the equation\\

\noindent \textbf{(i)} $
(M^\mathcal{A}_{a,\bar{\bar{a}},\bar{\bar{\bar{a}}}}\otimes M^\mathcal{B}_{b,\bar{\bar{b}},\bar{\bar{\bar{b}}}})
\circ m_{\mathcal{A}(\bar{\bar{a}},\bar{\bar{\bar{a}}}),\mathcal{B}(\bar{\bar{b}},\bar{\bar{\bar{b}}}),\mathcal{A}(a,\bar{\bar{a}}),\mathcal{B}(b,\bar{\bar{b}})}
\circ (1_{\mathcal{A}(\bar{\bar{a}},\bar{\bar{\bar{a}}})\otimes\mathcal{B}(\bar{\bar{b}},\bar{\bar{\bar{b}}})}\otimes (M^\mathcal{A}_{a,\bar{a},\bar{\bar{a}}}\otimes M^\mathcal{B}_{b,\bar{b},\bar{\bar{b}}}))\circ
(1_{\mathcal{A}(\bar{\bar{a}},\bar{\bar{\bar{a}}})\otimes\mathcal{B}(\bar{\bar{b}},\bar{\bar{\bar{b}}})}
\otimes m_{\mathcal{A}(\bar{a},\bar{\bar{a}}),\mathcal{B}(\bar{b},\bar{\bar{b}}),\mathcal{A}(a,\bar{a}),\mathcal{B}(b,\bar{b})})
\circ
\alpha_{\mathcal{A}(\bar{\bar{a}},\bar{\bar{\bar{a}}})\otimes\mathcal{B}(\bar{\bar{b}},
\bar{\bar{\bar{b}}}),\mathcal{A}(\bar{a},\bar{\bar{a}})\otimes\mathcal{B}(\bar{b},\bar{\bar{b}}),
\mathcal{A}(a,\bar{a})\otimes\mathcal{B}(b,\bar{b})}=\\
=(M^\mathcal{A}_{a,\bar{a},\bar{\bar{\bar{a}}}}\otimes M^\mathcal{B}_{b,\bar{b},\bar{\bar{\bar{b}}}})
\circ m_{\mathcal{A}(\bar{a},\bar{\bar{\bar{a}}}),\mathcal{B}(\bar{b},\bar{\bar{\bar{b}}}),\mathcal{A}(a,\bar{a}),\mathcal{B}(b,\bar{b})}
\circ ((M^\mathcal{A}_{\bar{a},\bar{\bar{a}},\bar{\bar{\bar{a}}}}\otimes M^\mathcal{B}_{\bar{b},\bar{\bar{b}},\bar{\bar{\bar{b}}}})\otimes 1_{\mathcal{A}(a,\bar{a})\otimes\mathcal{B}(b,\bar{b})})\circ
(m_{\mathcal{A}(\bar{\bar{a}},\bar{\bar{\bar{a}}}),\mathcal{B}(\bar{\bar{b}},\bar{\bar{\bar{b}}}),\mathcal{A}(\bar{a},\bar{\bar{a}}),\mathcal{B}(\bar{b},\bar{\bar{b}})}
\otimes1_{\mathcal{A}(a,\bar{a})\otimes\mathcal{B}(b,\bar{b})})$;\\

remark that, as $\mathcal{A}$ and $\mathcal{B}$ are $\mathbb{C}$-categories, the associativity axioms for each are the following two equations, for every $a,\bar{a},\bar{\bar{a}},\bar{\bar{\bar{a}}}\in\mathcal{A}$ and $b,\bar{b},\bar{\bar{b}},\bar{\bar{\bar{b}}}\in\mathcal{B}$,
$$M^\mathcal{A}_{a,\bar{\bar{a}},\bar{\bar{\bar{a}}}}\circ (1_{\mathcal{A}(\bar{\bar{a}},\bar{\bar{\bar{a}}})}\otimes M^\mathcal{A}_{a,\bar{a},\bar{\bar{a}}})\circ
\alpha_{\mathcal{A}(\bar{\bar{a}},\bar{\bar{\bar{a}}}),\mathcal{A}(\bar{a},\bar{\bar{a}}),\mathcal{A}(a,\bar{a})}=
M^\mathcal{A}_{a,\bar{a},\bar{\bar{\bar{a}}}}\circ (M^\mathcal{A}_{\bar{a},\bar{\bar{a}},\bar{\bar{\bar{a}}}}\otimes 1_{\mathcal{A}(a,\bar{a})})$$
and
$$M^\mathcal{B}_{b,\bar{\bar{b}},\bar{\bar{\bar{b}}}}\circ (1_{\mathcal{B}(\bar{\bar{b}},\bar{\bar{\bar{b}}})}\otimes M^\mathcal{B}_{b,\bar{b},\bar{\bar{b}}})\circ
\alpha_{\mathcal{B}(\bar{\bar{b}},\bar{\bar{\bar{b}}}),\mathcal{B}(\bar{b},\bar{\bar{b}}),\mathcal{B}(b,\bar{b})}=
M^\mathcal{B}_{b,\bar{b},\bar{\bar{\bar{b}}}}\circ (M^\mathcal{B}_{\bar{b},\bar{\bar{b}},\bar{\bar{\bar{b}}}}\otimes 1_{\mathcal{B}(b,\bar{b})}),$$ which imply, by tensoring,\\

\noindent \textbf{(ii)} $(M^\mathcal{A}_{a,\bar{\bar{a}},\bar{\bar{\bar{a}}}}\otimes M^\mathcal{B}_{b,\bar{\bar{b}},\bar{\bar{\bar{b}}}})\circ ((1_{\mathcal{A}(\bar{\bar{a}},\bar{\bar{\bar{a}}})}\otimes M^\mathcal{A}_{a,\bar{a},\bar{\bar{a}}})\otimes(1_{\mathcal{B}(\bar{\bar{b}},\bar{\bar{\bar{b}}})}\otimes M^\mathcal{B}_{b,\bar{b},\bar{\bar{b}}}))\circ
(\alpha_{\mathcal{A}(\bar{\bar{a}},\bar{\bar{\bar{a}}}),\mathcal{A}(\bar{a},\bar{\bar{a}}),\mathcal{A}(a,\bar{a})}
\otimes\alpha_{\mathcal{B}(\bar{\bar{b}},\bar{\bar{\bar{b}}}),\mathcal{B}(\bar{b},\bar{\bar{b}}),\mathcal{B}(b,\bar{b})})=
(M^\mathcal{A}_{a,\bar{a},\bar{\bar{\bar{a}}}}\otimes M^\mathcal{B}_{b,\bar{b},\bar{\bar{\bar{b}}}})\circ ((M^\mathcal{A}_{\bar{a},\bar{\bar{a}},\bar{\bar{\bar{a}}}}\otimes 1_{\mathcal{A}(a,\bar{a})})\otimes (M^\mathcal{B}_{\bar{b},\bar{\bar{b}},\bar{\bar{\bar{b}}}}\otimes 1_{\mathcal{B}(b,\bar{b})}))$;\\

the diagram corresponding to \textbf{(ii)}, which is a pentagon in which the top edge is $\alpha_{\mathcal{A}(\bar{\bar{a}},\bar{\bar{\bar{a}}}),\mathcal{A}(\bar{a},\bar{\bar{a}}),\mathcal{A}(a,\bar{a})}
\otimes\alpha_{\mathcal{B}(\bar{\bar{b}},\bar{\bar{\bar{b}}}),\mathcal{B}(\bar{b},\bar{\bar{b}}),\mathcal{B}(b,\bar{b})}$, can be extended upwards by the commutative rectangle (whose down edge coincides with the top edge $\alpha_{\mathcal{A}(\bar{\bar{a}},\bar{\bar{\bar{a}}}),\mathcal{A}(\bar{a},\bar{\bar{a}}),\mathcal{A}(a,\bar{a})}
\otimes\alpha_{\mathcal{B}(\bar{\bar{b}},\bar{\bar{\bar{b}}}),\mathcal{B}(\bar{b},\bar{\bar{b}}),\mathcal{B}(b,\bar{b})}$ of \textbf{(ii)}) corresponding to the equation\\

\noindent \textbf{(iii)} $(\alpha_{\mathcal{A}(\bar{\bar{a}},\bar{\bar{\bar{a}}}),\mathcal{A}(\bar{a},\bar{\bar{a}}),\mathcal{A}(a,\bar{a})}
\otimes\alpha_{\mathcal{B}(\bar{\bar{b}},\bar{\bar{\bar{b}}}),\mathcal{B}(\bar{b},\bar{\bar{b}}),\mathcal{B}(b,\bar{b})})
\circ m_{\mathcal{A}(\bar{\bar{a}},\bar{\bar{\bar{a}}})\otimes\mathcal{A}(\bar{a},\bar{\bar{a}}),
\mathcal{B}(\bar{\bar{b}},\bar{\bar{\bar{b}}})\otimes\mathcal{B}(\bar{b},\bar{\bar{b}}),\mathcal{A}(a,\bar{a}),\mathcal{B}(b,\bar{b})}
\circ (m_{\mathcal{A}(\bar{\bar{a}},\bar{\bar{\bar{a}}}),\mathcal{B}(\bar{\bar{b}},\bar{\bar{\bar{b}}}),
\mathcal{A}(\bar{a},\bar{\bar{a}}),\mathcal{B}(\bar{b},\bar{\bar{b}}))}
\otimes 1_{\mathcal{A}(a,\bar{a})\otimes\mathcal{B}(b,\bar{b})})=m_{\mathcal{A}(\bar{\bar{a}},\bar{\bar{\bar{a}}}),
\mathcal{B}(\bar{\bar{b}},\bar{\bar{\bar{b}}}),\mathcal{A}(\bar{a},\bar{\bar{a}})\otimes\mathcal{A}(a,\bar{a}),
\mathcal{B}(\bar{b},\bar{\bar{b}})\otimes\mathcal{B}(b,\bar{b})}\circ (1_{\mathcal{A}(\bar{\bar{a}},\bar{\bar{\bar{a}}})\otimes\mathcal{B}(\bar{\bar{b}},\bar{\bar{\bar{b}}})}
\otimes m_{\mathcal{A}(\bar{a},\bar{\bar{a}}),
\mathcal{B}(\bar{b},\bar{\bar{b}}),\mathcal{A}(a,\bar{a}),\mathcal{B}(b,\bar{b})})\circ \alpha_{\mathcal{A}(\bar{\bar{a}},\bar{\bar{\bar{a}}})\otimes
\mathcal{B}(\bar{\bar{b}},\bar{\bar{\bar{b}}}),\mathcal{A}(\bar{a},\bar{\bar{a}})\otimes\mathcal{B}(\bar{b},\bar{\bar{b}}),
\mathcal{A}(a,\bar{a})\otimes\mathcal{B}(b,\bar{b})} $;\\

the equation \textbf{(iii)} holds, since every diagram of natural transformations commutes provided each arrow of which is obtained by repeatedly applying the functor $\otimes$ to instances of $\alpha$, $\gamma$, $\rho$, their inverses and $1$ (Cf.\ \cite[\S VII]{SM:cat} for a precise formulation);\\

the two equations\\

$m_{\mathcal{A}(\bar{\bar{a}},\bar{\bar{\bar{a}}})\otimes\mathcal{A}(\bar{a},\bar{\bar{a}}),\mathcal{A}(a,\bar{a}),
\mathcal{B}(\bar{\bar{b}},\bar{\bar{\bar{b}}})\otimes\mathcal{B}(\bar{b},\bar{\bar{b}}),\mathcal{B}(b,\bar{b})}\circ m_{\mathcal{A}(\bar{\bar{a}},\bar{\bar{\bar{a}}})\otimes\mathcal{A}(\bar{a},\bar{\bar{a}}),
\mathcal{B}(\bar{\bar{b}},\bar{\bar{\bar{b}}})\otimes\mathcal{B}(\bar{b},\bar{\bar{b}}),\mathcal{A}(a,\bar{a}),\mathcal{B}(b,\bar{b})}\circ (m_{\mathcal{A}(\bar{\bar{a}},\bar{\bar{\bar{a}}}),\mathcal{B}(\bar{\bar{b}},\bar{\bar{\bar{b}}}),
\mathcal{A}(\bar{a},\bar{\bar{\bar{a}}}),\mathcal{B}(\bar{b},\bar{\bar{b}})}\otimes 1_{\mathcal{A}(a,\bar{a})\otimes\mathcal{B}(b,\bar{b})})=
m_{\mathcal{A}(\bar{\bar{a}},\bar{\bar{\bar{a}}}),\mathcal{B}(\bar{\bar{b}},
\bar{\bar{\bar{b}}}),\mathcal{A}(\bar{a},\bar{\bar{a}}),
\mathcal{B}(\bar{b},\bar{\bar{b}})}\otimes 1_{\mathcal{A}(a,\bar{a})\otimes\mathcal{B}(b,\bar{b})}$\\

and\\

$m^{-1}_{\mathcal{A}(\bar{a},\bar{\bar{\bar{a}}}),\mathcal{A}(a,\bar{a}),\mathcal{B}(\bar{b},\bar{\bar{\bar{b}}}),
\mathcal{B}(b,\bar{b})}\circ ((M^\mathcal{A}_{\bar{a},\bar{\bar{a}},\bar{\bar{\bar{a}}}}\otimes M^\mathcal{B}_{\bar{b},\bar{\bar{b}},\bar{\bar{\bar{b}}}})\otimes(1_{\mathcal{A}(a,\bar{a})}\otimes 1_{\mathcal{B}(b,\bar{b})}))\\
\circ m_{\mathcal{A}(\bar{\bar{a}},\bar{\bar{\bar{a}}})\otimes\mathcal{A}(\bar{a},\bar{\bar{a}}),\mathcal{A}(a,\bar{a}),
\mathcal{B}(\bar{\bar{b}},\bar{\bar{\bar{b}}})\otimes\mathcal{B}(\bar{b},\bar{\bar{b}}),\mathcal{B}(b,\bar{b})}
=(M^\mathcal{A}_{\bar{a},\bar{\bar{a}},\bar{\bar{\bar{a}}}}\otimes 1_{\mathcal{A}(a,\bar{a})})\otimes(M^\mathcal{B}_{\bar{b},\bar{\bar{b}},\bar{\bar{\bar{b}}}}\otimes 1_{\mathcal{B}(b,\bar{b})}) $\\

both hold (the first one for the same reason that equation \textbf{(iii)} holds; the second one holds since $m$ is natural); hence, the upwards extended pentagon (corresponding to equations \textbf{(ii)} and \textbf{(iii)}) can now be extended to the left; a similar process can be applied to obtain a right extension, proving finally that the composition law \textbf{(i)} holds for $\mathcal{A}\bigcirc \mathcal{B}$.\\

(1.2) Let's check now the unit axioms: are the following two equations true? (Cf.\ section \ref{sec-Vcats})\\

(1.2.1) $M^{\mathcal{A}\bigcirc\mathcal{B}}_{(a,b),(a,b),(\bar{a},\bar{b})}\circ (1_{\mathcal{A}\bigcirc\mathcal{B}((a,b),(\bar{a},\bar{b}))}\otimes j^{\mathcal{A}\bigcirc\mathcal{B}}_{(a,b)})=\rho_{\mathcal{A}\bigcirc\mathcal{B}((a,b),(\bar{a},\bar{b}))}$,\\

(1.2.2) $M^{\mathcal{A}\bigcirc\mathcal{B}}_{(a,b),(\bar{a},\bar{b}),(\bar{a},\bar{b})}
\circ (j^{\mathcal{A}\bigcirc\mathcal{B}}_{(\bar{a},\bar{b})}\otimes 1_{\mathcal{A}\bigcirc\mathcal{B}((a,b),(\bar{a},\bar{b}))})=
\rho_{\mathcal{A}\bigcirc\mathcal{B}((a,b),(\bar{a},\bar{b}))}\circ
\gamma_{E,\mathcal{A}\bigcirc\mathcal{B}((a,b),(\bar{a},\bar{b}))}$;\\

(1.2.1) as $\mathcal{A}$ and $\mathcal{B}$ are $\mathbb{C}$-categories, then the two equalities $M^\mathcal{A}_{a,a,\bar{a}}\circ (1_{\mathcal{A}(a,\bar{a})}\otimes j^\mathcal{A}_a)=\rho_{\mathcal{A}(a,\bar{a})}$ and $M^\mathcal{B}_{b,b,\bar{b}}\circ (1_{\mathcal{B}(b,\bar{b})}\otimes j^\mathcal{B}_b)=\rho_{\mathcal{B}(b,\bar{b})}$ hold, which imply by tensoring,\\

\textbf{(iv)} $(M^\mathcal{A}_{a,a,\bar{a}}\otimes M^\mathcal{B}_{b,b,\bar{b}})\circ ((1_{\mathcal{A}(a,\bar{a})}\otimes j^\mathcal{A}_a)\otimes (1_{\mathcal{B}(b,\bar{b})}\otimes j^\mathcal{B}_b))=\rho_{\mathcal{A}(a,\bar{a})}\otimes\rho_{\mathcal{B}(b,\bar{b})}$;\\

the triangular diagram corresponding to equation \textbf{(iv)} can be extended to the left with the square corresponding to the equation\\

\textbf{(v)} $((1_{\mathcal{A}(a,\bar{a})}\otimes j^\mathcal{A}_a)\otimes (1_{\mathcal{B}(b,\bar{b})}\otimes j^\mathcal{B}_b))\circ m_{\mathcal{A}(a,\bar{a}),\mathcal{B}(b,\bar{b}),E,E}=m_{\mathcal{A}(a,\bar{a}),\mathcal{B}(b,\bar{b}),\mathcal{A}(a,a),\mathcal{B}(b,b)}
\circ ((1_{\mathcal{A}(a,\bar{a})}\otimes 1_{\mathcal{B}(b,\bar{b})})\otimes (j^\mathcal{A}_a\otimes j^\mathcal{B}_b))$,\\

which holds since $m$ is natural;\\

the equation\\

\textbf{(vi)} $(\rho_{\mathcal{A}(a,\bar{a})}\otimes\rho_{\mathcal{B}(b,\bar{b})})\circ m_{\mathcal{A}(a,\bar{a}),\mathcal{B}(b,\bar{b}),E,E}
\circ(1_{\mathcal{A}(a,\bar{a})\otimes\mathcal{B}(b,\bar{b})}\otimes\rho_E))=
\rho_{\mathcal{A}(a,\bar{a})\otimes\mathcal{B}(b,\bar{b})}$\\

holds, since every diagram of natural transformations commutes, provided each arrow of which is obtained by repeatedly applying the functor $\otimes$ to instances of $\alpha$, $\gamma$, $\rho$, their inverses and $1$ (Cf.\ \cite[\S VII]{SM:cat} for a precise formulation);\\

the diagram corresponding to \textbf{(vi)} extends downwards the diagram corresponding to the extension of \textbf{(iv)} by \textbf{(v)}, giving the equation\\

$(M^\mathcal{A}_{a,a,\bar{a}}\otimes M^\mathcal{B}_{b,b,\bar{b}})\circ m_{\mathcal{A}(a,\bar{a}),\mathcal{B}(b,\bar{b}),\mathcal{A}(a,a),\mathcal{B}(b,b)}
\circ(1_{\mathcal{A}(a,\bar{a})\otimes\mathcal{B}(b,\bar{b})}\otimes (j^\mathcal{A}_a\otimes j^\mathcal{B}_b))
\circ(1_{\mathcal{A}(a,\bar{a})\otimes\mathcal{B}(b,\bar{b})}\otimes\rho_E))=
\rho_{\mathcal{A}(a,\bar{a})\otimes\mathcal{B}(b,\bar{b})}$\\

 which is exactly equation (1.2.1);\\

equation (1.2.2) also holds, whose proof can be obtained by mimicking the proof of (1.2.1) just above.\\

\noindent (2) Is $T\bigcirc S:\mathcal{A}\bigcirc \mathcal{B}\rightarrow\mathcal{A}'\bigcirc \mathcal{B}'$ a $\mathbb{C}$-functor?\\

(2.1) $T\bigcirc S_{(a,b),(\bar{\bar{a}},\bar{\bar{b}})}\circ M^{\mathcal{A}\bigcirc \mathcal{B}}_{(a,b),(\bar{a},\bar{b}),(\bar{\bar{a}},\bar{\bar{b}})}=
M^{\mathcal{A}'\bigcirc\mathcal{B}'}_{(T(a),S(b)),(T(\bar{a}),S(\bar{b})),(T(\bar{\bar{a}}),S(\bar{\bar{b}}))}\circ (T\bigcirc S_{(\bar{a},\bar{b}),(\bar{\bar{a}},\bar{\bar{b}})}\otimes T\bigcirc S_{(a,b),(\bar{a},\bar{b})})$, for every $a,\bar{a},\bar{\bar{a}}\in\mathcal{A}$ and $b,\bar{b},\bar{\bar{b}}\in\mathcal{B}$:\\

as $T$ and $S$ are $\mathbb{C}$-functors, the following two equations hold,\\

$$T_{a,\bar{\bar{a}}}\circ M^\mathcal{A}_{a,\bar{a},\bar{\bar{a}}}=
M^{\mathcal{A}'}_{T(a),T(\bar{a}),T(\bar{\bar{a}})}\circ (T_{\bar{a},\bar{\bar{a}}}\otimes T_{a,\bar{a}})$$ and
$$S_{b,\bar{\bar{b}}}\circ M^\mathcal{B}_{b,\bar{b},\bar{\bar{b}}}=
M^{\mathcal{B}'}_{S(b),S(\bar{b}),S(\bar{\bar{b}})}\circ (S_{\bar{b},\bar{\bar{b}}}\otimes S_{b,\bar{b}});$$

then, by tensoring,\\

$(T_{a,\bar{\bar{a}}}\otimes S_{b,\bar{\bar{b}}})\circ (M^\mathcal{A}_{a,\bar{a},\bar{\bar{a}}}\otimes M^\mathcal{B}_{b,\bar{b},\bar{\bar{b}}})=
(M^{\mathcal{A}'}_{T(a),T(\bar{a}),T(\bar{\bar{a}})}\otimes M^{\mathcal{B}'}_{S(b),S(\bar{b}),S(\bar{\bar{b}})})
\circ ((T_{\bar{a},\bar{\bar{a}}}\otimes T_{a,\bar{a}})\otimes (S_{\bar{b},\bar{\bar{b}}}\otimes S_{b,\bar{b}}))$,\\

which, together with the following equation,\\

$((T_{\bar{a},\bar{\bar{a}}}\otimes T_{a,\bar{a}})\otimes (S_{\bar{b},\bar{\bar{b}}}\otimes S_{b,\bar{b}}))\circ m_{\mathcal{A}(\bar{a},\bar{\bar{a}}),\mathcal{B}(\bar{b},\bar{\bar{b}}),\mathcal{A}(a,\bar{a}),\mathcal{B}(b,\bar{b})}
=\\
m_{\mathcal{A}'(T(\bar{a}),T(\bar{\bar{a}})),\mathcal{B}'(S(\bar{b}),S(\bar{\bar{b}})),\mathcal{A}'(T(a),T(\bar{a})),\mathcal{B}'(S(b),S(\bar{b}))}
\circ ((T_{\bar{a},\bar{\bar{a}}}\otimes S_{\bar{b},\bar{\bar{b}}})\otimes (T_{a,\bar{a}}\otimes S_{b,\bar{b}}))$\\

arising from the naturality of $m$, gives the equation\\

$(T_{a,\bar{\bar{a}}}\otimes S_{b,\bar{\bar{b}}})\circ (M^\mathcal{A}_{a,\bar{a},\bar{\bar{a}}}\otimes M^\mathcal{B}_{b,\bar{b},\bar{\bar{b}}})\circ m_{\mathcal{A}(\bar{a},\bar{\bar{a}}),\mathcal{B}(\bar{b},\bar{\bar{b}}),\mathcal{A}(a,\bar{a}),\mathcal{B}(b,\bar{b})}
=(M^{\mathcal{A}'}_{T(a),T(\bar{a}),T(\bar{\bar{a}})}\otimes M^{\mathcal{B}'}_{S(b),S(\bar{b}),S(\bar{\bar{b}})} )\circ m_{\mathcal{A}'(T(\bar{a}),T(\bar{\bar{a}})),\mathcal{B}'(S(\bar{b}),S(\bar{\bar{b}})),\mathcal{A}'(T(a),T(\bar{a})),\mathcal{B}'(S(b),S(\bar{b}))}\circ ((T_{\bar{a},\bar{\bar{a}}}\otimes S_{\bar{b},\bar{\bar{b}}})\otimes (T_{a,\bar{a}}\otimes S_{b,\bar{b}}))$\\

which is exactly the compatibility with composition (2.1), according to the definitions given above.\\

(2.2) $j^{\mathcal{A}'\bigcirc \mathcal{B}'}_{(T(a),S(b))}=T\bigcirc S_{(a,b),(a,b)}\circ j^{\mathcal{A}\bigcirc \mathcal{B}}_{(a,b)}$, for every $a\in\mathcal{A}$ and $b\in\mathcal{B}$:\\

as $\mathcal{A}$ and $\mathcal{B}$ are $\mathbb{C}$-categories, the following two equations hold,

$$j^{\mathcal{A}'}_{T(a)}=T_{a,a}\circ j^{\mathcal{A}}_a,$$ $$j^{\mathcal{B}'}_{S(b)}=S_{b,b}\circ j^{\mathcal{B}}_b,$$ and so $$j^{\mathcal{A}'}_{T(a)}\otimes j^{\mathcal{B}'}_{S(b)}=(T_{a,a}\otimes S_{b,b})\circ (j^{\mathcal{A}}_a\otimes j^{\mathcal{B}}_b);$$

composing with the isomorphism $\rho_E:E\rightarrow E\otimes E$, we get $$(j^{\mathcal{A}'}_{T(a)}\otimes j^{\mathcal{B}'}_{S(b)})\circ\rho_E=(T_{a,a}\otimes S_{b,b})\circ (j^{\mathcal{A}}_a\otimes j^{\mathcal{B}}_b)\circ\rho_E$$ which is the compatibility with identities (2.2), according to the definitions given above.\\

\noindent (3) Is $\bigcirc$ a (bi)functor?\\

(3.1) ($\bigcirc 1_{(\mathcal{A},\mathcal{B})}=\bigcirc (1_\mathcal{A},1_\mathcal{B})=$)$1_\mathcal{A}\bigcirc 1_\mathcal{B}=1_{\mathcal{A}\bigcirc\mathcal{B}}$, the image of any identity is an identity (cf.\ the characterization of identity morphisms in $\mathbb{C}$-$Cat$, given in section \ref{sec-base monoidal adjunction} just before Proposition \ref{proposition:IsosInC-Cat}):\\

$ob1_\mathcal{A}\bigcirc 1_\mathcal{B}=ob1_\mathcal{A}\times ob1_\mathcal{B}=1_{ob(\mathcal{A})}\times 1_{ob(\mathcal{B})}$, by definiton of $1_\mathcal{A}$ and $1_\mathcal{B}$

$=1_{ob(\mathcal{A})\times ob(\mathcal{B})}$, since $\times$ is a bifunctor for $Set$

$=1_{ob(\mathcal{A}\bigcirc\mathcal{B})}=ob 1_{\mathcal{A}\bigcirc\mathcal{B}}:ob(\mathcal{A})\times ob(\mathcal{B})\rightarrow ob(\mathcal{A})\times ob(\mathcal{B})$;\\

$(1_\mathcal{A}\bigcirc 1_\mathcal{B})_{(a,b),(\bar{a},\bar{b})}=(1_\mathcal{A})_{(a,\bar{a})}\otimes (1_\mathcal{B})_{(b,\bar{b})}=1_{\mathcal{A}(a,\bar{a})}\otimes 1_{\mathcal{B}(b,\bar{b})}$, by definition of $1_\mathcal{A}$ and $1_\mathcal{B}$

$=1_{\mathcal{A}(a,\bar{a})\otimes\mathcal{B}(b,\bar{b})}:\mathcal{A}(a,\bar{a})\otimes\mathcal{B}(b,\bar{b})\rightarrow\mathcal{A}(a,\bar{a})\otimes\mathcal{B}(b,\bar{b})$, because $\otimes$ is a bifunctor for $\mathbb{C}$.\\

(3.2) $(T'\bigcirc S')\circ(T\bigcirc S)=(T'\circ T)\bigcirc(S'\circ S)$:\footnote{Cf.\ the footnotes in section \ref{sec:derived}.}\\

$ob((T'\bigcirc S')\circ(T\bigcirc S))=ob(T'\bigcirc S')\circ ob(T\bigcirc S)$

$=(obT'\times obS')\circ (obT\times obS)=(obT'\circ obT)\times (obS'\circ obS)$, because $\times$ is a bifunctor for $Set$

$=ob(T'\circ T)\times ob(S'\circ S)=ob((T'\circ T)\bigcirc (S'\circ S))$;\\

$((T'\bigcirc S')\circ(T\bigcirc S))_{(a,b),(\bar{a},\bar{b})}=$

$=(T'\bigcirc S')_{(T(a),S(b)),(T(\bar{a}),S(\bar{b}))}\circ (T\bigcirc S)_{(a,b),(\bar{a},\bar{b})}$

$=(T'_{T(a),T(\bar{a})}\otimes S'_{S(b),S(\bar{b})})\circ (T_{a,\bar{a}}\otimes S_{b,\bar{b}})$

$=(T'_{T(a),T(\bar{a})}\circ T_{a,\bar{a}})\otimes (S'_{S(b),S(\bar{b})}\circ S_{b,\bar{b}})$, because $\otimes$ is a bifunctor for $\mathbb{C}$

$=(T'\circ T)_{a,\bar{a}}\otimes (S'\circ S)_{b,\bar{b}}=((T'\circ T)\bigcirc (S'\circ S))_{(a,b),(\bar{a},\bar{b})}$.\\

\noindent Conclusion: now, at the end of subsection \ref{subsection:A bifunctor}, we can state that $\bigcirc$ is indeed a (bi)functor.

\subsection{A natural isomorphism $\wedge$ for $\mathbb{C}$-$Cat$}\label{subsection:isowedge C-Cat}

We define $$\wedge_{\mathcal{A},\mathcal{B},\mathcal{C}}:(\mathcal{A}\bigcirc\mathcal{B})\bigcirc\mathcal{C}
\rightarrow\mathcal{A}\bigcirc(\mathcal{B}\bigcirc\mathcal{C})$$ as
$$ob\wedge_{\mathcal{A},\mathcal{B},\mathcal{C}}:(ob(\mathcal{A})\times ob(\mathcal{B}))\times ob(\mathcal{C})
\rightarrow ob(\mathcal{A})\times (ob(\mathcal{B})\times ob(\mathcal{C}))$$ the canonical isomorphism in $Set$, and\\

\noindent $(\wedge_{\mathcal{A},\mathcal{B},\mathcal{C}})_{((a,b),c),((\bar{a},\bar{b}),\bar{c})}
=\alpha_{\mathcal{A}(a,\bar{a}),\mathcal{B}(b,\bar{b}),\mathcal{C}(c,\bar{c})}:
(\mathcal{A}(a,\bar{a})\otimes\mathcal{B}(b,\bar{b}))\otimes\mathcal{C}(c,\bar{c})
\rightarrow\mathcal{A}(a,\bar{a})\otimes(\mathcal{B}(b,\bar{b})\otimes\mathcal{C}(c,\bar{c}))$,\\

\noindent for every $\mathcal{A},\mathcal{B},\mathcal{C}\in\mathbb{C}$-$Cat$ and every three pairs $(a,\bar{a})\in ob(\mathcal{A})\times ob(\mathcal{A})$, $(b,\bar{b})\in ob(\mathcal{B})\times ob(\mathcal{B})$ and $(c,\bar{c})\in ob(\mathcal{C})\times ob(\mathcal{C})$.\\

Is $\wedge$ natural? That is, for every triple of $\mathbb{C}$-functors $T:\mathcal{A}\rightarrow\mathcal{A}'$, $S:\mathcal{B}\rightarrow\mathcal{B}'$, $R:\mathcal{C}\rightarrow\mathcal{C}'$, does the diagram corresponding to the following equation commute? $$(T\bigcirc (S\bigcirc R))\circ\wedge_{\mathcal{A},\mathcal{B},\mathcal{C}}=\wedge_{\mathcal{A}',\mathcal{B}',\mathcal{C}'}\circ ((T\bigcirc S)\bigcirc R)).$$

The image of this equation by the functor $ob:\mathbb{C}$-$Cat\rightarrow Set$ is obviously true, since

\begin{picture}(200,60)
\put(20,35){$((a,b),c)$}
\put(70,37){\vector (1,0){45}}\put(120,35){$(a,(b,c))$}
\put(170,37){\vector (1,0){45}}\put(220,35){$(T(a),(S(b),R(c)))$}

\put(50,30){\vector (3,-1){50}}\put(170,13){\vector (3,1){50}}

\put(100,0){$((T(a),S(b)),R(c))$\hspace{20pt} .}
\end{picture}\\

As $\alpha$ is natural by assumption, the following equation holds,\\

$(T_{a,\bar{a}}\otimes (S_{b,\bar{b}}\otimes R_{c,\bar{c}}))\circ\alpha_{\mathcal{A}(a,\bar{a}),\mathcal{B}(b,\bar{b}),\mathcal{C}(c,\bar{c})}=
\alpha_{\mathcal{A}'(T(a),T(\bar{a})),\mathcal{B}'(S(b),S(\bar{b})),\mathcal{C}'(R(c),R(\bar{c}))}
\circ ((T_{a,\bar{a}}\otimes S_{b,\bar{b}})\otimes R_{c,\bar{c}})$,\\

for every three pairs $(a,\bar{a})\in ob(\mathcal{A})\times ob(\mathcal{A})$, $(b,\bar{b})\in ob(\mathcal{B})\times ob(\mathcal{B})$ and $(c,\bar{c})\in ob(\mathcal{C})\times ob(\mathcal{C})$; hence, $\wedge$ is natural.\\

The pentagon coherence axiom corresponds to the equation\\

$\wedge_{\mathcal{A},\mathcal{B},\mathcal{C}\bigcirc\mathcal{D}}\circ
\wedge_{\mathcal{A}\bigcirc\mathcal{B},\mathcal{C},\mathcal{D}}=
(1_\mathcal{A}\bigcirc\wedge_{\mathcal{B},\mathcal{C},\mathcal{D}})\circ
\wedge_{\mathcal{A},\mathcal{B}\bigcirc\mathcal{C},\mathcal{D}}\circ (\wedge_{\mathcal{A},\mathcal{B},\mathcal{C}}\bigcirc 1_\mathcal{D})$,\\

whose image by the functor $ob:\mathbb{C}$-$Cat\rightarrow Set$ is obviously true, since

\begin{picture}(200,40)
\put(0,15){$(((a,b),c),d)$}
\put(70,17){\vector (1,0){45}}\put(120,15){$((a,b),(c,d))$}
\put(190,17){\vector (1,0){45}}\put(240,15){$(a,(b,(c,d))$}

\put(35,10){\vector (0,-1){20}}\put(275,-10){\vector (0,1){20}}

\put(0,-20){$((a,(b,c)),d)$}\put(70,-18){\vector (1,0){165}}\put(240,-20){$(a,((b,c),d))$;}
\end{picture}\\\\

and, as $\alpha$ satisfies the coherence axioms in $\mathbb{C}$, it follows that the following equation holds,\\

$\alpha_{\mathcal{A}(a,\bar{a}),\mathcal{B}(b,\bar{b}),\mathcal{C}(c,\bar{c})\otimes\mathcal{D}(d,\bar{d})}
\circ\alpha_{\mathcal{A}(a,\bar{a})\otimes\mathcal{B}(b,\bar{b}),\mathcal{C}(c,\bar{c}),\mathcal{D}(d,\bar{d})}
=(1_{\mathcal{A}(a,\bar{a})}\otimes\alpha_{\mathcal{B}(b,\bar{b}),\mathcal{C}(c,\bar{c}),\mathcal{D}(d,\bar{d})})
\circ\alpha_{\mathcal{A}(a,\bar{a}),\mathcal{B}(b,\bar{b})\otimes\mathcal{C}(c,\bar{c}),\mathcal{D}(d,\bar{d})}
\circ (\alpha_{\mathcal{A}(a,\bar{a}),\mathcal{B}(b,\bar{b}),\mathcal{C}(c,\bar{c})}\otimes 1_{\mathcal{D}(d,\bar{d})})$,\\

for every four pairs $(a,\bar{a})\in ob(\mathcal{A})\times ob(\mathcal{A})$, $(b,\bar{b})\in ob(\mathcal{B})\times ob(\mathcal{B})$, $(c,\bar{c})\in ob(\mathcal{C})\times ob(\mathcal{C})$ and $(d,\bar{d})\in ob(\mathcal{D})\times ob(\mathcal{D})$; therefore, the pentagon coherence axiom holds for $\wedge$.

\subsection{A symmetry $\Gamma$ for $\mathbb{C}$-$Cat$}\label{subsection:isoGamma C-Cat}

We define $$\Gamma_{\mathcal{A},\mathcal{B}}:\mathcal{A}\bigcirc\mathcal{B}\rightarrow\mathcal{B}\bigcirc\mathcal{A}$$
as $$ob\Gamma_{\mathcal{A},\mathcal{B}}:ob(\mathcal{A})\times ob(\mathcal{B})\rightarrow ob(\mathcal{B})\times ob(\mathcal{A}),$$ the canonical isomorphism in $Set$, and $$(\Gamma_{\mathcal{A},\mathcal{B}})_{(a,b),(\bar{a},\bar{b})}=\gamma_{\mathcal{A}(a,\bar{a}),\mathcal{B}(b,\bar{b})}:
\mathcal{A}(a,\bar{a})\otimes\mathcal{B}(b,\bar{b})\rightarrow\mathcal{B}(b,\bar{b})\otimes\mathcal{A}(a,\bar{a}),$$
for every pair $\mathcal{A},\mathcal{B}$ of $\mathbb{C}$-categories and every $a,\bar{a}\in ob(\mathcal{A})$ and every $b,\bar{b}\in ob(\mathcal{B})$.\\

Is $\Gamma$ natural? That is, for every pair of $\mathbb{C}$-functors $T:\mathcal{A}\rightarrow\mathcal{A}'$, $S:\mathcal{B}\rightarrow\mathcal{B}'$, does the diagram corresponding to the following equation commute?
$$(S\bigcirc T)\circ\Gamma_{\mathcal{A},\mathcal{B}}=
\Gamma_{\mathcal{A}',\mathcal{B}'}\circ (T\bigcirc S).$$

The image of this equation by the functor $ob:\mathbb{C}$-$Cat\rightarrow Set$ is obviously true, since

\begin{picture}(200,40)
\put(20,15){$(a,b)$}
\put(55,17){\vector (1,0){60}}\put(120,15){$(b,a)$}
\put(155,17){\vector (1,0){60}}\put(220,15){$(S(b),T(a))$}

\put(50,10){\vector (3,-1){50}}\put(170,-7){\vector (3,1){50}}

\put(105,-20){$(T(a),S(b))$\hspace{20pt} .}
\end{picture}\\\\

As $\gamma$ is natural by assumption, the following equation holds,\\

$(S_{b,\bar{b}}\otimes T_{a,\bar{a}})\circ\gamma_{\mathcal{A}(a,\bar{a}),\mathcal{B}(b,\bar{b})}=
\gamma_{\mathcal{A}'(T(a),T(\bar{a})),\mathcal{B}'(S(b),S(\bar{b}))}
\circ (T_{a,\bar{a}}\otimes S_{b,\bar{b}})$,\\

for every two pairs $(a,\bar{a})\in ob(\mathcal{A})\times ob(\mathcal{A})$, $(b,\bar{b})\in ob(\mathcal{B})\times ob(\mathcal{B})$; hence, $\Gamma$ is natural.\\

We will check now the coherence axioms corresponding to the two equations
$$(1)\ \Gamma_{\mathcal{B},\mathcal{A}}\circ\Gamma_{\mathcal{A},\mathcal{B}}=1_{\mathcal{A}\bigcirc\mathcal{B}},$$ and
$$(2) \wedge_{\mathcal{B},\mathcal{C},\mathcal{A}}\circ\Gamma_{\mathcal{A},\mathcal{B}\bigcirc\mathcal{C}}
\circ\wedge_{\mathcal{A},\mathcal{B},\mathcal{C}}=(1_\mathcal{B}\bigcirc\Gamma_{\mathcal{A},\mathcal{C}})
\circ\wedge_{\mathcal{B},\mathcal{A},\mathcal{C}}\circ (\Gamma_{\mathcal{A},\mathcal{B}}\bigcirc 1_\mathcal{C}),$$
whose images by the functor $ob:\mathbb{C}$-$Cat\rightarrow Set$ are obviously true, since
$$(1)\ (a,b)\mapsto (b,a)\mapsto (a,b),$$ and

\begin{picture}(270,40)
\put(20,15){$(2)\ ((a,b),c)$}
\put(90,17){\vector (1,0){15}}\put(110,15){$(a,(b,c))$}\put(160,17){\vector (1,0){15}}\put(180,15){$((b,c),a)$}
\put(230,17){\vector (1,0){15}}\put(250,15){$(b,(c,a))$}

\put(60,10){\vector (2,-1){30}}\put(80,-20){$((b,a),c)$}
\put(130,-15){\vector (1,0){40}}\put(175,-20){$(b,(a,c))$\hspace{20pt} ;}\put(220,-7){\vector (2,1){30}}

\end{picture}\\\\

\noindent as $\gamma$ satisfies the coherence axioms in $\mathbb{C}$, the following two equations hold,\\

$(1)\ \gamma_{\mathcal{B}(b,\bar{b}),\mathcal{A}(a,\bar{a})}\circ\gamma_{\mathcal{A}(a,\bar{a}),\mathcal{B}(b,\bar{b})}
=1_{\mathcal{A}(a,\bar{a})\otimes\mathcal{B}(b,\bar{b})}$, and\\

$(2)\ \alpha_{\mathcal{B}(b,\bar{b}),\mathcal{C}(c,\bar{c}),\mathcal{A}(a,\bar{a})}\circ
\gamma_{\mathcal{A}(a,\bar{a}),\mathcal{B}(b,\bar{b})\otimes\mathcal{C}(c,\bar{c})}
\circ\alpha_{\mathcal{A}(a,\bar{a}),\mathcal{B}(b,\bar{b}),\mathcal{C}(c,\bar{c})}
=\\
(1_{\mathcal{B}(b,\bar{b})}\otimes\gamma_{\mathcal{A}(a,\bar{a}),\mathcal{C}(c,\bar{c})})
\circ\alpha_{\mathcal{B}(b,\bar{b}),\mathcal{A}(a,\bar{a}),\mathcal{C}(c,\bar{c})}
\circ (\gamma_{\mathcal{A}(a,\bar{a}),\mathcal{B}(b,\bar{b})}\otimes 1_{\mathcal{C}(c,\bar{c})})$,\\

\noindent for every pairs $(a,\bar{a})\in ob(\mathcal{A})\times ob(\mathcal{A})$, $(b,\bar{b})\in ob(\mathcal{B})\times ob(\mathcal{B})$ and $(c,\bar{c})\in ob(\mathcal{C})\times ob(\mathcal{C})$; therefore, the two coherence axioms $(1)$ and $(2)$ hold.

\subsection{An object $\mathfrak{E}\in\mathbb{C}$-$Cat$ and a natural isomorphism $\mho$ for $\mathbb{C}$-$Cat$}
\label{subsection:matcalE e mho}

In this last subsection \ref{subsection:matcalE e mho}, we will introduce the remaining (symmetric) monoidal structure for $\mathbb{C}$-$Cat$, and we will prove the remaining coherence axiom.\\

Consider $\mathfrak{E}\in\mathbb{C}$-$Cat$ such that:
\begin{itemize}
\item $ob(\mathfrak{E})=\{\ast\}$ is a singleton set;
\item $\mathfrak{E}(\ast,\ast)=E$ is the object in the base monoidal structure;
\item $M^\mathfrak{E}_{\ast,\ast,\ast}=\rho_E:E\otimes E\rightarrow E$;
\item $j^\mathfrak{E}_\ast=1_E:E\rightarrow E$.
\end{itemize}

The associativity and unit axioms hold for $\mathfrak{E}$ because their diagrams are instances of the class of diagrams known to commute (cf.\ \cite[\S VII]{SM:cat}).\\

Define
$$\mho_\mathcal{A}:\mathcal{A}\bigcirc \mathfrak{E}\rightarrow\mathcal{A},$$
as $ob\mho_\mathcal{A}:ob(\mathcal{A})\times ob(\mathfrak{E})\rightarrow ob(\mathcal{A})$, $(a,\ast)\mapsto a$, the projection which is an isomorphism in $Set$, and

$$(\mho_\mathcal{A})_{(a,\ast ),(\bar{a},\ast )}
=\rho_{\mathcal{A}(a,\bar{a})}:\mathcal{A}(a,\bar{a})\otimes E\rightarrow\mathcal{A}(a,\bar{a}),$$

for every $\mathbb{C}$-category $\mathcal{A}$ and $a,\bar{a}\in ob(\mathcal{A})$.\\

Is $\mho$ a natural isomorphism? That is, does the diagram corresponding to the following equation commute, for every $\mathbb{C}$-functor $T:\mathcal{A}\rightarrow \mathcal{B}$?
$$T\circ\mho_\mathcal{A}=\mho_\mathcal{B}\circ(T\bigcirc 1_\mathfrak{E}).$$

The image of this equation by the functor $ob:\mathbb{C}$-$Cat\rightarrow Set$ is true, since

\begin{picture}(200,40)
\put(20,15){$(a,\ast)$}
\put(55,17){\vector (1,0){60}}\put(125,15){$a$}
\put(140,17){\vector (1,0){60}}\put(205,15){$T(a)$}

\put(50,10){\vector (3,-1){50}}\put(155,-7){\vector (3,1){50}}

\put(105,-20){$(T(a),\ast)$\hspace{20pt} .}
\end{picture}\\\\

As $\rho$ is natural by assumption, the following equation holds,
$$T_{a,\bar{a}}\circ\rho_{\mathcal{A}(a,\bar{a})}=\rho_{\mathcal{B}(b,\bar{b})}\circ (T_{a,\bar{a}}\otimes 1_E),$$
for every pair $(a,\bar{a})\in ob(\mathcal{A})\times ob(\mathcal{A})$; hence, $\mho$ is natural.\\

We will check now the coherence axiom corresponding to the following equation, for every $\mathcal{A},\mathcal{B}\in\mathbb{C}$-$Cat$,\\

$\mho_\mathcal{A}\bigcirc 1_\mathcal{B}=(1_\mathcal{A}\bigcirc (\mho_\mathcal{B}\circ\Gamma_{\mathfrak{E},\mathcal{B}}))\circ\wedge_{\mathcal{A},\mathfrak{E},\mathcal{B}}:
(\mathcal{A}\bigcirc\mathfrak{E})\bigcirc\mathcal{B}\rightarrow
\mathcal{A}\bigcirc(\mathfrak{E}\bigcirc\mathcal{B})\rightarrow\mathcal{A}\bigcirc(\mathcal{B}\bigcirc\mathfrak{E})
\rightarrow\mathcal{A}\bigcirc\mathcal{B}$\\

\noindent whose image by the functor $ob:\mathbb{C}$-$Cat\rightarrow Set$ is obviously true, since $((a,\ast),b)\mapsto (a,(\ast,b))\mapsto (a,(b,\ast))\mapsto (a,b)$;\\

as $\rho$ satisfies the coherence axioms in $\mathbb{C}$, then the following equation holds\\

$\rho_{\mathcal{A}(a,\bar{a})}\otimes 1_{\mathcal{B}(b,\bar{b})}=
(1_{\mathcal{\mathcal{A}}(a,\bar{a})}\otimes(\rho_{\mathcal{B}(b,\bar{b})}
\circ\gamma_{E,\mathcal{B}(b,\bar{b})}))\circ\alpha_{\mathcal{A}(a,\bar{a}),E,B(b,\bar{b})}:\\
(\mathcal{A}(a,\bar{a})\otimes E)\otimes\mathcal{B}(b,\bar{b})\rightarrow\mathcal{A}(a,\bar{a})\otimes(E\otimes\mathcal{B}(b,\bar{b}))
\rightarrow\mathcal{A}(a,\bar{a})\otimes\mathcal{B}(b,\bar{b})$\\

\noindent for every two pairs $(a,\bar{a})\in ob(\mathcal{A})\times ob(\mathcal{A})$, $(b,\bar{b})\in ob(\mathcal{B})\times ob(\mathcal{B})$; therefore, the coherence axiom for $\mho$ holds.

\section{The derived adjunction is (symmetric) monoidal}
\label{section:derived adjunction monoidal}

In section \ref{section:derivedMonoidal} it was shown that

\begin{center}
$(1)\ \mathbb{C}$-$Cat=(\mathbb{C}$-$Cat,\bigcirc,\mathfrak{E},\wedge,\Gamma,\mho)$
\end{center}
is a symmetric monoidal category, obtained from the symmetric monoidal category $\mathbb{C}=(\mathbb{C},\otimes,E,\alpha,\gamma,\rho)$.

Analogously,
\begin{center}
$(2)\ \mathbb{X}$-$Cat=(\mathbb{X}$-$Cat,\nabla,\mathcal{I},\vee,\top,\Re)$\end{center} can be obtained from the symmetric monoidal category $\mathbb{X}=(\mathbb{X},\lozenge,I,\mathfrak{a},t,r)$, being also a symmetric monoidal category.\\

In this section \ref{section:derived adjunction monoidal}, it will be shown that the derived adjunction (cf.\ \ref{sec:derived})
\begin{center}$(\mathbb{F},\mathbb{G},\Theta,\Upsilon):\mathbb{C}$-$Cat\rightarrow\mathbb{X}$-$Cat$\end{center} is a (symmetric) monoidal adjunction, likewise the adjunction $(F,G,\eta,\varepsilon):\mathbb{C}\rightarrow\mathbb{X}$, but with respect to the derived monoidal categories $(1)$ and $(2)$ in the paragraph above.

\subsection{$\mathbb{F}$ preserves the (symmetric) monoidal structure}
\label{subsection:FpreservesMonoStruc}\ \\

\begin{itemize}
\item $\mathbb{F}(\mathcal{A}\bigcirc\mathcal{B})=\mathbb{F}(\mathcal{A})\nabla\mathbb{F}(\mathcal{B})$, for every $\mathcal{A},\mathcal{B}\in\mathbb{C}$-$Cat$, because:

    $ob\mathbb{F}(\mathcal{A}\bigcirc\mathcal{B})=ob(\mathcal{A}\bigcirc\mathcal{B})
    =ob(\mathcal{A})\times ob(\mathcal{B})=ob\mathbb{F}(\mathcal{A})\times ob\mathbb{F}(\mathcal{B})
    =ob(\mathbb{F}(\mathcal{A})\nabla\mathbb{F}(\mathcal{B}))$;

    $\mathbb{F}(\mathcal{A}\bigcirc\mathcal{B})((a,b),(\bar{a},\bar{b}))
    =F(\mathcal{A}\bigcirc\mathcal{B}((a,b),(\bar{a},\bar{b})))
    =F(\mathcal{A}(a,\bar{a})\otimes\mathcal{B}(b,\bar{b}))
    =F(\mathcal{A}(a,\bar{a}))\lozenge F(\mathcal{B}(b,\bar{b}))
    =\mathbb{F}(\mathcal{A})(a,\bar{a})\lozenge \mathbb{F}(\mathcal{B})(b,\bar{b})
    =(\mathbb{F}(\mathcal{A})\nabla\mathbb{F}(\mathcal{B}))((a,b),(\bar{a},\bar{b}))$,

    $M^{\mathbb{F}(\mathcal{A}\bigcirc \mathcal{B})}_{(a,b),(\bar{a},\bar{b}),(\bar{\bar{a}},\bar{\bar{b}})}=
    F(M^{\mathcal{A}\bigcirc \mathcal{B}}_{(a,b),(\bar{a},\bar{b}),(\bar{\bar{a}},\bar{\bar{b}})})\\
    = F((M^\mathcal{A}_{a,\bar{a},\bar{\bar{a}}}\otimes M^\mathcal{B}_{b,\bar{b},\bar{\bar{b}}})\circ m_{\mathcal{A}(\bar{a},\bar{\bar{a}}),\mathcal{B}(\bar{b},\bar{\bar{b}}),\mathcal{A}(a,\bar{a}),
    \mathcal{B}(b,\bar{b})})\\
    = F(M^\mathcal{A}_{a,\bar{a},\bar{\bar{a}}}\otimes M^\mathcal{B}_{b,\bar{b},\bar{\bar{b}}})\circ Fm_{\mathcal{A}(\bar{a},\bar{\bar{a}}),\mathcal{B}(\bar{b},\bar{\bar{b}}),\mathcal{A}(a,\bar{a}),\mathcal{B}(b,\bar{b})}\\
    =(F(M^\mathcal{A}_{a,\bar{a},\bar{\bar{a}}})\lozenge F(M^\mathcal{B}_{b,\bar{b},\bar{\bar{b}}}))\circ m_{F(\mathcal{A}(\bar{a},\bar{\bar{a}})),F(\mathcal{B}(\bar{b},\bar{\bar{b}})),
    F(\mathcal{A}(a,\bar{a})),F(\mathcal{B}(b,\bar{b}))}\\
    = (M^{\mathbb{F}(\mathcal{A})}_{a,\bar{a},\bar{\bar{a}}}\lozenge M^{\mathbb{F}(\mathcal{B})}_{b,\bar{b},\bar{\bar{b}}})\circ m_{\mathbb{F}(\mathcal{A})(\bar{a},\bar{\bar{a}}),\mathbb{F}(\mathcal{B})(\bar{b},\bar{\bar{b}}),
    \mathbb{F}(\mathcal{A})(a,\bar{a}),
    \mathbb{F}(\mathcal{B})(b,\bar{b})}\\
    =M^{\mathbb{F}(\mathcal{A})\nabla \mathbb{F}(\mathcal{B})}_{(a,b),(\bar{a},\bar{b}),(\bar{\bar{a}},\bar{\bar{b}})}$,

    $j^{\mathbb{F}(\mathcal{A}\bigcirc\mathcal{B})}_{(a,b)}=Fj^{\mathcal{A}\bigcirc\mathcal{B}}_{(a,b)}
    =F((j^\mathcal{A}_a\otimes j^\mathcal{B}_b)\circ\rho_E)=\\ F(j^\mathcal{A}_a\otimes j^\mathcal{B}_b)\circ F\rho_E=(Fj^\mathcal{A}_a\lozenge Fj^\mathcal{B}_b)\circ r_I=\\(j^{\mathbb{F}(\mathcal{A})}_a\lozenge j^{\mathbb{F}(\mathcal{B})}_b)\circ r_I=j^{\mathbb{F}(\mathcal{A})\nabla\mathbb{F}(\mathcal{B})}_{(a,b)}$,

    for every three pairs $(a,b),(\bar{a},\bar{b}),(\bar{\bar{a}},\bar{\bar{b}})\in ob(\mathcal{A})\times ob(\mathcal{B})$.
    \\

    \item $\mathbb{F}(T\bigcirc S)=\mathbb{F}(T)\nabla\mathbb{F}(S)$, for every two $\mathbb{C}$-functors $T:\mathcal{A}\rightarrow \mathcal{A}'$ and $S:\mathcal{B}\rightarrow\mathcal{B}'$, because:

       $ob\mathbb{F}(T\bigcirc S)=ob(T\bigcirc S)=obT\times obS=ob\mathbb{F}(T)\times ob\mathbb{F}(S)
       =ob(\mathbb{F}(T)\nabla\mathbb{F}(S))$;

       $(\mathbb{F}(T\bigcirc S))_{(a,b),(\bar{a},\bar{b})}=F(T\bigcirc S)_{(a,b),(\bar{a},\bar{b})}
       =F(T_{a,\bar{a}}\otimes S_{b,\bar{b}})=FT_{a,\bar{a}}\lozenge FS_{b,\bar{b}}
       =(\mathbb{F}T)_{a,\bar{a}}\lozenge (\mathbb{F}S)_{b,\bar{b}}
       =(\mathbb{F}T\nabla\mathbb{F}S)_{(a,b),(\bar{a},\bar{b})}$,

       for every two pairs $(a,b),(\bar{a},\bar{b})\in ob(\mathcal{A})\times ob(\mathcal{B})$.\\

\item $\mathbb{F}(\mathfrak{E})=\mathcal{I}$ because:

 $ob\mathbb{F}(\mathfrak{E})=ob(\mathfrak{E})=\{\ast\}=ob(\mathcal{I})$;

 $\mathbb{F}(\mathfrak{E})(\ast,\ast)=F(\mathfrak{E}(\ast,\ast))=F(E)=I=\mathcal{I}(\ast,\ast)$;

 $M^{\mathbb{F}(\mathfrak{E})}_{\ast,\ast,\ast}=FM^\mathfrak{E}_{\ast,\ast,\ast}
 =F\rho_E=r_I=M^\mathcal{I}_{\ast,\ast,\ast}:I\lozenge I\rightarrow I$, since $F$ preserves $\rho$;

 $j^{\mathbb{F}(\mathfrak{E})}_\ast =Fj^\mathfrak{E}_\ast =F1_E=1_{F(E)}=1_I=j^\mathcal{I}_\ast:I\rightarrow I$, since $F(E)=I$.\\

 \item $\mathbb{F}\wedge_{\mathcal{A},\mathcal{B},\mathcal{C}}=
     \vee_{\mathbb{F}(\mathcal{A}),\mathbb{F}(\mathcal{B}),\mathbb{F}(\mathcal{C})}$, for every $\mathcal{A},\mathcal{B},\mathcal{C}\in \mathbb{C}$-$Cat$, because:

 $ob\mathbb{F}\wedge_{\mathcal{A},\mathcal{B},\mathcal{C}}=ob\wedge_{\mathcal{A},\mathcal{B},\mathcal{C}}:
 (ob(\mathcal{A})\times ob(\mathcal{B}))\times ob(\mathcal{C})\rightarrow ob(\mathcal{A})\times (ob(\mathcal{B})\times ob(\mathcal{C}))$

 is the canonical isomorphism in $Set$,

 $ob\vee_{\mathbb{F}(\mathcal{A}),\mathbb{F}(\mathcal{B}),\mathbb{F}(\mathcal{C})}:(ob\mathbb{F}(\mathcal{A})\times ob\mathbb{F}(\mathcal{B}))\times ob\mathbb{F}(\mathcal{C})\rightarrow ob\mathbb{F}(\mathcal{A})\times (ob\mathbb{F}(\mathcal{B})\times ob\mathbb{F}(\mathcal{C}))$

  is also the canonical isomorphism in $Set$, and $ob\vee_{\mathbb{F}(\mathcal{A}),\mathbb{F}(\mathcal{B}),\mathbb{F}(\mathcal{C})}
 =ob\mathbb{F}\wedge_{\mathcal{A},\mathcal{B},\mathcal{C}}$ since $ob\mathbb{F}(\mathcal{A})=ob(\mathcal{A})$, $ob\mathbb{F}(\mathcal{B})=ob(\mathcal{B})$ and $ob\mathbb{F}(\mathcal{C})=ob(\mathcal{C})$;

 $(\mathbb{F}\wedge_{\mathcal{A},\mathcal{B},\mathcal{C}})_{((a,b),c),((\bar{a},\bar{b}),\bar{c})}=
 F(\wedge_{\mathcal{A},\mathcal{B},\mathcal{C}})_{((a,b),c),((\bar{a},\bar{b}),\bar{c})}
 =F\alpha_{\mathcal{A}(a,\bar{a}),\mathcal{B}(b,\bar{b}),\mathcal{C}(c,\bar{c})}
 =\mathfrak{a}_{F(\mathcal{A}(a,\bar{a})),F(\mathcal{B}(b,\bar{b})),F(\mathcal{C}(c,\bar{c}))}
 =\mathfrak{a}_{\mathbb{F}(\mathcal{A})(a,\bar{a}),\mathbb{F}(\mathcal{B})(b,\bar{b}),\mathbb{F}(\mathcal{C})(c,\bar{c})}
 =(\vee_{\mathbb{F}(\mathcal{A}),\mathbb{F}(\mathcal{B}),\mathbb{F}(\mathcal{C})})_{((a,b),c),((\bar{a},\bar{b}),\bar{c})}$,
 for every three pairs $(a,\bar{a})\in ob(\mathcal{A})\times ob(\mathcal{A}),(b,\bar{b})\in ob(\mathcal{B})\times ob(\mathcal{B})$ and $(c,\bar{c})\in ob(\mathcal{C})\times ob(\mathcal{C})$.\\

 \item $\mathbb{F}\Gamma_{\mathcal{A},\mathcal{B}}=\top_{\mathbb{F}(\mathcal{A}),\mathbb{F}(\mathcal{B})}$, for every $\mathcal{A},\mathcal{B}\in \mathbb{C}$-$Cat$, because:

     $ob\mathbb{F}\Gamma_{\mathcal{A},\mathcal{B}}=ob\Gamma_{\mathcal{A},\mathcal{B}}:ob(\mathcal{A})\times ob(\mathcal{B})\rightarrow ob(\mathcal{B})\times ob(\mathcal{A})$

     is the canonical isomorphism in $Set$,

     $ob\top_{\mathbb{F}(\mathcal{A}),\mathbb{F}(\mathcal{B})}:ob\mathbb{F}(\mathcal{A})\times ob\mathbb{F}(\mathcal{B})\rightarrow ob\mathbb{F}(\mathcal{B})\times ob\mathbb{F}(\mathcal{A})$

     is also the canonical isomorphism in $Set$, and $ob\top_{\mathbb{F}(\mathcal{A}),\mathbb{F}(\mathcal{B})}=ob\mathbb{F}\Gamma_{\mathcal{A},\mathcal{B}}$ since $ob\mathbb{F}(\mathcal{A})=ob(\mathcal{A})$ and $ob\mathbb{F}(\mathcal{B})=ob(\mathcal{B})$;

     $(\mathbb{F}\Gamma_{\mathcal{A},\mathcal{B}})_{(a,b),(\bar{a},\bar{b})}
     =F(\Gamma_{\mathcal{A},\mathcal{B}})_{(a,b),(\bar{a},\bar{b})}
     =F\gamma_{\mathcal{A}(a,\bar{a}),\mathcal{B}(b,\bar{b})}
     =t_{F(\mathcal{A}(a,\bar{a})),F(\mathcal{B}(b,\bar{b}))}
     =t_{\mathbb{F}(\mathcal{A})(a,\bar{a}),\mathbb{F}(\mathcal{B})(b,\bar{b})}
     =(\top_{\mathbb{F}(\mathcal{A}),\mathbb{F}(\mathcal{B})})_{(a,b),(\bar{a},\bar{b})}$,

     for every two pairs $(a,b),(\bar{a},\bar{b})\in ob(\mathcal{A})\times ob(\mathcal{B})$.\\

     \item $\mathbb{F}\mho_\mathcal{A}=\Re_{\mathbb{F}(\mathcal{A})}$, for every $\mathcal{A}\in\mathbb{C}$-$Cat$, because:

     $ob\mathbb{F}\mho_\mathcal{A}=ob\mho_\mathcal{A}:ob(\mathcal{A})\times\{\ast\}\rightarrow ob(\mathcal{A})$

     is the canonical projection in $Set$,

     $ob\Re_{\mathbb{F}(\mathcal{A})}=ob\mathbb{F}(\mathcal{A})\times\{\ast\}\rightarrow ob\mathbb{F}(\mathcal{A})$
     is also the canonical projection in $Set$, and $ob\mathbb{F}\mho_\mathcal{A}=ob\Re_{\mathbb{F}(\mathcal{A})}$ since $ob\mathbb{F}(\mathcal{A})=ob(\mathcal{A})$;

     $(\mathbb{F}\mho_\mathcal{A})_{(a,\ast),(\bar{a},\ast)}=F(\mho_\mathcal{A})_{(a,\ast),(\bar{a},\ast)}
     =F\rho_{\mathcal{A}(a,\bar{a})}=r_{\mathcal{A}(a,\bar{a})}=(\Re_\mathcal{A})_{(a,\ast),(\bar{a},\ast)}$,

     for every pair $(a,\bar{a})\in ob(\mathcal{A})\times ob(\mathcal{A})$.
\end{itemize}

\subsection{$\mathbb{G}$ preserves the (symmetric) monoidal structure}
\label{subsection:GpreservesMonoStruc}

Analogously, the same conclusions, drawn for $\mathbb{F}:\mathbb{C}$-$Cat\rightarrow\mathbb{X}$-$Cat$ in the previous subsection \ref{subsection:FpreservesMonoStruc}, hold as well for $\mathbb{G}:\mathbb{X}$-$Cat\rightarrow\mathbb{C}$-$Cat$, which therefore preserves the symmetric monoidal structure.

\section{Limits in a category $\mathcal{V}$-$Cat$ of all $\mathcal{V}$-categories ($\mathcal{V}=\mathbb{C},\mathbb{X}$)}\label{section:limitsV-Cat}

In this section, it will be proved that the existence of limits in a category $\mathbb{C}$ implies their existence in $\mathbb{C}$-$Cat$ as well (cf.\ the following Proposition \ref{proposition:limitsC-Cat}), showing how they can be obtained in $\mathbb{C}$-$Cat$ using the limits in $\mathbb{C}$ and in the category $Set$ of all sets.

\begin{proposition}\label{proposition:limitsC-Cat}
Given a functor \begin{center}$A:\mathbb{I}\rightarrow\mathbb{C}$-$Cat$,\end{center} $$\tau :i\rightarrow j\mapsto T=A\tau :\mathcal{A}_i\rightarrow \mathcal{A}_j,$$ consider:

\begin{itemize}
\item the underlying object functor \begin{center}$ob:\mathbb{C}$-$Cat\rightarrow Set$,\end{center} $$S:\mathcal{A}\rightarrow \mathcal{B}\mapsto obS:ob(\mathcal{A})\rightarrow ob(\mathcal{B});$$

\item its composition with $A$ \begin{center}$ob\circ A:\mathbb{I}\rightarrow\mathbb{C}$-$Cat\rightarrow Set$,\end{center} $$\tau :i\rightarrow j\mapsto obT=obA\tau:ob(\mathcal{A}_i)\rightarrow ob(\mathcal{A}_j),$$

\item and the limiting cone of $ob\circ A$,\footnote{$\Delta :Set\rightarrow Set^\mathbb{I}$ is the diagonal functor, right-adjoint of $\underleftarrow{Lim}$.} $$ob\pi :\Delta ob(\mathcal{L})\rightarrow ob\circ A;$$

\item each projection of this limiting cone $ob\pi$ in $Set$ will be denoted by $$ob\pi^i:ob(\mathcal{L})\rightarrow ob(\mathcal{A}_i);$$

\item if $a\in ob(\mathcal{L})$ then $a=(a_i)_{i\in\mathbb{I}}$, since $ob(\mathcal{L})\subseteq\prod_{i\in\mathbb{I}}ob(\mathcal{A}_i)$;

\item for each pair $a=(a_i)_{i\in\mathbb{I}}$, $b=(b_i)_{i\in\mathbb{I}}\in ob(\mathcal{L})$, the functor $$A_{a,b}:\mathbb{I}\rightarrow\mathbb{C}$$ $$\tau :i\rightarrow j\mapsto T_{a_i,b_i}:\mathcal{A}_i(a_i,b_i)\rightarrow\mathcal{A}_j(a_j,b_j),$$ where $a_i=ob\pi^i(a)$, $b_i=ob\pi^i(b)$, $a_j=obT(a_i)=ob\pi^j(a)$, $b_j=obT(b_i)=ob\pi^j(b)$.

\end{itemize}
\vspace{10pt}
Suppose that $A_{a,b}$ has a limiting cone in $\mathbb{C}$,\footnote{$\Delta :\mathbb{I}\rightarrow\mathbb{C}^\mathbb{I}$, another diagonal functor.} $$\pi_{a,b}:\Delta\mathcal{L}(a,b)\rightarrow A_{a,b},$$ with projections $\pi^i_{a,b}:\mathcal{L}(a,b)\rightarrow\mathcal{A}_i(a_i,b_i)$, $i\in\mathbb{I}$, for every pair $a,b\in ob(\mathcal{L})$.

Then, the functor $A$ has a limiting cone $\pi:\Delta\mathcal{L}\rightarrow A$.\footnote{$\Delta:\mathbb{I}\rightarrow (\mathbb{C}$-$Cat)^\mathbb{I}$, another diagonal functor.}

\end{proposition}

\begin{proof}

First, the description of $\mathcal{L}\in\mathbb{C}$-$Cat$ and of each $\mathbb{C}$-functor projection $\pi^i:\mathcal{L}\rightarrow\mathcal{A}_i$ will be given, $i\in\mathbb{I}$.

Secondly, it will be shown that $\pi=(\pi^i)_{i\in\mathbb{I}}$ is in fact a limiting cone in $\mathbb{C}$-$Cat$.

In more detail: we will describe (1) $ob(\mathcal{L})$;

\noindent (2) $\mathcal{L}(a,b)$, (3) $M^\mathcal{L}_{a,b,c}:\mathcal{L}(b,c)\otimes\mathcal{L}(a,b)\rightarrow\mathcal{L}(a,c)$, (4) $j^\mathcal{L}_a:E\rightarrow\mathcal{L}(a,a)$, for all $a,b,c\in ob(\mathcal{L})$;

and we will check if (5) $\mathcal{L}$ is a $\mathbb{C}$-category, that is, if (5.1) the associativity axioms and (5.2) unit axioms hold;

(6) if $\pi^i:\mathcal{L}\rightarrow\mathcal{A}_i$ is a $\mathbb{C}$-functor, for each $i\in\mathbb{I}$;

and finally if (7) $\pi$ is a limiting cone in $\mathbb{C}$-$Cat$.\\

(1) The set $ob(\mathcal{L})$ of objects was already given above in the statement, it is the limit of $ob\circ A$ in $Set$.

(2) For each pair $a,b$ of objects of $\mathcal{L}$, the hom-object $\mathcal{L}(a,b)$ is the limit of $A_{a,b}$ in $\mathbb{C}$, which exists by assumption: $\Delta\mathcal{L}(a,b)^{\underrightarrow{\pi_{a,b}}}A_{a,b}$ (the projections are $\mathcal{L}(a,b)^{\underrightarrow{\pi^i_{a,b}}}A_i(a_i,b_i)$, $i\in\mathbb{I}$).

(3) The composition law for each triple $a,b,c$ of objects is given by the unique morphism
$$M^\mathcal{L}_{a,b,c}:\mathcal{L}(b,c)\otimes\mathcal{L}(a,b)\rightarrow\mathcal{L}(a,c)$$
in the obvious limit diagram in $\mathbb{C}$ corresponding to the equation
$$M_{a,b,c}\cdot (\pi_{b,c}\otimes\pi_{a,b})=\pi_{a,c}\cdot\Delta M^\mathcal{L}_{a,b,c}:\Delta(\mathcal{L}(b,c)\otimes\mathcal{L}(a,b))\rightarrow\Delta\mathcal{L}(a,c)\rightarrow A_{a,c},$$
where
$$M_{a,b,c}=(M^{\mathcal{A}_i}_{a_i,b_i,c_i})_{i\in\mathbb{I}}:\otimes\circ <A_{b,c},A_{a,b}>\rightarrow A_{a,c}:\mathbb{I}\rightarrow\mathbb{C}\times\mathbb{C}\rightarrow\mathbb{C}$$
is a natural transformation due to the compatibility with composition of the $\mathbb{C}$-functors,
$$\pi_{b,c}\otimes\pi_{a,b}=(\pi^i_{b,c}\otimes\pi^i_{a,b})_{i\in\mathbb{I}}:\Delta(\mathcal{L}(b,c)\otimes\mathcal{L}(a,b))\rightarrow\otimes\circ <A_{b,c},A_{a,b}>:\mathbb{I}\rightarrow\mathbb{C}\times\mathbb{C}\rightarrow\mathbb{C}$$
is the cone obtained by tensoring the two cones

$\pi_{b,c}:\Delta\mathcal{L}_{b,c}\rightarrow A_{b,c}:\mathbb{I}\rightarrow\mathbb{C}$ and $\pi_{a,b}:\Delta\mathcal{L}_{a,b}\rightarrow A_{a,b}:\mathbb{I}\rightarrow\mathbb{C}$.

(4) The identity element is given by the unique morphism
$$j^\mathcal{L}_a:E\rightarrow\mathcal{L}(a,a),$$ for each object $a\in ob(\mathcal{L})$,
in the obvious limit diagram in $\mathbb{C}$ corresponding to the equation
$$j_a=\pi_{a,a}\cdot\Delta j^\mathcal{L}_a:\Delta E\rightarrow\Delta\mathcal{L}(a,a)\rightarrow A_{a,a},$$
where
$$j_a=(j^{\mathcal{A}_i}_a)_{i\in\mathbb{I}}:\Delta E\rightarrow A_{a,a}:\mathbb{I}\rightarrow\mathbb{C}$$
is a natural transformation due to the compatibility with identities of the $\mathbb{C}$-functors.\\

It is then a trivial task to check that (5.1) the associative and (5.2) unit axioms hold for $\mathcal{L}$ so defined,

and therefore that (5) it is a $\mathbb{C}$-category:\\

checking the associativity axiom (5.1) is to ask if the following equation holds,
$$M^\mathcal{L}_{a,b,d}\circ (M^\mathcal{L}_{b,c,d}\otimes 1_{\mathcal{L}(a,b})
=M^\mathcal{L}_{a,c,d}\circ (1_{\mathcal{L}(c,d)}\otimes M^\mathcal{L}_{a,b,c})\circ\alpha_{\mathcal{L}(c,d),\mathcal{L}(b,c),\mathcal{L}(a,b)},$$
which is true by the uniqueness due to the limiting cone
$$\pi_{a,d}:\Delta \mathcal{L}(a,d)\rightarrow A(a,d),$$
since (cf.\ (2) and (3) just above for the notations)

$M_{a,b,d}\cdot (M_{b,c,d}\otimes 1_{A(a,b})\cdot ((\pi_{c,d}\otimes\pi_{b,c})\otimes\pi_{a,b})\\
=M_{a,c,d}\cdot (1_{A(c,d)}\otimes M_{a,b,c})\cdot (\pi_{c,d}\otimes (\pi_{b,c}\otimes\pi_{a,b}))\cdot\Delta\alpha_{\mathcal{L}(c,d),\mathcal{L}(b,c),\mathcal{L}(a,b)}
:\Delta ((\mathcal{L}(c,d)\otimes\mathcal{L}(b,c))\otimes\mathcal{L}(a,b))\rightarrow A(a,d)$;\\

consider the commutative diagram in $\mathbb{C}$

\begin{picture}(370,70)

\put(-20,40){$\mathcal{L}(a,b)\otimes E$}\put(55,43){\vector(1,0){60}}\put(55,49){$1_{\mathcal{L}(a,b)}\otimes j^\mathcal{L}_a$}
\put(130,40){$\mathcal{L}(a,b)\otimes\mathcal{L}(a,a)$}\put(230,43){\vector(1,0){60}}
\put(237,49){$M^\mathcal{L}_{a,a,b}$}\put(293,40){$\mathcal{L}(a,b)$}

\put(20,33){\vector(0,-1){20}}\put(23,20){$\pi^i_{a,b}\otimes 1_E$}
\put(175,20){$\pi^i_{a,b}\otimes\pi^i_{a,a}$}\put(173,33){\vector(0,-1){20}}
\put(305,20){$\pi^i_{a,b}$}\put(303,33){\vector(0,-1){20}}

\put(-20,0){$\mathcal{A}_i(a_i,b_i)\otimes E$}\put(55,3){\vector(1,0){60}}
\put(52,9){$1_{\mathcal{A}_i(a_i,b_i)}\otimes j^{\mathcal{A}_i}_{a_i}$}
\put(120,0){$\mathcal{A}_i(a_i,b_i)\otimes\mathcal{A}_i(a_i,a_i)$}
\put(230,3){\vector(1,0){60}}\put(237,9){$M^{\mathcal{A}_i}_{a_i,a_i,b_i}$}
\put(293,0){$\mathcal{A}_i(a_i,b_i)$,}

\end{picture}\\

as $\rho_{\mathcal{A}_i(a_i,b_i)}=M^{\mathcal{A}_i}_{a_i,a_i,b_i}\circ (1_{\mathcal{A}_i(a_i,b_i)}\otimes j^{\mathcal{A}_i}_{a_i})$, for every $i\in \mathbb{I}$,

then

$\pi^i_{a,b}\circ\rho_{\mathcal{L}(a,b)}=\rho_{\mathcal{A}_i(a_i,b_i)}\circ(\pi^i_{a,b}\otimes 1_E)=M^{\mathcal{A}_i}_{a_i,a_i,b_i}\circ (1_{\mathcal{A}_i(a_i,b_i)}\otimes j^{\mathcal{A}_i}_{a_i})\circ(\pi^i_{a,b}\otimes 1_E)=\pi^i_{a,b}\circ M^{\mathcal{L}}_{a,a,b}\circ (1_{\mathcal{L}(a,b)}\otimes j^\mathcal{L}_a)$, for every $i\in \mathbb{I}$,

and so, by the uniqueness due to the limiting cone $\pi_{a,b}:\Delta\mathcal{L}(a,b)\rightarrow A(a,b)$, the right unit axiom (5.2.1) $\rho_{\mathcal{L}(a,b)}=M^\mathcal{L}_{a,a,b}\circ (1_{\mathcal{L}(a,b)}\otimes j^\mathcal{L}_a)$ holds;\\

consider the commutative diagram in $\mathbb{C}$

\begin{picture}(370,70)

\put(-20,40){$E\otimes\mathcal{L}(a,b)$}\put(55,43){\vector(1,0){60}}
\put(55,49){$j^\mathcal{L}_b\otimes 1_{\mathcal{L}(a,b)} $}
\put(130,40){$\mathcal{L}(b,b)\otimes\mathcal{L}(a,b)$}\put(230,43){\vector(1,0){60}}
\put(237,49){$M^\mathcal{L}_{a,b,b}$}\put(293,40){$\mathcal{L}(a,b)$}

\put(20,33){\vector(0,-1){20}}\put(-28,20){$1_E\otimes\pi^i_{a,b}$}
\put(175,20){$\pi^i_{b,b}\otimes\pi^i_{a,b}$}\put(173,33){\vector(0,-1){20}}
\put(305,20){$\pi^i_{a,b}$}\put(303,33){\vector(0,-1){20}}

\put(-20,0){$E\otimes\mathcal{A}_i(a_i,b_i)$}\put(55,3){\vector(1,0){60}}
\put(52,9){$j^{\mathcal{A}_i}_{b_i}\otimes 1_{\mathcal{A}_i(a_i,b_i)}$}
\put(120,0){$\mathcal{A}_i(b_i,b_i)\otimes\mathcal{A}_i(a_i,b_i)$}
\put(230,3){\vector(1,0){60}}\put(237,9){$M^{\mathcal{A}_i}_{a_i,b_i,b_i}$}
\put(293,0){$\mathcal{A}_i(a_i,b_i)$,}

\end{picture}\\

as $\rho_{\mathcal{A}_i(a_i,b_i)}\circ\gamma_{E,\mathcal{A}_i(a_i,b_i)}=
M^{\mathcal{A}_i}_{a_i,b_i,b_i}\circ
(j^{\mathcal{A}_i}_{b_i}\otimes 1_{\mathcal{A}_i(a_i,b_i)})$, for every $i\in \mathbb{I}$,

then

$\pi^i_{a,b}\circ\rho_{\mathcal{L}(a,b)}\circ\gamma_{E,\mathcal{L}(a,b)}=
\rho_{\mathcal{A}_i(a_i,b_i)}\circ\gamma_{E,\mathcal{A}_i(a_i,b_i)}\circ(1_E\otimes\pi^i_{a,b})=
M^{\mathcal{A}_i}_{a_i,b_i,b_i}\circ
(j^{\mathcal{A}_i}_{b_i}\otimes 1_{\mathcal{A}_i(a_i,b_i)})\circ (1_E\otimes\pi^i_{a,b})=
\pi^i_{a,b}\circ M^{\mathcal{L}}_{a,b,b}\circ (j^\mathcal{L}_b\otimes 1_{\mathcal{L}(a,b)})$, for every $i\in \mathbb{I}$,

and so, by the uniqueness due to the limiting cone $\pi_{a,b}:\Delta\mathcal{L}(a,b)\rightarrow A(a,b)$, the left unit axiom (5.2.2) $\rho_{\mathcal{L}(a,b)}\circ\gamma_{E,\mathcal{L}(a,b)}=M^\mathcal{L}_{a,b,b}
\circ (j^\mathcal{L}_b\otimes 1_{\mathcal{L}(a,b)})$ holds.\\

It is also obvious that (6) $\pi^i$ is a $\mathbb{C}$-functor, by its definition,
\begin{itemize}
\item $ob\pi^i:ob(\mathcal{L})\rightarrow ob(\mathcal{A}_i)$, $a=(a_i)_{i\in \mathbb{I}}\mapsto a_i$,
\item $\pi^i_{a,b}:\mathcal{L}(a,b)\rightarrow \mathcal{A}_i(a_i,b_i)$,
\end{itemize}
and by the definitions of $M^\mathcal{L}_{a,b,c}$ and $j^\mathcal{L}_a$, for every $a,b,c\in ob(\mathcal{L})$. In fact, $\pi^i$ is compatible with composition by definition of $M^\mathcal{L}_{a,b,c}$, and $\pi^i$ is compatible with the identities by definition of $j^\mathcal{L}_a$, for every $a,b,c\in ob(\mathcal{L})$.\\

It remains to show that $\pi :\Delta\mathcal{L}\rightarrow A$ is a universal cone. Let $\lambda:\Delta\mathcal{A}\rightarrow A$ be another cone,
from any $\mathcal{A}$($\in\mathbb{C}$-$Cat$) into $A:\mathbb{I}\rightarrow \mathbb{C}$-$Cat$.

The functor $ob:\mathbb{C}$-$Cat\rightarrow Set$ takes the cone $\pi$ into a limiting cone $$ob\pi :\Delta ob(\mathcal{L})\rightarrow ob\circ A$$ in the category $Set$ of all sets, and it is assumed that there is a limiting cone $$\pi_{a,b}:\Delta\mathcal{L}(a,b)\rightarrow A_{a,b}$$ in $\mathbb{C}$, for each ordered pair $(a,b)\in ob(\mathcal{L})\times ob(\mathcal{L})$.

The universality of these cones, in $Set$ and $\mathbb{C}$ respectively, provides a unique function $obL:ob(\mathcal{A})\rightarrow ob(\mathcal{L})$ and unique morphisms\footnote{The notation is going to be simplified as mentioned before; for instance, $L(x)$ means in fact $ob(L(x))$.} $L_{x,y}:\mathcal{A}(x,y)\rightarrow \mathcal{L}(L(x),L(y))$ for each pair of objects $x,y\in ob(\mathcal{A})$. Hence, if $L$ is a $\mathbb{C}$-functor it must be unique, and $L=\underleftarrow{Lim}A$ in $\mathbb{C}$-$Cat$.\\

$L$ is compatible with the composition, that is, $$L_{x,z}\circ M^\mathcal{A}_{x,y,z}=M^\mathcal{L}_{L(x),L(y),L(z)}\circ (L_{y,z}\otimes L_{x,y})$$ for every $x,y,z\in ob(\mathcal{A})$, because

$\pi^i_{L(x),L(z)}\circ L_{x,z}\circ M^\mathcal{A}_{x,y,z}=
\lambda^i_{x,z}\circ M^\mathcal{A}_{x,y,z}\\
=M^{\mathcal{A}_i}_{L(x)_i,L(y)_i,L(z)_i}\circ(\lambda^i_{y,z}\otimes\lambda^i_{x,y})$, since $\lambda^i$ is a $\mathbb{C}$-functor\\
$=M^{\mathcal{A}_i}_{L(x)_i,L(y)_i,L(z)_i}\circ(\pi^i_{L(y),L(z)}\otimes\pi^i_{L(x),L(y)})\circ(L_{y,z}\otimes L_{x,y})$, since $\otimes$ is a bifunctor\\
$=\pi^i_{L(x),L(z)}\circ M^\mathcal{L}_{L(x),L(y),L(z)}\circ (L_{y,z}\otimes L_{x,y})$, since (6) $\pi^i$ is a $\mathbb{C}$-functor, for every $i\in \mathbb{I}$.\\

$L$ is also compatible with the identities, that is, $$L_{x,x}\circ j^\mathcal{A}_x=j^\mathcal{L}_{L(x)},$$ for every $x\in ob(\mathcal{A})$, because

$\pi^i_{L(x),L(x)}\circ L_{x,x}\circ j^\mathcal{A}_x=
\lambda^i_{x,x}\circ j^\mathcal{A}_x=j^{\mathcal{A}_i}_{L(x)_i}$, since $\lambda^i$ is a $\mathbb{C}$-functor\\
$=\pi^i_{L(x),L(x)}\circ j^\mathcal{L}_{L(x)}$, since (6) $\pi^i$ is a $\mathbb{C}$-functor, for every $i\in \mathbb{I}$.\\

The proof that limits in $\mathbb{C}$-$Cat$ may be calculated ``hom-componentwise" in $\mathbb{C}$ is done.
\end{proof}

\begin{remark}\label{remark:cartesian monoidal}
It is straightforward to confirm that, if $(\mathbb{C},\otimes,E,\alpha,\gamma,\rho)$ is a cartesian monoidal category then its derived monoidal category $\mathbb{C}$-$Cat=(\mathbb{C}$-$Cat,\bigcirc,\mathfrak{E},\wedge,\Gamma,\mho)$ is also cartesian. Just check in this section how to obtain the cartesian products in a $\mathcal{V}$-category, and compare with the definition of the derived monoidal structure in section \ref{section:derivedMonoidal}, provided $\otimes=\times$.
\end{remark}

\section{The base monoidal reflection}\label{section:derivedReflection}

The results in this paper are to be applied in categorical Galois theory (cf.\ \cite{G. Janelidze}) to admissible\footnote{Also called semi-left-exact as introduced in \cite{CHK:fact} (cf.\ \cite{CJKP:stab}).} reflections of full subcategories. With this in mind, a special case of the base monoidal adjunction will be considered now.

If $\mathbb{X}$ is reflective in $\mathbb{C}$, that is, $\mathbb{X}$ is a subcategory of $\mathbb{C}$, $G$ is the inclusion functor and $F$ is the reflector, it is obvious that $E=I$, and that the bifunctor $\lozenge$ and the natural transformations $\mathfrak{a}$, $t$ and $r$ are just the respective restrictions of $\otimes$, $\alpha$, $\gamma$ and $\rho$.

It is well known that, provided the inclusion $G$ is also a full functor then every counit morphism $\varepsilon_X:FG(X)\rightarrow X$ is an isomorphism, $X\in\mathbb{X}$ (cf.\ \cite[\S IV.3]{SM:cat}). Furthermore, if $\mathbb{X}$ is a replete subcategory of $\mathbb{C}$, that is, $\mathbb{X}$ contains any object of $\mathbb{C}$ which is isomorphic to some other object of $\mathbb{X}$ (cf.\ \cite[\S 3.1]{CJKP:stab}), then this allows to choose the unit $\eta:1_\mathbb{C}\rightarrow GF$ so that the counit is the identity $\varepsilon:FG=1_\mathbb{X}$ ($G\varepsilon_X\circ \eta_{G(X)}=1_{G(X)}\Leftrightarrow \eta_{G(X)}=\varepsilon^{-1}_X$, provided $G$ is an inclusion; cf.\ Theorem 2(ii) in \cite[\S IV.1]{SM:cat}).\\

\begin{definition}\label{def:base monoidal reflection}
An adjunction $(A)$ as in section \ref{sec-base monoidal adjunction} and satisfying $(B)$ will be called \emph{base monoidal reflection} if $\mathbb{X}$ is a full replete subcategory of $\mathbb{C}$ and $G:\mathbb{X}\subseteq \mathbb{C}$ is the inclusion functor.
\end{definition}

In the following Proposition \ref{proposition:derivedReflection}, it is shown that from a base monoidal reflection another base monoidal reflection is derived, using the process introduced in section \ref{sec:derived}.

\begin{proposition}\label{proposition:derivedReflection}

Consider a base monoidal reflection as in Definition \ref{def:base monoidal reflection}, determined by the (symmetric) monoidal category $(\mathbb{C},\otimes,E,\alpha,\gamma,\rho)$ and the reflection $(F,G,\eta,\varepsilon):\mathbb{C}\rightarrow \mathbb{X}$. Then, there is a derived adjunction \begin{center}$(\mathbb{F},\mathbb{G},\Theta,\Upsilon):\mathbb{C}$-$Cat\rightarrow\mathbb{X}$-$Cat$,\end{center} such that $\mathbb{G}$ is the inclusion functor and $\mathbb{X}$-$Cat$ is a full replete subcategory of $\mathbb{C}$-$Cat$.

\end{proposition}

\begin{proof}
In order to check that there is a derived adjunction, we have to show that both conditions $(C)$ and $(D)$ (cf.\ section \ref{sec-base monoidal adjunction}) hold for the base adjunction:

\noindent $(C)$ holds because, for every $X,Y\in\mathbb{X}$,

$\varepsilon_X\lozenge\varepsilon_Y=\varepsilon_{X\lozenge Y}\Leftrightarrow \varepsilon_X\otimes\varepsilon_Y=\varepsilon_{X\otimes Y}$, since $\lozenge$ is the restriction of $\otimes$

$\Leftrightarrow 1_X\otimes 1_Y=1_{X\otimes Y}$, since $\varepsilon :FG=1_\mathbb{X}$,

\noindent which is true simply because $\otimes$ is a bifunctor, implying condition $(C)$ by the duality given in Proposition \ref{proposition:duality eta epsilon}$(i)$;

\noindent $(D)$ holds since $\eta_E=1_E\Leftrightarrow \varepsilon_I=1_I$, by Proposition \ref{proposition:duality eta epsilon}$(ii)$, and $\varepsilon :FG=1_\mathbb{X}$.

As $ob\Upsilon_\mathcal{X}=1_{ob(\mathcal{X})}$ and $(\Upsilon_\mathcal{X})_{a,b}=\varepsilon_{\mathcal{X}(a,b)}=1_{\mathcal{X}(a,b)}$, for every $\mathcal{X}\in\mathbb{X}$-$Cat$ and for every pair $a,b\in ob(\mathcal{X})$, it follows that $\Upsilon :\mathbb{FG}=1_{\mathbb{X}}$-$_{Cat}$ is the identity.\\

It is also obvious that $\mathbb{G}$ is the inclusion functor:

$ob\mathbb{G}(\mathcal{X})=ob(\mathcal{X})$, $\mathbb{G}(\mathcal{X})(a,b)=G(\mathcal{X}(a,b))=\mathcal{X}(a,b)$, $j^{\mathbb{G}(\mathcal{X})}_a=Gj^\mathcal{X}_a=j^\mathcal{X}_a$, and $M^{\mathbb{G}(\mathcal{X})}_{a,b,c}=GM^\mathcal{X}_{a,b,c}=M^\mathcal{X}_{a,b,c}$, for every $\mathcal{X}\in\mathbb{X}$-$Cat$ and $a,b,c\in ob(\mathcal{X})$;

for every $\mathbb{X}$-functor $T:\mathcal{X}\rightarrow \mathcal{Y}$, $ob\mathbb{G}T=obT$ and $(\mathbb{G}T)_{x,y}=GT_{x,y}=T_{x,y}:\mathcal{X}(x,y)\rightarrow \mathcal{Y}(T(x),T(y))$, for every pair $x,y\in\mathcal{X}$;

so that, if $\mathbb{G}T=\mathbb{G}S$ then necessarily $T=S$.\\

It remains to show that $\mathbb{X}$-$Cat$ is a replete subcategory of $\mathbb{C}$-$Cat$, which follows from the characterization of isomorphisms in $\mathbb{C}$-$Cat$ (cf.\ Proposition \ref{proposition:IsosInC-Cat} in section \ref{sec-Vcats}):

 if $\mathcal{A}\in\mathbb{C}$-$Cat$ is isomorphic to $\mathcal{X}\in\mathbb{X}$-$Cat$, then there is a $\mathbb{C}$-functor $T:\mathcal{A}\rightarrow \mathcal{X}$ such that $T_{a,b}:A(a,b)\rightarrow X(T(a),T(b))$ is an isomorphism in $\mathbb{C}$, for every pair $a,b\in ob(\mathcal{A})$;

 as $\mathbb{X}$ is replete in $\mathbb{C}$, this means that $A(a,b)\in \mathbb{X}$, for every pair $a,b\in ob(\mathcal{A})$, which implies that $\mathcal{A}\in\mathbb{X}$-$Cat$ by definition of $\mathbb{X}$-$Cat$ (remark that $M^\mathcal{A}_{a,b,c}:\mathcal{A}(b,c)\otimes\mathcal{A}(a,b)=\mathcal{A}(b,c)\lozenge\mathcal{A}(a,b)\rightarrow\mathcal{A}(a,c)$ and $j^\mathcal{A}_a:E=I\rightarrow\mathcal{A}(a,a)$ are morphisms in $\mathbb{X}$, by the fullness of $G:\mathbb{X}\subseteq\mathbb{C}$).
\end{proof}

\section{Simplicity, semi-left-exactness\\ and the stable units property}\label{sec:Simp., semi-left-exact. and stable units}

In the present section we refer to a base monoidal reflection $$(F,G,\eta,\varepsilon):\mathbb{C}\rightarrow\mathbb{X}$$ as in Definition \ref{def:base monoidal reflection}, and such that $\mathbb{C}$ has pullbacks.

Remember that, by definition, in a base monoidal reflection the right adjoint is the inclusion functor of a full and replete subcategory, and the counit is the identity.

The following Proposition \ref{proposition:stableunits} states that some properties of the left adjoint $F$, concerning the preservation of certain kinds of pullback diagrams (which are relevant in Galois categorical theory), are inherited by the left adjoint $\mathbb{F}$ in the derived base monoidal reflection.

Namely, going from the weaker to the stronger property, if $F\dashv G$ is a \emph{simple reflection}, or \emph{semi-left-exact} (also called \emph{admissible}), or having \emph{stable units} in the sense of \cite{CHK:fact} (cf.\ \cite{CJKP:stab}), then $\mathbb{F}\dashv\mathbb{G}$ does have as well the respective property.

\begin{proposition}\label{proposition:stableunits}
Consider a category $\mathbb{C}$ with pullbacks, a base monoidal reflection $$(F,G,\eta,\varepsilon):\mathbb{C}\rightarrow\mathbb{X}$$ as in Definition \ref{def:base monoidal reflection}, and its derived adjunction \begin{center}$(\mathbb{F},\mathbb{G},\Theta,\Upsilon):\mathbb{C}$-$Cat\rightarrow\mathbb{X}$-$Cat$.\end{center}

Then, if $F\dashv G$ has stable units, if it is semi-left-exact or simple, so is $\mathbb{F}\dashv \mathbb{G}$ respectively.
\end{proposition}
\begin{proof}
Notice first that $\mathbb{F}\dashv \mathbb{G}$ is also a base monoidal reflection (cf.\ Proposition \ref{proposition:derivedReflection}), such that $\mathbb{C}$-$Cat$ has pullbacks (cf.\ Proposition \ref{proposition:limitsC-Cat}).\\

1. (stable units) Consider the diagram in $\mathbb{X}$-$Cat$, which is the image of a pullback diagram in $\mathbb{C}$-$Cat$,

\begin{picture}(300,80)
\put (40,0){$\mathbb{F}(\mathcal{A}_1)$}\put (40,50){$\mathbb{F}(\mathcal{L})$}
\put (145,0){$\mathbb{F}(\mathcal{A}_3)$\ ,}\put (145,50){$\mathbb{F}(\mathcal{A}_2)$}

\put(30,25){$\mathbb{F}U$}\put(160,25){$\mathbb{F}S$}\put(200,25){(9.1)}
\put(95,10){$\mathbb{F}T$}\put(95,60){$\mathbb{F}V$}

\put(75,3){\vector(1,0){58}}\put(75,55){\vector(1,0){58}}
\put(50,45){\vector(0,-1){33}}\put (155,45){\vector(0,-1){33}}
\end{picture}\\

\noindent and such that $\mathcal{A}_3\in\mathbb{X}$-$Cat$. We have to prove that this is also a pullback diagram in $\mathbb{X}$-$Cat$, provided that the reflection $F\dashv G$ has stable units (that is, if $F$ preserves any pullback diagram in $\mathbb{C}$ in which the right down object is in the subcategory $\mathbb{X}$).

As $\mathbb{F}$ does not change the sets of objects, the image by $ob:\mathbb{X}$-$Cat\rightarrow Set$ of diagram (9.1) is the same as the image by $ob:\mathbb{C}$-$Cat\rightarrow Set$ of the diagram prior to the application of $\mathbb{F}$, and therefore it is a pullback in $Set$.

Now, in order to confirm that diagram (9.1) is a pullback in $\mathbb{X}$-$Cat$, we need only to prove that the following square (9.2) is a pullback, according to Proposition \ref{proposition:limitsC-Cat} and the definition of $\mathbb{F}$, for every pair $a,b$ of objects in $\mathcal{L}$,

\begin{picture}(230,80)
\put (5,0){$F(\mathcal{A}_1(a_1,b_1))$}\put (20,50){$F(\mathcal{L}(a,b))$}
\put (140,0){$F(\mathcal{A}_3(a_3,b_3))$\ ;}\put (140,50){$F(\mathcal{A}_2(a_2,b_2))$}

\put(20,25){$FU_{a,b}$}\put(170,25){$FS_{a_2,b_2}$}\put(220,25){(9.2)}
\put(90,10){$FT_{a_1,b_1}$}\put(90,60){$FV_{a,b}$}

\put(75,3){\vector(1,0){58}}\put(75,55){\vector(1,0){58}}
\put(50,45){\vector(0,-1){33}}\put (165,45){\vector(0,-1){33}}
\end{picture}\\

\noindent which is true since $\mathcal{A}_3\in\mathbb{X}$-$Cat$ implies that $\mathcal{A}_3(a_3,b_3)\in\mathbb{X}$, and diagram (9.2) is an image by $F$ of a pullback diagram in $\mathbb{C}$, being $F$ the left-adjoint of a reflection into $\mathbb{C}$ with stable units.\\

2. (semi-left-exact) Consider the pullback diagram in $\mathbb{C}$-$Cat$

\begin{picture}(210,80)
\put (57,0){$\mathcal{A}$}\put (60,50){$\mathcal{L}$}
\put (145,0){$\mathbb{GF}(\mathcal{A})$\ ,}\put(145,50){$\mathbb{G}(\mathcal{X})$}

\put(50,25){$U$}\put(160,25){$\mathbb{G}S$}
\put(100,10){$\Theta_{\mathcal{A}}$}\put(100,60){$V$}

\put(75,3){\vector(1,0){62}}\put(75,55){\vector(1,0){62}}
\put(65,45){\vector(0,-1){33}}\put (155,45){\vector(0,-1){33}}
\end{picture}\\

\noindent where $\Theta_{\mathcal{A}}$ is a unit morphism in the derived reflection.\\

We have to show that $\mathbb{F}V:\mathbb{F}(\mathcal{L})\rightarrow \mathbb{FG}(\mathcal{X})=\mathcal{X}$ is an isomorphism in $\mathbb{X}$-$Cat$, provided the reflection $F\dashv G$ is semi-left-exact (that is, $F$ preserves the pullback diagrams in $\mathbb{C}$ in which the down arrow is a unit morphism in the base reflection, and the right up object is in $\mathbb{X}$).\\

As the diagram is a pullback in $\mathbb{C}$-$Cat$, $obV$ is a bijection since $ob\Theta_{\mathcal{A}}$ is the identity. Hence, $ob\mathbb{F}V:ob(\mathbb{F}(\mathcal{L}))\rightarrow ob(\mathcal{X})$ is a bijection as well, because $\mathbb{F}$ does not change the sets of objects. It remains to show, that the following diagram is a pullback in $\mathbb{X}$, for every $a,b\in ob(\mathcal{L})$,

\begin{picture}(100,80)
\put (5,0){$F(\mathcal{A}(a_1,b_1))$}\put (20,50){$F(\mathcal{L}(a,b))$}
\put (140,0){$F(\mathcal{A}(a_1,b_1))$\ ,}\put (140,50){$F(\mathcal{X}(a,b))$}

\put(20,25){$FU_{a,b}$}\put(170,25){$FS_{a,b}$}
\put(90,10){$F\eta_{\mathcal{A}(a_1,b_1)}$}\put(90,60){$FV_{a,b}$}

\put(75,3){\vector(1,0){58}}\put(75,55){\vector(1,0){58}}
\put(50,45){\vector(0,-1){33}}\put (165,45){\vector(0,-1){33}}
\end{picture}\\

\noindent which is true since: it is the image by $F$ of a pullback diagram in $\mathbb{C}$ (according to Proposition \ref{proposition:limitsC-Cat}), in which the down arrow is a unit morphism in the reflection $F\dashv G$ and the right up object is in the subcategory $\mathbb{X}$; hence, as by hypothesis $F\dashv G$ is semi-left-exact, $FV_{a,b}$ is an isomorphism in $\mathbb{X}$ as required.\\

3. (simple) Consider the following diagram in $\mathbb{C}$-$Cat$

\begin{picture}(220,130)
\put(100,0){$\mathcal{B}$}\put(100,50){$\mathcal{L}$}
\put(200,0){$\mathbb{GF}(\mathcal{B})$\hspace{10pt,}}\put(200,50){$\mathbb{GF}(\mathcal{A})$}
\put(108,25){$U$} \put(218,25){$\mathbb{GF}T$} \put(150,8){$\Theta_{\mathcal{B}}$}
\put(150,58){$V$}
\put(115,3){\vector(1,0){80}}\put(115,53){\vector(1,0){80}}
\put(104,45){\vector(0,-1){33}}\put(215,45){\vector(0,-1){33}}
\put(45,105){$\mathcal{A}$}\put(60,50){$T$}\put(121,88){$\Theta_{\mathcal{A}}$}
\put(81,80){$W$}
\put(58,100){\vector(1,-1){40}}\put(55,99){\vector(1,-2){44}}
\put(60,103){\vector(3,-1){135}}
\end{picture}\\

\noindent where $\Theta_{\mathcal{A}}$ and $\Theta_{\mathcal{B}}$ are unit morphisms for the derived reflection $\mathbb{F}\dashv \mathbb{G}$, the square is a pullback in $\mathbb{C}$-$Cat$ and $W$ is the $\mathbb{C}$-functor determined by the pair $(T,\Theta_{\mathcal{A}})$. We have to show that $\mathbb{F}W:\mathbb{F}(\mathcal{A})\rightarrow\mathbb{F}(\mathcal{L})$ is an isomorphism in $\mathbb{X}$-$Cat$, provided the base monoidal reflection $F\dashv G$ is simple.

Remark that $ob\mathbb{F}V\circ ob\mathbb{F}W=ob\mathbb{F}\Theta_{\mathcal{A}}\Leftrightarrow obV\circ obW=1_{ob(\mathcal{A})}$ and so $obW$ is a bijection, since $obV$ is a bijection because $V$ is the pullback of $\Theta_{\mathcal{B}}$ along $\mathbb{GF}T$.

\begin{picture}(300,130)
\put(85,0){$\mathcal{B}(a_2,b_2)$}\put(90,50){$\mathcal{L}(a,b)$}
\put(200,0){$GF(\mathcal{B}(a_2,b_2))$\hspace{10pt,}}\put(200,50){$GF(\mathcal{A}(a_1,b_1))$}
\put(108,25){$U_{a,b}$}\put(218,25){$GFT_{a_1,b_1}$}\put(280,25){(9.3)}
\put(150,8){$\eta_{\mathcal{B}(a_2,b_2)}$}\put(150,58){$V_{a,b}$}
\put(140,3){\vector(1,0){55}}\put(130,53){\vector(1,0){65}}
\put(104,45){\vector(0,-1){33}}\put(215,45){\vector(0,-1){33}}
\put(45,105){$\mathcal{A}(a_1,b_1)$}\put(50,50){$T_{a_1,b_1}$}\put(121,88){$\eta_{\mathcal{A}(a_1,b_1)}$}
\put(81,80){$W_{a_1,b_1}$}
\put(58,100){\vector(1,-1){38}}\put(55,99){\vector(1,-2){43}}
\put(65,100){\vector(3,-1){130}}
\end{picture}\\

It remains to show, in the diagram (9.3) immediately above, that the image by $F$ of $W_{a_1,b_1}$ is an isomorphism in $\mathbb{X}$, for every $a,b\in ob(\mathcal{L})$; which is so since $F\dashv G$ is a simple reflection to start with.
\end{proof}


\section{The category of all $n$-categories}\label{sec:nCat}

Consider the category $nCat$, with objects all $n$-categories and whose morphisms are the (strict) $n$-functors (see \cite[\S XII.5]{SM:cat}). Its definition is going to be stated in a way that suits our purposes.\footnote{In this and the following sections, some of the contents of \cite{X:2ml} will be repeated (and generalized), so that the article is self-contained, making the reader's task easier.}\\

First, consider the category $\mathbb{P}$ generated by the following \emph{precategory diagram},\\
\begin{picture}(200,40)(0,0)
\put(60,0){$P_2$}
\put(80,20){\vector(1,0){40}}\put(80,5){\vector(1,0){40}}\put(80,-10){\vector(1,0){40}}
\put(95,-5){$r$}\put(95,8){$m$}\put(95,25){$q$}
\put(125,0){$P_1$}
\put(140,20){\vector(1,0){40}}\put(180,5){\vector(-1,0){40}}\put(140,-10){\vector(1,0){40}}
\put(155,-5){$c$}\put(155,8){$e$}\put(155,25){$d$}
\put(185,0){$P_0$\hspace{30pt} (10.1)}
\end{picture}\\

\noindent in which

$d\circ e=1_{P_0}=c\circ e$, $d\circ m =d\circ q$, $c\circ m =c\circ r$ and $c\circ q = d\circ r$,

\noindent where $1_{P_0}$ stands for the identity morphism of $P_0$ (see \cite[\S 4.1]{CJKP:stab}).

A \emph{precategory} is an object in the category of presheaves $\hat{\mathbb{P}}=Set^\mathbb{P}$, that is, any functor $P:\mathbb{P}\rightarrow Set$ to the category of sets.

If \begin{picture}(100,40)(0,0)
\put(60,0){$Q_2$}
\put(80,20){\vector(1,0){40}}\put(80,5){\vector(1,0){40}}\put(80,-10){\vector(1,0){40}}
\put(95,-5){$r'$}\put(95,8){$m'$}\put(95,25){$q'$}
\put(125,0){$Q_1$}
\put(140,20){\vector(1,0){40}}\put(180,5){\vector(-1,0){40}}\put(140,-10){\vector(1,0){40}}
\put(155,-5){$c'$}\put(155,8){$e'$}\put(155,25){$d'$}
\put(185,0){$Q_0$}
\end{picture}\\

\noindent is another precategory diagram, then a triple $(f_2,f_1,f_0)$ with $f_2:P_2\rightarrow Q_2$, $f_1:P_1\rightarrow Q_1$ and $f_0:P_0\rightarrow Q_0$, will be called a \emph{precategory morphism diagram} provided the following equations hold: $f_0\circ d=d'\circ f_1,\ f_0\circ c=c'\circ f_1,\ f_1\circ e=e'\circ f_0,\ f_1\circ q=q'\circ f_2, f_1\circ m=m'\circ f_2,\ f_1\circ r=r'\circ f_2.$

The category $Cat$ of all categories is the full subcategory of $\hat{\mathbb{P}}=Set^{\mathbb{P}}$, determined by its objects $C:\mathbb{P}\rightarrow Set$ such that the image by $C$ of the category generated by the precategory diagram $(10.1)$ is also a category, that is:

the commutative square

\begin{picture}(100,80)
\put (30,0){$C(P_1)$}\put (30,50){$C(P_2)$}\put (135,0){$C(P_0)$}\put (135,50){$C(P_1)$}
\put (30,25){$Cr$}\put (90,60){$Cq$}
\put(160,25){$Cc$}
\put(90,10){$Cd$}
\put(70,3){\vector(1,0){58}}\put(70,55){\vector(1,0){58}}
\put(50,45){\vector(0,-1){33}}\put (155,45){\vector(0,-1){33}}
\end{picture}\\

\noindent is a pullback diagram in $Set$;

the associative and unit laws hold for the operation $Cm$, that is, the following respective diagrams commute in $Set$,

\begin{picture}(100,80)
\put (30,0){$C(P_2)$}\put (0,50){$C(P_2)\times_{C(P_1)}C(P_2)$}\put (165,0){$C(P_1)$\hspace{10pt},}\put (165,50){$C(P_2)$}
\put (-5,25){$Cr\times Cm$}\put (100,60){$Cm\times Cq$}
\put(190,25){$Cm$}
\put(100,10){$Cm$}
\put(70,3){\vector(1,0){88}}\put(103,55){\vector(1,0){55}}
\put(50,45){\vector(0,-1){33}}\put (185,45){\vector(0,-1){33}}
\end{picture}

\begin{picture}(300,85)

\put (-10,50){$C(P_0)\times_{C(P_0)}C(P_1)$}\put(40,45){\vector(0,-1){30}}\put(92,55){\vector(1,0){23}}
\put(200,55){\vector(-1,0){45}}

\put(20,0){$C(P_1)$}\put (22,25){$pr_2$}
\put(55,3){\vector(1,0){60}}
\put(75,65){$Ce\times1_{C(P_1)}$}\put(60,10){$1_{C(P_1)}$}

\put (120,0){$C(P_1)$}\put (120,50){$C(P_2)$}\put (230,0){$C(P_1)$\hspace{10pt.}}\put (205,50){$C(P_1)\times_{C(P_0)}C(P_0)$}

\put (145,25){$Cm$}\put (155,65){$1_{C(P_1)}\times Ce$}\put(250,25){$pr_1$}
\put(180,10){$1_{C(P_1)}$}

\put(225,3){\vector(-1,0){70}}
\put(140,45){\vector(0,-1){30}}\put (245,45){\vector(0,-1){30}}
\end{picture}\\

Secondly, consider the categories $2\mathbb{P}_{kji}$ generated by the following \emph{2-precategory diagrams} (see \cite[\S 2]{X:2ml}), for integers $i$, $j$ and $k$ such that $0\leq i<j<k$,\footnote{These arbitrary integers will only become relevant in the definition of an $n$-category as a presheaf, which follows at the end of this section.}

\begin{picture}(200,105)

\put(55,70){$P_{kji}$}
\put(105,90){\vector(1,0){70}}\put(105,75){\vector(1,0){70}}\put(105,60){\vector(1,0){70}}
\put(120,64){$r^{ki}\times r^{ki}$}\put(120,78){$m^{ki}\times m^{ki}$}\put(120,95){$q^{ki}\times q^{ki}$}
\put(198,70){$P_{kj}$}
\put(235,90){\vector(1,0){70}}\put(305,75){\vector(-1,0){70}}\put(235,60){\vector(1,0){70}}
\put(250,64){$c^{ki}\times c^{ki}$}\put(250,78){$e^{ki}\times e^{ki}$}\put(250,95){$d^{ki}\times d^{ki}$}
\put(310,70){$P_i$}

\put (-12,35){$r^{kj}\times r^{kj}$}
\put (35,60){\vector(0,-1){40}}
\put (37,35){$m^{kj}\times m^{kj}$}
\put (65,60){\vector(0,-1){40}}
\put (92,35){$q^{kj}\times q^{kj}$}
\put (90,60){\vector(0,-1){40}}

\put (190,35){$r^{kj}$}
\put (185,60){\vector(0,-1){40}}
\put (211,35){$m^{kj}$}
\put (210,60){\vector(0,-1){40}}
\put (235,35){$q^{kj}$}
\put (230,60){\vector(0,-1){40}}

\put (320,35){$1_{P_i}$}\put (315,60){\vector(0,-1){40}}

\put(55,0){$P_{ki}$}
\put(105,20){\vector(1,0){70}}\put(105,5){\vector(1,0){70}}\put(105,-10){\vector(1,0){70}}
\put(128,-6){$r^{ki}$}\put(128,8){$m^{ki}$}\put(128,25){$q^{ki}$}
\put(198,0){$P_k$}
\put(235,20){\vector(1,0){70}}\put(305,5){\vector(-1,0){70}}\put(235,-10){\vector(1,0){70}}
\put(258,-6){$c^{ki}$}\put(258,8){$e^{ki}$}\put(258,25){$d^{ki}$}
\put(310,0){$P_i$}\put(330,0){$(10.2)_{kji}$}

\put (-12,-35){$c^{kj}\times c^{kj}$}
\put (35,-10){\vector(0,-1){40}}
\put (39,-35){$e^{kj}\times e^{kj}$}
\put (65,-50){\vector(0,1){40}}
\put (92,-35){$d^{kj}\times d^{kj}$}
\put (90,-10){\vector(0,-1){40}}

\put (190,-35){$c^{kj}$}
\put (185,-10){\vector(0,-1){40}}
\put (211,-35){$e^{kj}$}
\put (210,-50){\vector(0,1){40}}
\put (235,-35){$d^{kj}$}
\put (230,-10){\vector(0,-1){40}}

\put (320,-35){$1_{P_i}$}\put (315,-10){\vector(0,-1){40}}

\put(55,-70){$P_{ji}$}
\put(105,-50){\vector(1,0){70}}\put(105,-65){\vector(1,0){70}}\put(105,-80){\vector(1,0){70}}
\put(130,-76){$r^{ji}$}\put(130,-62){$m^{ji}$}\put(130,-45){$q^{ji}$}
\put(198,-70){$P_j$}
\put(235,-50){\vector(1,0){70}}\put(295,-65){\vector(-1,0){70}}\put(235,-80){\vector(1,0){70}}
\put(260,-76){$c^{ji}$}\put(260,-62){$e^{ji}$}\put(260,-45){$d^{ji}$}
\put(310,-70){$P_i$\hspace{10pt,}}

\end{picture}
\vspace{80pt}

in which:

\begin{itemize}
\item each one of the three horizontal diagrams is a precategory diagram, and will be called, respectively upwards, $P^{ji}$, $P^{ki}$ and $P^{kji}$;

\item each one of the three vertical diagrams is a precategory diagram, and will be called, respectively from the left to the right, $P^{kiji}$, $P^{kj}$ and (the trivial) $P_i$;

\item $(c^{kj}\times c^{kj},c^{kj},1_{P_i})$, $(e^{kj}\times e^{kj},e^{kj},1_{P_i})$, $(d^{kj}\times d^{kj},d^{kj},1_{P_i})$, $(r^{kj}\times r^{kj},r^{kj},1_{P_i})$, $(m^{kj}\times m^{kj},m^{kj},1_{P_i})$, $(q^{kj}\times q^{kj},q^{kj},1_{P_i})$ are all six precategory morphism diagrams (equivalently, $(q^{ki}\times q^{ki},q^{ki},q^{ji})$, $(m^{ki}\times m^{ki},m^{ki},m^{ji})$, $(r^{ki}\times r^{ki},r^{ki},r^{ji})$, $(d^{ki}\times d^{ki},d^{ki},d^{ji})$, $(e^{ki}\times e^{ki},e^{ki},e^{ji})$, $(c^{ki}\times c^{ki},c^{ki},c^{ji})$ are all six precategory morphism diagrams).

\end{itemize}

Notice that the names given to objects and morphisms in $(10.2)_{kji}$ are arbitrary, being so chosen in order to relate to the following definition of 2-category (for instance, $q^{kj}\times q^{kj}$ will denote the morphism uniquely determined by a pullback diagram).\\

The category $2Cat$ of all 2-categories is the full subcategory of $\hat{2\mathbb{P}_{kji}}=Set^{2\mathbb{P}_{kji}}$, determined by its objects $C:2\mathbb{P}_{kji}\rightarrow Set$ such that the image by $C$ of each horizontal and vertical precategory diagram in $(10.2)_{kji}$ is a category (for any integers $i$, $j$ and $k$ such that $0\leq i<j<k$).

It would be a long and trivial calculation to check that there is an isomorphism between the category of all 2-categories (in the sense of \cite[\S XII.3]{SM:cat}) and the full subcategory of $\hat{2\mathbb{P}_{kji}}$ just defined. Notice that: the requirement that the horizontal composite of two vertical identities is itself a vertical identity is encoded in diagram $(10.2)_{kji}$ in the commutativity of the square $m^{ki}\circ (e^{kj}\times e^{kj})=e^{kj}\circ m^{ji}$; the interchange law, which relates the vertical and the horizontal composites of $2$-cells, is encoded in diagram $(10.2)_{kji}$ in the commutativity of the square $m^{kj}\circ (m^{ki}\times m^{ki})=m^{ki}\circ (m^{kj}\times m^{kj})$.\\

Finally, for any integer $n\geq 2$, consider the category $n\mathbb{P}$ generated simultaneously by all the 2-precategory diagrams $(10.2)_{kji}$, for every integers $i$, $j$ and $k$ such that $0\leq i<j<k\leq n$.\footnote{The number $\frac{(n-1)n(n+1)}{6}$ of these 2-precategory diagrams can easily be obtained by using the counting principle. Remark that these $\frac{(n-1)n(n+1)}{6}$ diagrams are not independent since they share common objects and arrows, and that if $n=2$ there is only one 2-precategory diagram.}

The category $nCat$ of all $n$-categories is the full subcategory of $\hat{n\mathbb{P}}=Set^{n\mathbb{P}}$, determined by its objects $C:n\mathbb{P}\rightarrow Set$ such that the image by $C$ of each horizontal and vertical precategory diagram in $(10.1)_{kji}$ ($0\leq i<j<k\leq n$) is a category (as defined above).

Notice that, being $C:n\mathbb{P}\rightarrow Set$ a $n$-category:

\begin{itemize}
\item the image by $C$ of each 2-precategory diagram $(10.1)_{kji}$ is a 2-category, for each integer i, j and k such that $0\leq i<j<k\leq n$;

\item $C(P^{ki})$ and $C(P^{kj})$ ($j>i$) are categories corresponding respectively to the ``horizontal" structure, with $k$-cells as morphisms and $i$-cells as objects, and to the ``vertical" structure, with $k$-cells as morphisms and $j$-cells as objects;

\item the ``intersections" of the images of the 2-precategory diagrams make the 2-categories ``commute", so that $nCat$ is indeed the category of all $n$-categories (cf. \cite[\S XII.5]{SM:cat});

\item the morphisms in $C(P_{kj})$ and $C(P_{ki})$ of the categories $C(P^{kji})$ and $C(P^{kiji})$ will be called vertically and horizontally composable pairs of $k$-cells with $j$ and $i$-objects, respectively.

\end{itemize}

\section{Effective descent morphisms in $nCat$}
\label{sec:edm in nCat}

In section \ref{sec:nCat}, if the category $Set$ of sets is replaced by any category $\mathcal{C}$ with pullbacks, then we obtain the definition of $nCat(\mathcal{C})$, the category of internal $n$-categories in $\mathcal{C}$. Hence, $nCat=nCat(Set)$ for $\mathcal{C}=Set$. If $n=1$, $1Cat=1Cat(Set)=Cat$ and $1\mathbb{P}=\mathbb{P}$ is the category generated by the precategory diagram in $(10.1)$ (or, alternatively, by the precategory diagram $P^{10}$ in $(10.2)_{k10}$, for any $k>1$), corresponding to the composition of $1$-cells.

The second part of the ensuing Corollary \ref{corollary:limits nCat} will be needed extensively in this text. Its proof is well known.\footnote{It follows trivially from the results in \cite[\S 3]{X:2ml}, after an obvious simple generalization.}

\begin{corollary}\label{corollary:limits nCat}

If $\mathcal{C}$ has all limits then $nCat(\mathcal{C})$ is closed under limits in the functor category $\mathcal{C}^{n\mathbb{P}}$, $n\geq 1$, where all limits exist and are calculated pointwise.

In particular, for $\mathcal{C}=Set$, $nCat$ is closed under limits in $\hat{n\mathbb{P}}=Set^{n\mathbb{P}}$.
\end{corollary}

Consider again the category of all categories $Cat$ and its full inclusion in the category of precategories $\hat{\mathbb{P}}=Set^\mathbb{P}$. A functor $p:\mathbb{E}\rightarrow \mathbb{B}$ is an effective descent morphism (e.d.m.)\footnote{Also called a \emph{monadic extension} in categorical Galois theory.} in $Cat = 1Cat$ if and only if it is surjective on composable triples of morphisms. The proof of this statement can be found in \cite[Proposition 6.2]{JST:edm}. In a completely analogous way, a class of effective descent morphisms in $nCat$ is going to be given in the following Proposition \ref{proposition:EDM(nCat)} (cf.\ \cite[§4]{X:2ml}, where e.d.m.\ in $2Cat$ are obtained similarly for the special case $n=2$).

\begin{proposition}\label{proposition:EDM(nCat)}

A $n$-functor $np:n\mathbb{E}\rightarrow n\mathbb{B}$ is an e.d.m.\ in the category of all $n$-categories $nCat$ ($n\geq 2$) if it is surjective both on
\begin{itemize}
\item vertically composable triples of horizontally composable pairs of $n$-cells, and on

\item horizontally composable triples of vertically composable pairs of $n$-cells,

for all $i$ and $j$ such that $0\leq i<j<n$.\footnote{See the following Example \ref{example:EDM(nCat)}, and the last paragraph of section \ref{sec:nCat}.}
\end{itemize}

Meaning that, every composable triple of morphisms in the categories $n\mathbb{B}(P^{niji})$ and $n\mathbb{B}(P^{nji})$ are the image of some composable triple of morphisms in the categories $n\mathbb{E}(P^{niji})$ and $n\mathbb{E}(P^{nji})$, by (the restrictions of the $n$-functor) $np(P^{ni}):n\mathbb{E}(P^{ni})\rightarrow n\mathbb{B}(P^{ni})$ and $np(P^{nj}):n\mathbb{E}(P^{nj})\rightarrow n\mathbb{B}(P^{nj})$, respectively ($0\leq i<j<n$).
\end{proposition}

\begin{proof}

Let $np:n\mathbb{E}\rightarrow n\mathbb{B}$ be surjective on vertically/horizontally composable triples of horizontally/vertically composable pairs of $n$-cells, $0\leq i<j<n$. Then, $np$ is an e.d.m.\ in $\hat{n\mathbb{P}}=Set^{n\mathbb{P}}$, since the effective descent morphisms in a category of presheaves are simply those surjective pointwise (which is implied by either surjectivity on triples of composable $n$-cells, since $k$-cells are special ``degenerate" $n$-cells, $k\leq n$). Hence, the following instance of \cite[Corollary 3.9]{JST:edm} can be applied:

if $np:n\mathbb{E}\rightarrow n\mathbb{B}$ in $nCat$ is an e.d.m.\ in $\hat{n\mathbb{P}}=Set^{n\mathbb{P}}$ then $np$ is an e.d.m.\ in $nCat$ if and only if, for every pullback square

\begin{picture}(100,80)
\put (40,0){$n\mathbb{E}$}\put (40,50){$n\mathbb{D}$}\put (145,0){$n\mathbb{B}$}\put (145,50){$nA$}
\put(200,25){$(11.1)$}
\put(90,10){$np$}
\put(70,3){\vector(1,0){58}}\put(70,55){\vector(1,0){58}}
\put(50,45){\vector(0,-1){33}}\put (155,45){\vector(0,-1){33}}
\end{picture}\\

\noindent in $\hat{n\mathbb{P}}=Set^{n\mathbb{P}}$ such that $n\mathbb{D}$ is in $nCat$, then also $nA$ is in $nCat$.

Since the pullback square $(11.1)$ is calculated pointwise (cf.\ Corollary \ref{corollary:limits nCat}), it induces other pullback squares in $\hat{\mathbb{P}}=Set^{\mathbb{P}}$, corresponding to the three rows $P^{ji}$, $P^{ki}$ and $P^{kji}$, and the three columns $P^{kiji}$, $P^{kj}$ and $P_i$ in the 2-precategory diagrams $(10.2)_{kji}$, for all $0\leq i<j<k\leq n$.\footnote{The total number of these pullbacks is $(n-1)(\frac{n(n+1)}{3}+\frac{n}{2}+1)+1$, corresponding to the distinct precategories in $n\mathbb{P}$, as it is easy to count.}

The fact that $np$ is surjective on vertically/horizontally composable triples of horizontally/vertically composable pairs of $n$-cells, implies that its restrictions (to the rows and columns $n\mathbb{E}(P^{ji})$, $n\mathbb{E}(P^{ki})$, $n\mathbb{E}(P^{kji})$, $n\mathbb{E}(P^{kiji})$, $n\mathbb{E}(P^{kj})$ and $n\mathbb{E}(P_i)$, $0\leq i<j<k\leq n$) are surjective on triples of composable morphisms in $Cat$. Hence, these restrictions are effective descent morphisms in $Cat$. Therefore, $nA$ must always be a $n$-category, provided so is $n\mathbb{D}$.\end{proof}

\begin{example}\label{example:EDM(nCat)}

This example will provide an e.d.m.\ $p:n\mathbb{E}\rightarrow n\mathbb{B}$ in $nCat$ for each $n$-category $n\mathbb{B}$, constructing its domain $n\mathbb{E}$ as the coproduct\footnote{It is obvious that the coproduct of $n$-categories is just the disjoint union, as for categories.} of certain appropriate $n$-subcategories\footnote{Being a $n$-subcategory means that it is contained in $n\mathbb{B}$, and that it is closed under all compositions and identities in $n\mathbb{B}$.} of $n\mathbb{B}$. Such $n$-subcategories are not just $n$-categories, but more specifically $n$-preorders, as defined in the beginning of next section \ref{sec:stableunits&ml nCat}.\\

Consider all the vertically composable triples of horizontally composable pairs $v_{nji}$ and all the horizontally composable triples of vertically composable pairs $h_{nji}$ of $n$-cells in a $n$-category $n\mathbb{B}$, represented in the following diagrams, respectively from left to right, where the vertical arrows $\Downarrow$ are $n$-cells, the horizontal arrows $\rightarrow$ are $j$-cells and the nodes $\bullet$ are $i$-cells, $0\leq i<j<n$:

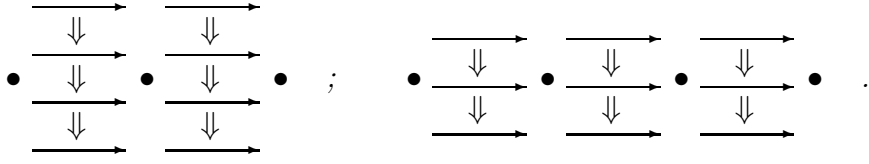
\begin{picture}(60,40)


\put (0,0){$\bullet$}\put (50,0){$\bullet$}\put (100,0){$\bullet
$}

\put(120,0){;}

\put(10,30){\vector(1,0){35}}\put(10,12){\vector(1,0){35}}
\put(10,-6){\vector(1,0){35}}\put(10,-24){\vector(1,0){35}}

\put(23,18){$\Downarrow$}\put(23,0){$\Downarrow$}
\put(23,-18){$\Downarrow$}

\put(60,30){\vector(1,0){35}}\put(60,12){\vector(1,0){35}}
\put(60,-6){\vector(1,0){35}}\put(60,-24){\vector(1,0){35}}

\put(73,18){$\Downarrow$}\put(73,0){$\Downarrow$}
\put(73,-18){$\Downarrow$}


\put (150,0){$\bullet$}\put (200,0){$\bullet$}
\put (250,0){$\bullet$}\put (300,0){$\bullet$}

\put(320,0){.}

\put(160,18){\vector(1,0){35}}\put(160,0){\vector(1,0){35}}\put(160,-18){\vector(1,0){35}}
\put(210,18){\vector(1,0){35}}\put(210,0){\vector(1,0){35}}\put(210,-18){\vector(1,0){35}}
\put(260,18){\vector(1,0){35}}\put(260,0){\vector(1,0){35}}\put(260,-18){\vector(1,0){35}}

\put(173,6){$\Downarrow$}\put(223,6){$\Downarrow$}\put(273,6){$\Downarrow$}
\put(173,-12){$\Downarrow$}\put(223,-12){$\Downarrow$}\put(273,-12){$\Downarrow$}

\end{picture}\\\\\\

Then, for each of the above two diagrams, build the respective $n$-categories $v\mathbf{4}_{nji}$ and $h\mathbf{4}_{nji}$ as follows:

the structure for the $k$-cells with $0\leq k<n$ is the same as in $n\mathbb{B}$;

the $n$-cells constitute the smallest preorder on $(n-1)$-cells such that

(1) a pair of $(n-1)$-cells is related only if it does share the same initial $(n-2)$-cell and the same terminal $(n-2)$-cell,\footnote{So that the relation underlying the preorder may by codified by arrows corresponding to $n$-cells: a pair $(f,v)$ of $(n-1)$-cells belongs to such a relation only if the domain of $f$ and $g$ is the same $(n-2)$-cell $a$, and the codomain of $f$ and $g$ is also the same $(n-2)$-cell $b$, pictured by \begin{picture}(55,20)(0,0)

\put (0,0){$a$}\put (50,0){$b$}

\put(10,12){\vector(1,0){35}}\put(25,15){$f$}
\put(10,-6){\vector(1,0){35}}\put(25,-15){$g$}

\put(23,0){$\Downarrow$}\put(32,0){$\theta$}

\end{picture} where $\theta$ is the $n$-cell relating $f$ to $g$.\\} and

(2) the diagram $v_{nji}$/$h_{nji}$ corresponding to the vertically/horizontally composable triples of horizontally/vertically composable pairs in question is contained in $v\mathbf{4}_{nji}$/$h\mathbf{4}_{nji}$, and

(3) with such preorder, $v\mathbf{4}_{nji}$/$h\mathbf{4}_{nji}$ is in fact a $n$-category;

notice that there is a trivial preorder satisfying (1), (2) and (3), the one which relates every ordered pair of $(n-1)$-cells with the same initial $(n-2)$-cell and the same terminal $(n-2)$-cell;\footnote{The existence and uniqueness of a $n$-cell between every ordered pair of $(n-1)$-cells, with common initial and terminal $(n-2)$ cells, guarantee that $v\mathbf{4}_{nji}$/$h\mathbf{4}_{nji}$ is closed under all compositions and identities.} hence, the smallest preorder can be obtained intersecting all such preorders satisfying (1), (2) and (3).\\

The $n$-category $$n\mathbb{E}=(\coprod_{0\leq i<j<n}\coprod_{v\in V_{ji}} v\mathbf{4}_{nji})+(\coprod_{0\leq i<j<n}\coprod_{h\in H_{ji}} h\mathbf{4}_{nji}),$$ such that $V_{ji}$/$H_{ji}$ is the set of all vertically/horizontally composable triples of horizontally/vertically composable pairs of $n$-cells in $n\mathbb{B}$, with $j$-cells as arrows and $i$-cells as objects in $n\mathbb{B}$.

Then, there is an e.d.m.\ $np:n\mathbb{E}\rightarrow n\mathbb{B}$ which projects the corresponding copy of $v\mathbf{4}_{ji}$/$h\mathbf{4}_{ji}$, for every $v\in V_{ji}$/$h\in H_{ji}$, $0\leq i<j<n$.\footnote{Remark that $v\mathbf{4}_{ji}$ and $h\mathbf{4}_{ji}$ are really $n$-preorders, as defined just below at the beginning of the following section \ref{sec:stableunits&ml nCat}.}

\end{example}

\section{The reflection of $n$-categories into $n$-preorders has stable units and a monotone-light factorization}
\label{sec:stableunits&ml nCat}

Let $nPreord$ be the full subcategory of $nCat$ determined by the objects $C:n\mathbb{P}\rightarrow Set$ such that $Cd^{n(n-1)}$ and $Cc^{n(n-1)}$ are jointly monic (cf.\ diagram $(10.2)_{n(n-1)(n-2)}$), that is,

\begin{picture}(200,40)(0,0)
\put(0,0){$C(P^{n(n-1)})=$}
\put(70,0){$C(P_{n(n-1)})$}
\put(130,20){\vector(1,0){40}}\put(130,5){\vector(1,0){40}}\put(130,-10){\vector(1,0){40}}
\put(130,-7){$Cr^{n(n-1)}$}\put(130,8){$Cm^{n(n-1)}$}\put(130,25){$Cq^{n(n-1)}$}
\put(175,0){$C(P_n)$}
\put(215,20){\vector(1,0){40}}\put(255,5){\vector(-1,0){40}}\put(215,-10){\vector(1,0){40}}
\put(215,-7){$Cc^{n(n-1)}$}\put(215,8){$Ce^{n(n-1)}$}\put(215,25){$Cd^{n(n-1)}$}
\put(260,0){$C(P_{n-1})$}
\end{picture}\\

\noindent is a preordered set.\\

There is a reflection

\begin{picture}(300,20)

\put(30,0){$H\vdash I: nCat$}
\put(100,3){\vector(1,0){20}}
\put(125,0){$nPreord$,}
\put(-10,0){$(12.1)$}

\put (180,0){$a$}\put (230,0){$b$}
\put (260,0){$a$}\put (310,0){$b$,}

\put(190,12){\vector(1,0){35}}\put(205,15){$f$}
\put(190,-6){\vector(1,0){35}}\put(205,-15){$g$}

\put(243,0){$\mapsto$}

\put(270,12){\vector(1,0){35}}\put(285,15){$f$}
\put(270,-6){\vector(1,0){35}}\put(285,-15){$g$}

\put(203,0){$\Downarrow$}\put(212,0){$\theta$}
\put(283,0){$\Downarrow$}\put(292,0){$\leq$}

\end{picture}\vspace{15pt}

\noindent which identifies all $n$-cells which have the same domain and codomain with respect to $P^{n(n-1)}$.\footnote{In $(12.1)$, $f$ and $g$ stand for generic $(n-1)$-cells and $a$ and $b$ for $(n-2)$-cells.} That is, the reflector $I$ takes the middle vertical category $C(P^{n(n-1)})$ (see diagram $(10.2)_{n(n-1)(n-2)}$) to its image by the well known reflection $Cat\rightarrow Preord$ from categories into preordered sets (see \cite{X:ml}).\\

\subsection{Stable units}\label{subsection:stable units}

The reflection $Cat\rightarrow Preord$ from categories into preorders is known to be a \emph{base monoidal reflection} in the sense of Definition \ref{def:base monoidal reflection}, considering $Cat$ equipped with its cartesian symmetric monoidal structure. It has also stable units in the sense of \cite{CHK:fact} (cf.\ section \ref{sec:Simp., semi-left-exact. and stable units}).

Therefore we can iterate it n times, obtaining exactly the reflection $(12.1)$, also with stable units according to Proposition \ref{proposition:stableunits}:

\begin{theorem}\label{theorem:stable units}

The reflection $H\vdash I: nCat\rightarrow nPreord$ has stable units, for every $n\geq 1$.

\end{theorem}

\subsection{Monotone-light factorization for $n$-categories via $n$-preorders}\label{subsection:existence m-l fact}

\begin{theorem}\label{theorem:monotone-light fact}

The reflection $H\vdash I: nCat\rightarrow nPreord$ does have a monotone-light factorization, for every $n\geq 1$.

\end{theorem}

\begin{proof}

The statement is a consequence of the central result of \cite{CJKP:stab} (cf.\ Corollary 6.2 in \cite{X:iml}), because $H\vdash I$ has stable units (cf.\ Theorem \ref{theorem:stable units}) and for every $n\mathbb{B}\in nCat$ there is an e.d.m.\ $np:n\mathbb{E}\rightarrow n\mathbb{B}$ with $n\mathbb{E}\in nPreord$ (cf.\ Example \ref{example:EDM(nCat)}).
\end{proof}

In the following section \ref{sec:Vertical and stably-vertical}, it will be proved that the monotone-light factorization system is not trivial, for every $n\geq 1$. That is, it does not coincide with the reflective factorization system associated to the reflection of $nCat$ into $nPreord$.

\section{Vertical and stably-vertical $n$-functors}
\label{sec:Vertical and stably-vertical}

In this section, a characterization will be given of the class of vertical morphisms $\mathcal{E}_I$ in the reflective factorization system $(\mathcal{E}_I,\mathcal{M}_I)$, and of the class of the stably-vertical morphisms $\mathcal{E}'_I$ ($\subseteq\mathcal{E}_I$)\footnote{$\mathcal{E}'_I$ is the largest subclass of $\mathcal{E}_I$ stable under pullbacks. The terminologies ``vertical morphisms" and ``stably-vertical morphisms" were introduced in \cite{for:ins:sep:factorization}.} in the monotone-light factorization system $(\mathcal{E}'_I,\mathcal{M}^*_I)$, both associated to the reflection $I: nCat\rightarrow nPreord$. Then, since $\mathcal{E}'_I$ is a proper class of $\mathcal{E}_I$, we conclude that $(\mathcal{E}'_I,\mathcal{M}^*_I)$ is a non-trivial monotone-light factorization system.\\

Consider a $n$-functor $f:A\rightarrow B$, which is determined by the $n+1$ functions $f_0:A(P_0)\rightarrow B(P_0)$, $f_1:A(P_1)\rightarrow B(P_1)$,..., $f_n:A(P_n)\rightarrow B(P_n)$ (cf.\ diagrams $(10.2)_{kji}$, $0\leq i<j<k\leq n$), so that we may make the identification $f=(f_n,f_{n-1},...,f_0)$.

\begin{proposition}\label{proposition:vertical morphisms}

A $n$-functor $f=(f_n,f_{n-1},...,f_0):A\rightarrow B$ belongs to the class $\mathcal{E}_I$ of vertical $n$-functors if and only if the following two conditions hold:
\begin{enumerate}
\item $f_0$, $f_1$,..., $f_{n-1}$ are bijections;
\item for every two elements $h$ and $h'$ in $A(P_{n-1})$,

\noindent if $Hom_{B(P^{n(n-1)})}(f_{n-1}h,f_{n-1}h')$ is nonempty

\noindent then so is $Hom_{A(P^{n(n-1)})}(h,h')$.
\end{enumerate}
\end{proposition}

\begin{proof}

The $n$-functor $f=(f_n,f_{n-1},...,f_0)$ belongs to $\mathcal{E}_I$ if and only if $If$ is an isomorphism (cf.\ \cite[\S 3.1]{CJKP:stab}), that is, $If_0$, $If_1$,..., $If_n$ are bijections. Since $If_0=f_0$,..., $If_{n-1}=f_{n-1}$, the fact that $f\in \mathcal{E}_I$ implies and is implied by (1) and (2) is trivial.
\end{proof}

\begin{proposition}\label{proposition:stably-vertical}

A $n$-functor $f=(f_n,f_{n-1},...,f_0):A\rightarrow B$ belongs to the class $\mathcal{E}'_I$ of stably-vertical $n$-functors if and only if the following two conditions hold:

\begin{enumerate}

\item $f_0$, $f_1$,..., $f_{n-1}$ are bijections;
\item for every two elements $h$ and $h'$ in $A(P_{n-1})$,

\noindent $f$ induces a surjection

\noindent $Hom_{A(P^{n(n-1)})}(h,h')\rightarrow Hom_{B(P^{n(n-1)})}(f_{n-1}h,f_{n-1}h')$

\noindent ($f$ is a ``full functor on $n$-cells").

\end{enumerate}

\end{proposition}

\begin{proof}

As every pullback $g^*(f)=\pi_1:C\times_BA\rightarrow C$ in $nCat$ of $f$ along any $n$-functor $g:C\rightarrow B$ is calculated pointwise, and being $(f_n,f_{n-1}):A(P^{n(n-1)})\rightarrow B(P^{n(n-1)})$ a stably-vertical functor for the reflection $Cat\rightarrow Preord$, that is, $f_{n-1}$ a bijection and $(f_n,f_{n-1})$ a full functor (cf.\ Propositions 4.4 and 3.2 in \cite{X:ml}), then (1) and (2) imply that $g^*(f)$ belongs to $\mathcal{E}_I$ (cf.\ last Proposition \ref{proposition:vertical morphisms}).

Hence, $f\in \mathcal{E}'_I$ if (1) and (2) hold.\\

If $f\in \mathcal{E}'_I$, then $f\in \mathcal{E}_I$ ($\mathcal{E}'_I\subseteq \mathcal{E}_I$), and therefore (1) holds.

Suppose now that (2) does not hold, so that there is $\theta :f_{n-1}h\rightarrow f_{n-1}h'$ not in the image of $f$, and consider the $n$-category $C$ generated\footnote{Generated in the sense of Example \ref{example:EDM(nCat)}, that is, $C$ has the same structure of $B$, with the exception of $n$-cells, whose structure is the least which includes the following diagram.} by the diagram \begin{picture}(110,30)(0,0)

\put (0,0){$f_{n-2}(a)$}\put (70,0){$f_{n-2}(a')$,}

\put(35,12){\vector(1,0){35}}\put(40,16){$f_1h$}
\put(35,-6){\vector(1,0){35}}\put(40,-16){$f_1h'$}

\put(48,0){$\Downarrow$}\put(57,0){$\theta$}

\end{picture}\ \ and let $g$ be the inclusion of $C$ in $B$. Then, $C\times_BA$ contains \begin{picture}(110,30)(0,0)

\put (0,0){$f_{n-2}(a)$}\put (70,0){$f_{n-2}(a')$,}

\put(35,12){\vector(1,0){35}}\put(40,16){$f_1h$}
\put(35,-6){\vector(1,0){35}}\put(40,-16){$f_1h'$}

\end{picture}\vspace{20pt}

\noindent with no non-identity $2$-cells, and so $g^*(f)$ is not in $\mathcal{E}_I$.\\

Hence, if $f\in \mathcal{E}'_I$ then (1) and (2) must hold.\end{proof}

It is evident that $\mathcal{E}'_I$ is a proper class of $\mathcal{E}_I$, therefore the monotone-light factorization system $(\mathcal{E}'_I,\mathcal{M}^*_I)$ is non-trivial ($\neq (\mathcal{E}_I,\mathcal{M}_I)$).

\section{Trivial coverings for $n$-categories via $n$-preorders}
\label{sec:Trivial coverings}

A $n$-functor $f:A\rightarrow B$ belongs to the class $\mathcal{M}_I$ of trivial coverings (with respect to the reflection $H\vdash I:nCat\rightarrow nPreord$) if and only if the following square

\begin{picture}(100,80)
\put (45,0){$B$}\put (45,50){$A$}\put (140,0){$I(B)$}\put (140,50){$I(A)$}
\put (35,25){$f$}\put (90,60){$\eta_A$}
\put(160,25){$If$}\put(200,25){$(14.1)$}
\put(90,10){$\eta_B$}
\put(70,3){\vector(1,0){58}}\put(70,55){\vector(1,0){58}}
\put(50,45){\vector(0,-1){33}}\put (155,45){\vector(0,-1){33}}
\end{picture}\\

\noindent is a pullback diagram, where $\eta_A$ and $\eta_B$ are unit morphisms for the reflection $H\vdash I:nCat\rightarrow nPreord$ (cf.\ \cite[Theorem 4.1]{CHK:fact}).\\

Since the pullback (as any limit) is calculated pointwise in $nCat$ (cf.\ Corollary \ref{corollary:limits nCat}), then $f\in \mathcal{M}_I$ if and only if the following squares are pullbacks, corresponding to the pointwise components of $\eta_A$ and of $\eta_B$ (cf.\ diagrams $(10.2)_{kji}$), $0\leq l\leq n$ and $0\leq i<j<k\leq n$:

\begin{picture}(100,80)
\put (37,0){$B(P_l)$}\put (37,50){$A(P_l)$}\put (135,0){$I(B)(P_l)$}\put (135,50){$I(A)(P_l)$}
\put (33,25){$f_{P_l}$}\put (90,60){$\eta_{A(P_l)}$}
\put(160,25){$If_{P_l}$;}\put(87,25){($D_l$)}
\put(90,10){$\eta_{B(P_l)}$}
\put(70,3){\vector(1,0){58}}\put(70,55){\vector(1,0){58}}
\put(50,45){\vector(0,-1){33}}\put (155,45){\vector(0,-1){33}}
\end{picture}

\begin{picture}(400,80)
\put (-4,0){$B(P_{ji})$}\put (-4,50){$A(P_{ji})$}\put (100,0){$I(B)(P_{ji})$}\put (100,50){$I(A)(P_{ji})$}
\put (-7,25){$f_{P_{ji}}$}\put (55,60){$\eta_{A(P_{ji})}$}
\put(125,25){$If_{P_{ji}}$;}\put(52,25){($D_{ji}$)}
\put(55,10){$\eta_{B(P_{ji})}$}
\put(35,3){\vector(1,0){58}}\put(35,55){\vector(1,0){58}}
\put(15,45){\vector(0,-1){33}}\put (120,45){\vector(0,-1){33}}

\put (196,0){$B(P_{kj})$}\put (196,50){$A(P_{kj})$}\put (300,0){$I(B)(P_{kj})$}\put (300,50){$I(A)(P_{kj})$}
\put (193,25){$f_{P_{kj}}$}\put (255,60){$\eta_{A(P_{kj})}$}
\put(325,25){$If_{P_{kj}}$;}\put(252,25){($D_{kj}$)}
\put(255,10){$\eta_{B(P_{kj})}$}
\put(235,3){\vector(1,0){58}}\put(235,55){\vector(1,0){58}}
\put(215,45){\vector(0,-1){33}}\put (320,45){\vector(0,-1){33}}

\end{picture}

\begin{picture}(400,80)
\put (-4,0){$B(P_{ki})$}\put (-4,50){$A(P_{ki})$}\put (100,0){$I(B)(P_{ki})$}\put (100,50){$I(A)(P_{ki})$}
\put (-7,25){$f_{P_{ki}}$}\put (55,60){$\eta_{A(P_{ki})}$}
\put(125,25){$If_{P_{ki}}$;}\put(52,25){($D_{ki}$)}
\put(55,10){$\eta_{B(P_{ki})}$}
\put(35,3){\vector(1,0){58}}\put(35,55){\vector(1,0){58}}
\put(15,45){\vector(0,-1){33}}\put (120,45){\vector(0,-1){33}}

\put (190,0){$B(P_{kji})$}\put (190,50){$A(P_{kji})$}\put (300,0){$I(B)(P_{kji})$.}\put (300,50){$I(A)(P_{kji})$}
\put (189,25){$f_{P_{kji}}$}\put (255,60){$\eta_{A(P_{kji})}$}
\put(325,25){$If_{P_{kji}}$}\put(252,25){($D_{kji}$)}
\put(255,10){$\eta_{B(P_{kji})}$}
\put(235,3){\vector(1,0){58}}\put(235,55){\vector(1,0){58}}
\put(215,45){\vector(0,-1){33}}\put (320,45){\vector(0,-1){33}}

\end{picture}\\

The squares $(D_1),(D_2),...,(D_{n-1})$ are pullbacks since $\eta_{A(P_l)}$ and $\eta_{B(P_l)}$ are identity maps for $l<n$ (cf.\ diagrams $(10.2)_{kji}$ and the definition of the reflection $H\vdash I:nCat\rightarrow nPreord$ in $(12.1)$).\\

Notice that, restricting to the precategory diagram $P^{n(n-1)}$, we obtain from $(14.1)$ the following square in $Cat$, with unit morphisms of the reflection of all categories into preorders $Cat\rightarrow Preord$ (cf.\ \cite{X:ml}),

\begin{picture}(100,80)
\put (10,0){$B(P^{n(n-1)})$}\put (10,50){$A(P^{n(n-1)})$}\put (140,0){$I(B)(P^{n(n-1)})$.}\put (140,50){$I(A)(P^{n(n-1)})$}
\put (5,25){$f_{P^{n(n-1)}}$}\put (70,62){$\eta_{A(P^{n(n-1)})}$}
\put(160,25){$If_{P^{n(n-1)}}$}
\put(70,10){$\eta_{B(P^{n(n-1)})}$}
\put(70,3){\vector(1,0){58}}\put(70,55){\vector(1,0){58}}
\put(50,45){\vector(0,-1){33}}\put (155,45){\vector(0,-1){33}}
\end{picture}\\

It is known (cf.\ \cite[Proposition 3.1]{X:ml}) that this square is a pullback in $Cat$ if and only if, for every two objects $h$ and $h'$ in $A(P_{n-1})$ with $Hom_{A(P^{n(n-1)})}(h,h')$ nonempty, the map
$$Hom_{A(P^{n(n-1)})}(h,h')\rightarrow Hom_{B(P^{n(n-1)})}(f_{n-1}h,f_{n-1}h')$$
induced by $f$ is a bijection.

A necessary condition for the $n$-functor $f$ to be a trivial covering was just stated; the following Lemma \ref{lemma:remaining pullbacks} will help to show that this necessary condition is also sufficient in next Proposition \ref{proposition:trivial coverings}.

\begin{lemma}\label{lemma:remaining pullbacks}
Consider the following commutative parallelepiped

\begin{picture}(100,130)(50,0)\setlength{\unitlength}{0.5mm}

\put(60,0){$B_2$}\put(70,-5){\vector(1,-1){27}}\put(66,-20){$\eta_{B,2}$}
\put(75,5){\vector(1,0){40}}\put(75,0){\vector(1,0){40}}
\put(90,-10){$r^B$}\put(90,10){$q^B$}
\put(120,0){$B_1$}\put(128,-5){\vector(2,-1){48}}\put(162,-20){$\eta_{B,1}$}
\put(135,5){\vector(1,0){40}}\put(135,0){\vector(1,0){40}}
\put(157,-10){$d^B$}\put(157,10){$c^B$}
\put(180,0){$B_0$}\put(190,-2){\vector(2,-1){53}}\put(215,-10){$\eta_{B,0}$}

\put (55,55){$f_2$} \put (128,55){$f_1$}\put
(188,55){$f_0$}\put (65,65){\vector(0,-1){53}}\put
(183,65){\vector(0,-1){53}}\put (123,65){\vector(0,-1){53}}

\put(60,70){$A_2$}\put(70,65){\vector(1,-1){25}}\put(72,45){$\eta_{A,2}$}
\put(75,75){\vector(1,0){40}}\put(75,70){\vector(1,0){40}}
\put(90,60){$r^A$}\put(90,80){$q^A$}
\put(120,70){$A_1$}\put(135,65){\vector(1,-1){25}}\put(162,45){$\eta_{A,1}$}
\put(135,75){\vector(1,0){40}}\put(135,70){\vector(1,0){40}}
\put(157,60){$d^A$}\put(157,80){$c^A$}
\put(180,70){$A_0$}\put(190,65){\vector(2,-1){50}}\put(215,60){$\eta_{A,0}$}

\put(92,-40){$I(B)_2$}
\put(115,-35){\vector(1,0){50}}\put(115,-40){\vector(1,0){50}}
\put(130,-50){$Ir^B$}\put(130,-30){$Iq^B$}
\put(168,-40){$I(B)_1$}
\put(190,-35){\vector(1,0){50}}\put(190,-40){\vector(1,0){50}}
\put(205,-50){$Id^B$}\put(205,-30){$Ic^B$}
\put(245,-40){$I(B)_0$\hspace{10pt,}}

\put(106,-15){$If_2$} \put (180,-15){$If_1$}\put
(250,-15){$If_0$}\put (105,25){\vector(0,-1){53}}\put
(248,25){\vector(0,-1){53}}\put (177,25){\vector(0,-1){53}}

\put(93,30){$I(A)_2$}
\put(115,30){\vector(1,0){50}}\put(115,35){\vector(1,0){50}}
\put(130,21){$Ir^A$}\put(130,40){$Iq^A$}
\put(168,30){$I(A)_1$}
\put(190,35){\vector(1,0){50}}\put(190,30){\vector(1,0){50}}
\put(205,21){$Id^A$}\put(205,40){$Ic^A$}
\put(245,30){$I(A)_0$}\put(260,10){$(14.2)$}
\end{picture}
\vspace{80pt}

\noindent where the five squares $c^Aq^A=d^Ar^A$, $c^Bq^B=d^Br^B$, $Ic^AIq^A=Id^AIr^A$, $If_0\eta_{A,0}=\eta_{B,0}f_0$ and $If_1\eta_{A,1}=\eta_{B,1}f_1$ are pullbacks.

Then, the square $If_2\eta_{A,2}=\eta_{B,2}f_2$ is also a pullback.\footnote{The notation used in diagram $(14.2)$ is arbitrary and was chosen to make the application of Lemma \ref{lemma:remaining pullbacks} to this section easily understandable.}

\end{lemma}

\begin{proof} The proof is obtained by an obvious diagram chase.
\end{proof}

\begin{proposition}\label{proposition:trivial coverings}
A $n$-functor $f:A\rightarrow B$ is a trivial covering for the reflection $H\vdash I:nCat\rightarrow nPreord$ (in notation, $f\in \mathcal{M}_I$) if and only if, for every two objects $h$ and $h'$ in $A(P_{n-1})$ with $Hom_{A(P^{n(n-1)})}(h,h')$ nonempty, the map
$$Hom_{A(P^{n(n-1)})}(h,h')\rightarrow Hom_{B(P^{n(n-1)})}(f_1h,f_1h')$$
induced by $f$ is a bijection.
\end{proposition}

\begin{proof}
In the considerations just above, it was showed that the statement warrants that the square $(D_n)$ is a pullback, adding to the fact that $(D_0),(D_1),...,(D_{n-1})$ are all pullbacks.

Then, all the others squares $(D_{ji}),(D_{kj}),(D_{ki})$ and $(D_{kji})$ must also be pullbacks according to Lemma \ref{lemma:remaining pullbacks}, $0\leq i<j<k\leq n$.
\end{proof}

\section{Coverings for $n$-categories via $n$-preorders}
\label{sec:Coverings}

A $n$-functor $f:A\rightarrow B$ belongs to the class $\mathcal{M}^*_I$ of coverings (with respect to the reflection $H\vdash I:nCat\rightarrow n
Preord$) if there is some effective descent morphism (also called monadic extension in categorical Galois theory) $p:C\rightarrow B$ in $nCat$ with codomain $B$ such that the pullback $p^*(f):C\times_BA\rightarrow C$ of $f$ along $p$ is a trivial covering ($p^*(f)\in\mathcal{M}_I$).\\

The following Lemma \ref{lemma:sufficient condition for being covering} can be found in \cite[Lemma 4.2]{X:ml}, in the context of the reflection of categories into preorders, and for $n$-categories via $n$-preorders the proof is exactly the same, since the same conditions hold (cf.\ Theorem \ref{theorem:stable units} and Example \ref{example:EDM(nCat)}). The next Proposition \ref{proposition:coverings} characterizes the coverings for $n$-categories via $n$-preorders.

\begin{lemma}\label{lemma:sufficient condition for being covering}

A $n$-functor $f:A\rightarrow B$ in $nCat$ is a covering (for the reflection $H\vdash I:nCat\rightarrow nPreord$) if and only if, for every $n$-functor $\varphi :X\rightarrow B$ over $B$ from any $n$-preorder $X$, the pullback $X\times_BA$ of $f$ along $\varphi$ is also a $n$-preorder.

\end{lemma}

\begin{proposition}\label{proposition:coverings}

A $n$-functor $f:A\rightarrow B$ in $nCat$ is a covering (for the reflection $H\vdash I:nCat\rightarrow nPreord$) if and only if it is faithful with respect to $n$-cells, that is, for every pair $g,g'\in A(P_{n-1})$, the map $$Hom_{A(P^{n(n-1)})}(g,g')\rightarrow Hom_{B(P^{n(n-1)})}(f_{n-1}g,f_{n-1}g')$$ induced by $f$ is an injection.
\end{proposition}

\begin{proof}
Consider again the $n$-preorder $T$ generated\footnote{Cf.\ the footnote in the proof of Proposition \ref{proposition:stably-vertical}.} by the diagram \begin{picture}(60,25)(0,0)

\put (0,0){$a$}\put (50,0){$a'$.}

\put(10,12){\vector(1,0){35}}\put(25,16){$h$}
\put(10,-6){\vector(1,0){35}}\put(25,-16){$h'$}

\put(23,0){$\Downarrow$}\put(32,0){$\leq$}

\end{picture}\vspace{20pt}

If $f$ is not faithful with respect to $n$-cells (in the sense of the statement), then, by including $T$ in $B$, we could obtain a pullback $T\times_BA$ that is not a $n$-preorder; and so, $f$ would not be a covering, by the previous Lemma \ref{lemma:sufficient condition for being covering}.\\

Reciprocally, consider any $n$-functor $\varphi :X\rightarrow B$ such that $X$ is a $n$-preorder. If $f$ is faithful with respect to $n$-cells (in the sense of the statement), then the pullback $X\times_BA$ is a $n$-preorder, given the nature of $X$. Hence, $f$ is a covering, by the previous Lemma \ref{lemma:sufficient condition for being covering}.
\end{proof}

\section*{Acknowledgement}
This work was supported by
 The Center for Research and Development in Mathematics and Applications (CIDMA) through the Portuguese Foundation for Science and Technology

(FCT - Funda\c{c}\~ao para a Ci\^encia e a Tecnologia),

references UIDB/04106/2020 and UIDP/04106/2020.

\end{document}